\documentclass{article}

\usepackage{arxiv}
\usepackage{amsmath}
\usepackage{comment}
\usepackage{amsfonts}
\usepackage{amssymb}
\usepackage{amsthm}
\usepackage[flushleft]{threeparttable}
\usepackage{anyfontsize}
\usepackage[square,numbers]{natbib}

\usepackage{titletoc}
\usepackage[toc, page, header]{appendix} 

\usepackage[algo2e,ruled,noend]{algorithm2e}

\SetCommentSty{mycommfont}
\usepackage{nicefrac}
\usepackage{bbm}

\usepackage{enumitem}
\setitemize{leftmargin=*}
\setenumerate{leftmargin=*}

\usepackage{footnote}
\usepackage{multirow}
\usepackage{colortbl}

\usepackage{wrapfig}

\allowdisplaybreaks

\newcommand{\longbo}[1]{{\color{red}  [\text{Longbo:} #1]}}
\newcommand{\jiatai}[1]{{\color{cyan}  [\text{Jiatai:} #1]}}
\newcommand{\nop}[1]{}
\newcommand{\leana}[1]{{\color{blue}  [\text{Leana:} #1]}}

\renewcommand{\longbo}[1]{}
\renewcommand{\jiatai}[1]{}
\renewcommand{\leana}[1]{}

\renewcommand{\tilde}{\widetilde}

\renewcommand{\bar}{\overline}
\renewcommand{\O}{\operatorname{\mathcal O}}
\newcommand{\Otil}{\operatorname{\tilde{\mathcal O}}}

  \newtheorem{theorem}{Theorem}[section]
  
  \newtheorem{proposition}[theorem]{Proposition}
  \newtheorem{lemma}[theorem]{Lemma}
  \newtheorem{corollary}[theorem]{Corollary}

\newtheorem{assumption}{Assumption}

\usepackage{cleveref}
\Crefformat{equation}{Eq. #2(#1)#3}
\Crefrangeformat{equation}{Eqs. #3(#1)#4 to #5(#2)#6}
\Crefmultiformat{equation}{Eqs. #2(#1)#3}{ and #2(#1)#3}{, #2(#1)#3}{ and #2(#1)#3}
\Crefrangemultiformat{equation}{Eqs. #3(#1)#4 to #5(#2)#6}{ and #3(#1)#4 to #5(#2)#6}{, #3(#1)#4 to #5(#2)#6}{ and #3(#1)#4 to #5(#2)#6}

\DeclareMathOperator*{\argmax}{arg\,max}
\DeclareMathOperator*{\argmin}{arg\,min}

\SetKwInput{kwIntermediate}{Intermediate Variables}

\usepackage{float}
\usepackage{graphicx}
\usepackage{caption}
\usepackage{subcaption}

\title{Queue Scheduling with Adversarial Bandit Learning}

\author{
        Jiatai Huang\\
	Tsinghua University\\
	\texttt{hjt18@mails.tsinghua.edu.cn} \\
 	\And
        Leana Golubchik\\
	University of Southern California\\
	\texttt{leana@usc.edu} \\
	\And
        Longbo Huang\\
	Tsinghua University\\
	\texttt{longbohuang@tsinghua.edu.cn} \\
}

\renewcommand{\shorttitle}{Queue Scheduling with Adversarial Bandit Learning}

\begin{document}
\maketitle

\begin{abstract}
In this paper, we study scheduling of a queueing system with zero knowledge of instantaneous network conditions. We consider a one-hop single-server queueing system consisting of $K$ queues, each with time-varying and non-stationary arrival and service rates. Our scheduling approach builds on an innovative combination of adversarial bandit learning and Lyapunov drift minimization, without knowledge of the instantaneous network state (the arrival and service rates) of each queue. We then present two novel algorithms \texttt{SoftMW} (SoftMaxWeight) and \texttt{SSMW} (Sliding-window SoftMaxWeight), both capable of stabilizing systems that can be stablized by some (possibly unknown) sequence of randomized policies whose time-variation satisfies a mild condition. 
  We further generalize our results to the setting where  arrivals and departures only have bounded moments instead of being deterministically bounded and propose \texttt{SoftMW+}   and \texttt{SSMW+}   that are capable of stabilizing the system. 
  As a building block of our new algorithms, we also extend the classical \texttt{EXP3.S} \cite{auer2002nonstochastic} algorithm for multi-armed bandits to handle unboundedly large feedback signals, which can be of independent interest.
\end{abstract}

\keywords{Scheduling, Queueing \and Bandit Learning \and Lyapunov Analysis}

\section{Introduction}
Stochastic network scheduling is concerned with a fundamental problem of allocating resources to serving demand in dynamic environments, and it has found wide applicability in modeling real-world networked systems, including data communication \cite{kong2019optimal,tsanikidis2021power}, cloud computing and server farms \cite{maguluri2012stochastic,el2017stochastic,psychas2021theory,berg2020optimal}, smart grid management 
\cite{hu2020modeling,kim2020parallel,lv2021contract}, supply chain management 
\cite{rahdar2018tri,ben2019optimization}, and control of transportation networks \cite{wei2019presslight,braverman2019empty,braverman2017fluid}. 
%
One basic requirement of most existing scheduling solutions is having knowledge of the instantaneous network state -- i.e., the amount of arrival traffic and the amount of service under any feasible control action, e.g., the power allocation among all links -- before taking a new scheduling action. 
Given this information, 
there have been many successful network scheduling algorithms, with various aspects of theoretical performance guarantees, including queue stability \cite{tsibonis2003exploiting,sadiq2009throughput,liu2011throughput}, delays \cite{neely2008order,neely2009delay,huang2012lifo}, and utilities \cite{huang2011utility,huang2012lifo,neely2012delay}.

However, in many real-world scenarios, such network-state knowledge may not always be available if its measurement or estimation is too difficult or costly to obtain. Even when such knowledge is available, it can be biased and imperfect. For instance, in an IoT system, due to sensors' temperature-drift or device malfunction, unexpected changes in traffic and channel patterns can occur at any time \cite{gaddam2020detecting}. In an underwater communication system, it is extremely challenging to perform perfect channel state estimation \cite{khan2020channel}. Moreover, in applications where the communicating parties can move rapidly, e.g., self-driving vehicles \cite{ashjaei2021time}, or in an arbitrary manner, e.g., wireless AR/VR devices \cite{chen2022enhancing}, channel conditions can also change rapidly and thus difficult to estimate accurately.  
Therefore, scheduling policies relying on precise network-state knowledge may not be applicable to many real-world tasks; relying on such policies can result in significant performance degradation due to inaccurate information. 
Hence, network scheduling \textit{without} 
instantaneous knowledge and accurate estimation of the network state is important both, in theory and in practice, i.e., it can significantly improve robustness and availability of large-scale networked systems while reducing operational and maintenance costs.

To this end, in this paper, we focus on a novel  \emph{scheduling without network-state knowledge} formulation. Specifically,  we focus on a one-hop scheduling task, where a single-server serves $K$ queues, each corresponding to a job type. 
%
%
The server chooses a single queue to serve at each time slot.  The network dynamics, i.e., arrival and service rates, evolve in an \textit{oblivious adversarial} manner and are unknown before the scheduling decision. Moreover, the service outcome is only observed after the action with bandit feedback, i.e., only the served queue produces an observation. Our goal is to seek an efficient scheduling policy to stabilize the network. 

To solve this problem, we introduce novel learning-augmented scheduling algorithms, inspired by the celebrated \texttt{MaxWeight} queue scheduling algorithm \cite{neely2010stochastic} and the success of the \texttt{EXP3} family of algorithms on non-stationary Multi-Armed Bandits (MAB) problems \cite{auer2002nonstochastic}. The proposed algorithms are capable of stabilizing  a non-stationary 
system, 
as long as the system can be stabilized by a randomized policy whose total variation of probabilities to serve each type of job is not too large. Perhaps surprisingly, our algorithms rely on neither knowing the network statistics before-hand, nor on complicated explicit real-time estimation of the system. As a result, compared to its network-state knowledge dependent counterparts, our algorithms are naturally more robust to jitter and unexpected traffic/service patterns in the system. 
Indeed, \Cref{sec-apdx-experiments} gives a numerical comparison of our algorithms with their accurate knowledge dependent counterparts. From this comparison, we show that  the presented algorithms do give superior performance on systems with  service state noise, as depicted in \Cref{figure-curve-intro}.

\begin{wrapfigure}[15]{r}{0.45\textwidth}
    \centering
    \includegraphics[width=0.45\textwidth]{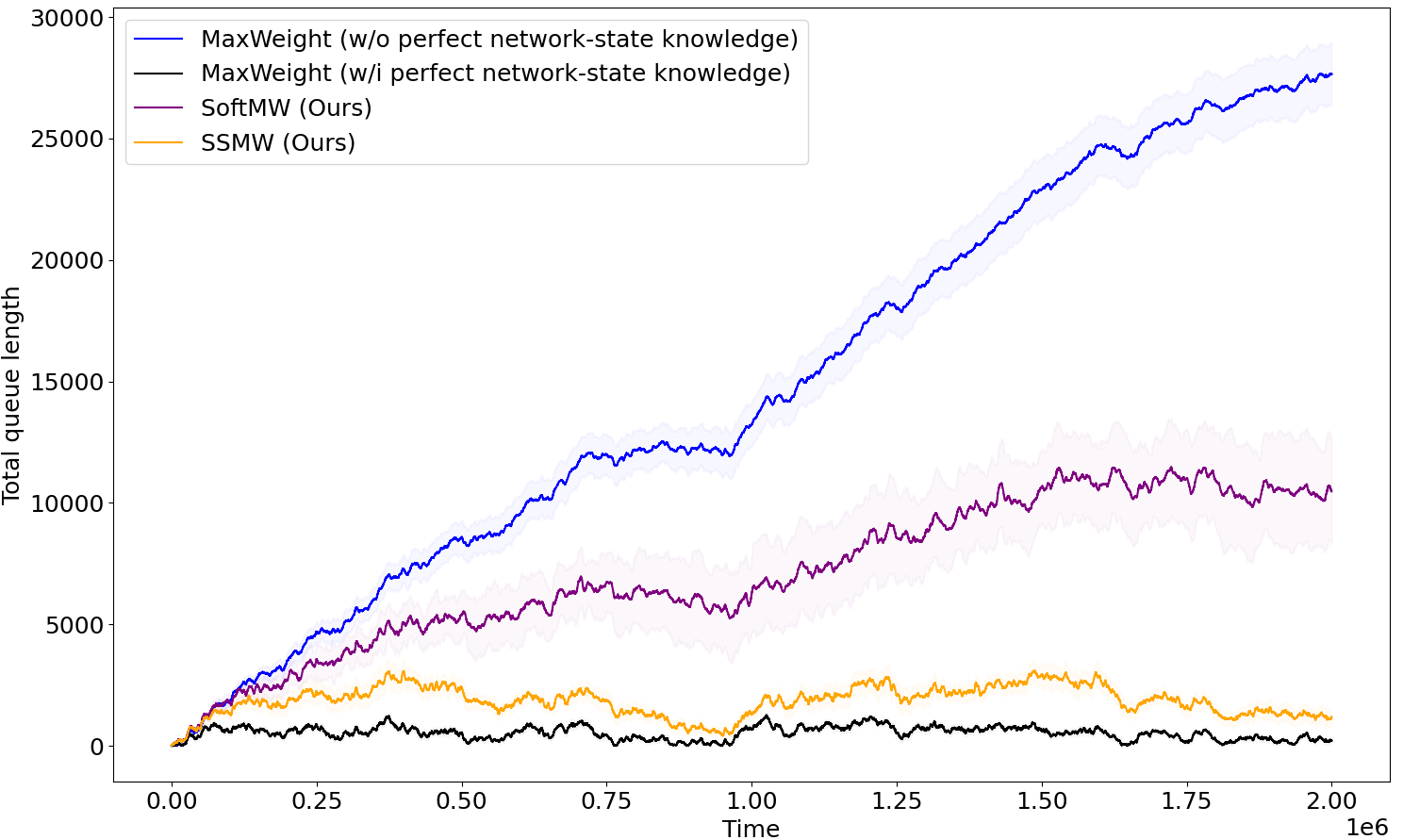}
    \caption{Numerical evaluation on a non-stationary system (see \Cref{sec-apdx-experiments} for details)}
    \label{figure-curve-intro}
\end{wrapfigure}

Our work differs from the existing learning-augmented network control literature, e.g., \cite{choudhury2021job,krishnasamy2021learning,hsu2022integrated,yang2022maxweight}, in the following aspects. \cite{choudhury2021job,hsu2022integrated,krishnasamy2021learning} study the scheduling or load-balancing tasks on stationary systems with rate statistics unknown before-hand, while in our setting the system can be time-varying and adversarial. \cite{yang2022maxweight} also considers non-stationary systems, but they assume smoothly time-varying service rates and explicitly estimate the instantaneous service rates using exponential average and discounted UCB bonus. Compared to these works, 
our approach requires neither to explicitly optimize off-line problems nor to explicitly probe and estimate the instantaneous channel states, but rather uses adversarial bandit learning techniques to coherently explore and stabilize the system at the same time. 

Our contributions in this work can be summarized as follows:
\begin{itemize}
    \item We propose two novel scheduling algorithms \texttt{SoftMW} (\Cref{softmw}) and \texttt{SSMW} (Sliding-window \texttt{SoftMW}, \Cref{ssmw}) that are capable of scheduling one-hop queueing systems without channel state knowledge,  while stablizing the systems under mild conditions on the time-variation of the reference randomized policy. 
    \item In designing these two algorithms, we carefully combine techniques from online bandit learning and Lyapunov drift based scheduling approaches and analysis. The bandit part can guarantee that our algorithms' ``regret'' against an unknown time-varying randomized policy over a finite time-horizon is small. The regret guarantee can be coupled (in an innovative manner) with Lyapunov drift analysis to develop the stability result (see \Cref{sec-softmw-analysis,sec-ssmw-analysis}).
    \item We extend the \texttt{EXP3.S} algorithm \cite{auer2002nonstochastic}, originally designed for adversarial MAB problems with bounded rewards, such that time-varying learning rates and exploration rates are applicable to handling unboundedly large feedback (see \Cref{sec-exp3s}, \Cref{exp3s}). This extended \texttt{EXP3.S} algorithm (we call \texttt{EXP3.S+}) is used as a building block in \texttt{SoftMW} and \texttt{SSMW}. However, it is also of independent interest beyond the scope of queueing.
    \item We further generalize our results to the setting where  arrivals and departures have bounded moments instead of being deterministically bounded (see \Cref{sec-moment}). 
    We present \texttt{SoftMW+} (\Cref{softmw-moment}) and \texttt{SSMW+} (\Cref{ssmw-moment}) that are capable of stabilizing the system. 
\end{itemize}

\Cref{table-overview} provides a comparison summary  between our proposed algorithms and closely related efforts.
To our knowledge, our work is the first to utilize adversarial MAB algorithms with dynamic regret guarantees in queueing systems scheduling. Most prior work is based on epsilon-greedy or Upper Confidence Bounds (UCB), where the assumption is needed that the system is either stationary or non-stationary but with arrival (departure) rates having adequate smoothness. 
Hence, our algorithms can apply to more general and complex settings. We believe our approach can facilitate novel and interesting insights to \texttt{MaxWeight}-type as well as other queueing scheduling algorithm design problems.

\begin{threeparttable}
\caption{Overview of Our Algorithms and Closely Related Work}
\fontsize{8pt}{9.6pt}\selectfont
\label{table-overview}
\renewcommand{\arraystretch}{1.5}
\begin{tabular}{|c|l|c|}\hline
Algorithm & \multicolumn{1}{|c|}{Systems Stabilizable} & Average Queue Length \\\hline
\multirow{2}{*}{\texttt{MaxWeight} \cite{tassiulas1993dynamic}} & Homogeneous Jobs & $\O(\frac{KM^2}{\epsilon})$ \\\cline{2-3}
& Assumption~\ref{assumption-theta} + service rate forecasts & $\O(\frac{C_WKM^2}{\epsilon})$ \\\hline
\multirow{2}{*}{\shortstack{\texttt{MaxWeight} with
\\Discounted UCB}\cite{yang2022maxweight}} & \multirow{2}{*}{\shortstack[l]{Assumption~\ref{assumption-theta},\\Service rates have smoothness matching the discounting factor\tnote{1}}} & \multirow{2}{*}{$\left(MK \epsilon^{-1}\right)^{\O(1/\delta)}$}\\
& & \\\hline
\multirow{2}{*}{\shortstack{\texttt{SoftMW}
\\(\textbf{Ours}, \Cref{softmw})}} & \multirow{2}{*}{\shortstack[l]{Assumption~\ref{assumption-theta},\\Assumption~\ref{assumption-delta} ($\O(T^{\frac 1 2 -\delta})$ reference policy total variation)}} & \multirow{2}{*}{$\O(\frac{C_WKM^2}{\epsilon})$}\\
& & \\\hline
\multirow{2}{*}{\shortstack{\texttt{SSMW}
\\(\textbf{Ours}, \Cref{ssmw})}} & \multirow{2}{*}{\shortstack[l]{Assumption~\ref{assumption-theta},\\Assumption~\ref{assumption-delta3} ($\O(T^{1 -\delta})$ reference policy time-homogeneous total variation)}} & \multirow{2}{*}{$\left((1+C_V)MK \epsilon^{-1}\right)^{\O(1/\delta)}$}\\
& & \\\hline
\multirow{2}{*}{\shortstack{\texttt{SoftMW+}
\\(\textbf{Ours}, \Cref{softmw-moment})}} & \multirow{2}{*}{\shortstack[l]{Assumption~\ref{assumption-theta}, Assumption~\ref{assumption-delta} ($\O(T^{\frac 1 2 -\delta})$ reference policy total variation),\\Arrivals and departures can be unbounded, but have bounded $\alpha$-th moment, $\alpha \cdot \delta > 7$}} & \multirow{2}{*}{$\O(\frac{C_WKM^2}{\epsilon})$}\\
& & \\\hline
\multirow{3}{*}{\shortstack{\texttt{SSMW+}
\\(\textbf{Ours}, \Cref{ssmw-moment})}} & \multirow{3}{*}{\shortstack[l]{Assumption~\ref{assumption-theta}, \\Assumption~\ref{assumption-delta3} ($\O(T^{1 -\delta})$ reference policy time-homogeneous total variation),\\Arrivals and departures can be unbounded, but have bounded 2nd moment}} & \multirow{3}{*}{$\left((1+C_V)MK \epsilon^{-1}\right)^{\O(1/\delta)}$}\\
& & \\
& & \\\hline
\end{tabular}
\begin{tablenotes}
    \item[1] \cite{yang2022maxweight} uses similar assumption where the one-step service rate drift of each channel is universally upper-bounded by some polynomial of $(1-\gamma)^{-1}$. Here $\gamma$ is a hyper-parameter of their algorithm, namely the discounting factor in UCB. 
  \end{tablenotes}
\end{threeparttable}

\section{Notation}
Throughout this paper, for $n\ge 1$, we denote the set $\left\{1,2,\ldots, n\right\}$ by $[n]$ and the $(n-1)$-dimensional probability simplex over $[n]$ by $\triangle^{[n]}$. We use bold English letters (e.g., $\mathbf Q_t$, $\mathbf S_t$) and Greek letters with arrows above (e.g., $\vec \sigma_t$, $\vec \lambda_t$) to denote vector-valued variables. We use $\mathbf 0$ to denote the all-zero vector, and $\mathbf 1$ to denote the all-one vector. We use $\mathbf{1}_{i}$ to denote the one-hot vector with $1$ on the $i$-th coordinate, i.e., $(\mathbf 1_i)_j=1$ if $i=j$ and $0$ otherwise. We use $\mathbbm 1[\text{statement}]$ to denote the indicator of a given statement; its value is taken as $1$ if the statement holds and $0$ otherwise. We use $\mathbf x \odot \mathbf y$ to denote the element-wise product of two vectors $\mathbf x$ and $\mathbf y$.

Let $f$ be a strictly convex function defined on some convex domain $A \subseteq \mathbb{R}^K$. For any $\mathbf x,\mathbf y \in A$, if $\nabla f(\mathbf x)$ exists, we write the Bregman divergence between $y$ and $x$ induced by $f$ as
\begin{equation*}
    D_f(\mathbf y, \mathbf x) \triangleq f(\mathbf y) - f(\mathbf x) - \langle \nabla f(\mathbf x), \mathbf y-\mathbf x\rangle
\end{equation*}

We use $f^\ast(\mathbf y) \triangleq \sup_{\mathbf x \in \mathbb{R}^K} \left\{ \langle \mathbf y, \mathbf x \rangle - f(\mathbf x) \right\}$
to denote the convex conjugate of $f$. 

We use $\Otil$, $\tilde{\Omega}$ or $\tilde{\Theta}$ to suppress poly-logarithmic factors in $T$ (the length of the decision horizon) and $K$ (the number of queues).
Unless stated otherwise, we use
\begin{equation*}
    \mathcal{F}_t = \sigma\left(a_1, \ldots, a_t, \mathbf Q_0, \ldots, \mathbf Q_t, \mathbf A_1, \ldots, \mathbf A_t, S_{1,a_1},\ldots, S_{t,a_t}\right)
\end{equation*}
for any $t\ge 0$ to denote the filtration of $\sigma$-algebra  when studying random quantities indexed by time, i.e., $\mathcal F_t$ is generated by all decisions and quantities visible to a scheduling policy at the end of $t$-th time slot.

\section{Problem Setting}

We consider the problem of scheduling $K$ job types on a single work-conserving server with a slotted time system. Each arriving job first joins a queue associated with its type $i$, which we denote by $Q_i$. Denote by $A_{t,i}$ the amount of arriving jobs of type $i$ in the $t$-th time slot, and by $S_{t,i}$ the maximum amount of jobs of type $i$ the server can serve in the $t$-th time slot. At the beginning of each time slot $t$, the server chooses \textit{exactly one} type of a job $a_{t} \in [K]$ to serve. Denote by $Q_{t,i}$ the queue length of type $i$ jobs at the end of time slot $t$. Then, each $Q_{t,i}$  evolves according to the following equation:
\begin{equation*}
    Q_{t,i} = \max\left\{Q_{t-1,i} + A_{t,i} - S_{t,i}\mathbbm 1[i = a_t], 0\right\}
\end{equation*}
where $\mathbf Q_0 = (Q_{0, 1}, ..., Q_{0, K})=\mathbf 0$.
At the beginning of each time slot $t$, the latest queue lengths $Q_{t-1,1},\ldots, Q_{t-1,K}$ are available to the server for making new decisions. The maximum service amount of past actions $S_{0,a_0},\ldots,S_{t-1,a_t}$ are also visible to the server. 

We assume that there are two sequences of distributions $\{\mathcal A_1, \mathcal A_2, \ldots\}$ and $\{\mathcal S_1, \mathcal S_2, \ldots\}$, all fixed before the queue process starts, and their statistics are known to the scheduler before-hand. All distributions $\mathcal A_t$s and $\mathcal S_t$s are supported on $[0,M]^K$, where $M$ is a constant known before-hand. We further assume that each $\mathbf A_t$ is randomly sampled from $\mathcal A_t$, each $\mathbf S_t$ is sampled from $\mathcal S_t$, and all $\mathbf A_t$s and $\mathbf S_t$s are independent random vectors.
We denote by $\vec \lambda_t$ the mean of $\mathcal A_t$, and by $\vec \sigma_t$ the mean of $\mathcal S_t$.

Our objective is to design a scheduling policy, under which we have the following upper-bound on the average expected queue lengths:
\begin{equation*}
    \frac 1 T\sum_{t=0}^{T-1}\sum_{i=1}^K\mathbb E[Q_{t,i}] = \O(1).
\end{equation*}
We say a scheduling policy \textit{stabilizes} the system, or the system is \textit{stable} under some scheduling policy, if the above bound holds.

Classical scheduling tasks on stationary systems (e.g., \cite{choudhury2021job,krishnasamy2021learning}) correspond to the case where $\mathcal A_t=\mathcal A_1$ ($\mathcal S_t=\mathcal S_1$), i.e.,  the distributions are time-invariant in our setting. In our problem, it is complicated to explore and estimate the time-varying service distributions subject to the queue stability. 

\section{A Sufficient Condition for Stabilizing the System}
\label{sec-assumption-theta}
\nop{this part does not look necessary to me?  ------
Recall that the \texttt{MaxWeight} policy \cite{tassiulas1993dynamic} chooses 
\begin{equation}
\label{eq-mw}
    a_t \in \argmax_{i\in [K]} Q_{t-1,i}\sigma_{t,i},
\end{equation} 
given access to the service rate vector $\vec \sigma_t$ at the beginning of time slot $t$. It is well-known (see e.g., \cite{neely2010stochastic}) that for stationary arrival and service processes, a problem instance can be stabilized by some scheduling policy (i.e., the problem instance is in the ``capacity region'') if and only if it can be stabilized by \texttt{MaxWeight}.

For our \emph{non-stationary} setting, it is still unknown whether such equivalence exists, but at least we have the following sufficient condition for Max-Weight to stabilize a problem instance: }

In our paper, we make the following assumption on the system, which is analogous to the capacity region definition in stationary network scheduling \cite{neely2010stochastic}, and 
can be viewed as a generalized stability condition for scheduling in adversarial environments. 
\begin{assumption}[Piecewise Stabilizability]
\label{assumption-theta}
\noindent There exist $C_W\ge0$, $\epsilon > 0$, $\vec \theta_1, \vec \theta_2, \cdots \in \Delta^{[K]}$ and a partition of $\mathbb N_+$ into intervals $W_0, W_1,\cdots$, such that for any $T\ge 1$ we have
\begin{equation}
    \sum_{i:\min_{t\in W_i} t < T} \left( \lvert W_i\rvert - 1 \right)^2 \le C_W T \label{eq-assumption-theta-C_W}
\end{equation}
and for any $i\ge 0$ and $j\in [K]$ we have
\begin{equation}
    \frac 1 {\lvert W_i \rvert} \sum_{t\in W_i} \theta_{t,j} \sigma_{t,j} \ge \epsilon + \frac 1 {\lvert W_i \rvert} \sum_{t\in W_i} \lambda_{t,j}. \label{eq-assumption-theta-eps}
\end{equation}
\end{assumption}
Assumption \ref{assumption-theta} can be regarded as a generalizition of the $(W,\epsilon)$-constrained dynamics in \cite{liang2018minimizing}. It essentially assumes that the time horizon can be divided into  intervals, within which there exist stationary policies that can stablize the network (Eq. \eqref{eq-assumption-theta-eps}). 
As a quick sanity check, for stationary instances where the arrival rate vector is in the interior of the capacity region, Assumption~\ref{assumption-theta} is automatically satisfied with $C_W = 0$ (hence all $W_i$s are singleton sets) and all $\vec \theta_i$s are equal to some fixed element $\vec \theta \in \Delta^{[K]}$, which is a randomized policy capable of stabilizing the system.

\noindent\textbf{Remark.} In fact, under the above assumption, by a quadratic Lyapunov drift argument (see \Cref{prop-lyapunov-queue-length}), we can also show that a policy, in which at each time step $t$ we serve a type of job $a_t$ independently at random according to the distribution indicated by $\vec \theta_t$, can stabilize the system as well (require knowing $\vec \theta_t$ beforehand). 
We call $\{\vec \theta_t:t \ge 1\}$ \emph{the reference mixed action sequence}, and refer to the above randomized policy induced by $\{\vec \theta_t:t \ge 1\}$ as \emph{the reference randomized policy}.

With Assumption~\ref{assumption-theta}, in general, it is still a challenging problem to scheduling the system. For our main results in \Cref{sec-softmw}, we need another technical assumption presented below.

\begin{assumption}[Reference Policy Stationarity]
\label{assumption-delta}
For the reference mixed action sequence $\{\vec \theta_t\}$ in Assumption~\ref{assumption-theta}, there exist some $\delta > 0$ and $C_V > 0$ such that
\begin{equation*}
    \sum_{t=1}^{T-1} \lVert \vec \theta_{t+1} - \vec \theta_t \rVert_1 \le C_V T^{\frac 1 2 - \delta}
\end{equation*}
for any $T \ge 1$.
\end{assumption}

Intuitively speaking, Assumption~\ref{assumption-delta} says that the sequence $\{\vec \theta_t\}$ (and hence the environment) does not change in a very abrupt way. 
Similar smooth assumptions have also been made in existing results, e.g., \cite{yang2022maxweight}.  \footnote{Strictly speaking,  \cite{yang2022maxweight} introduces a smoothness assumption on the arrival and service rate rather than the reference randomized policy.} In \Cref{sec-ssmw}, we will also study when can we handle problems where the reference policy has significantly larger 
variation.



\section{Queue Scheduling with only Bandit Feedback}
\label{sec-softmw}


In contrast to the setting with perfect network state knowledge, in our case, there is no such accurate channel condition for the scheduler. Specifically, the server only receives a bandit feedback for each time step's actual service, i.e., only $S_{t,a_t}$ is known after the service decision $a_t$ is made. 

In this section, we present a novel algorithm, which is capable of stabilizing the system using only bandit feedback, $S_{t,a_t}$. 
Our core idea is to embed a suitable Multi-Armed Bandit algorithm into the \texttt{MaxWeight} scheduler \cite{tassiulas1993dynamic},  so that the term $\mathbb E[\sum_{t=1}^{T} Q_{t-1,a_t}S_{t,a_t}]$, which is the key ingredient of \texttt{MaxWeight}, is guaranteed to be not too far from $\mathbb E[\sum_{t=1}^{T} \langle \mathbf Q_{t-1} \odot \mathbf S_t, \vec \theta_t \rangle]$. Given access to $\vec \sigma_t$, \texttt{MaxWeight} achieves this by greedily choosing $a_t = \argmax_i Q_{t-1,i}\sigma_{t,i}$ at each time step $t$. However, when $\vec \sigma_t$ is unknown and time-varying, it is hard to guarantee that each summand $Q_{t-1,a_t}S_{t,a_t}$ is large. Thus, we  focus on optimizing the whole sum $\mathbb E[\sum_{t=1}^{T} Q_{t-1,a_t}S_{t,a_t}]$.




In the remainder of this section, we will first present \texttt{EXP3.S+}, an extended version of the \texttt{EXP3.S} \cite{auer2002nonstochastic} algorithm for adversarial MAB (\Cref{sec-exp3s}). \texttt{EXP3.S+} has adequate flexibility to serve as an important building block of our novel scheduling algorithm \texttt{SoftMW} (\Cref{sec-softmw-real}). We also present its performance guarantee, as it is key for understanding our later analysis. 
Finally, in \Cref{sec-softmw-analysis}, we outline analysis of \texttt{SoftMW} and describe several important novel techniques to relate 
adversarial MAB learning to Lyapunov drift analysis. 

\subsection{\texttt{EXP3.S+}: An Extended Version of \texttt{EXP3.S}}
\label{sec-exp3s}
We first present  \texttt{EXP3.S+}, which extends the \texttt{EXP3.S} algorithm \cite{auer2002nonstochastic}, designed originally for solving adversarial Multi-Armed Bandit (MAB) problems, 
to address the potentially unbounded queue lengths in queueing systems, which cannot be directly handled by existing bandit algorithms. 

More formally, \texttt{EXP3.S+} applies to the following scenario: there is an agent and an adversary simultaneously making decisions on a finite-length time-horizon $t=1\ldots T$. At each time $t$, the agent chooses an $\mathbf x_t \in \Delta^{[K]}$  deterministically based on observed history, then samples $a_t \in [K]$ according to $\mathbf x_t$. Simultaneously (at time $t$), the adversary chooses $\mathbf g_t \in \mathbb R^{K}$ deterministically, based on observed history. Then, $g_{t,a_t}$ is revealed to the agent. The high-level objective for the agent is to maximize the cumulative feedback $\sum_{t=1}^{T}g_{t,a_t}$. The details of our \texttt{EXP3.S+} are described in \Cref{exp3s}. 
\begin{algorithm2e}
\caption{\texttt{EXP3.S+}}
\label{exp3s}
\LinesNumbered
\DontPrintSemicolon
\KwIn{Number of actions $K$, time-horizon length $T$,  initial mixed action $\mathbf x_1 \in \Delta^{[K]}$}
\KwOut{A sequence of actions $a_1, a_2,\ldots, a_T \in [K]$}
\kwIntermediate{A sequence of learning rates $\eta_1,\eta_2,\ldots,\eta_T \in \mathbb R_+$, a sequence of implicit exploration rates $\beta_1, \beta_2, \ldots, \beta_T \in [0,1/K]$, a sequence of explicit exploration rates $\gamma_1,\gamma_2,\ldots,\gamma_T \in [0,1/2]$, a sequence of explicit exploration normal vectors $\mathbf e_1, \mathbf e_2,\ldots, \mathbf e_T \in \Delta^{[K]}$}
\BlankLine
$\Psi(\mathbf x) \triangleq \sum_{i=1}^K (x_i \ln x_i - x_i)$ \;
\For{$t=1,2,\ldots,T$}{
    Choose $\beta_t$, $\eta_t$, $\mathbf e_t$ and $\gamma_i$ \;
    Below denote by $\Delta^{[K], \beta_t} \triangleq \{\mathbf x \in \Delta^{[K]} : \mathbf x_i \ge \beta_t
\hspace{1em} \forall i \in [K]\}$ \;
    $\mathbf p_t \gets (1-\gamma_t)\mathbf x_t + \gamma_t \mathbf e_t$ \;
    Sample $a_t \sim \mathbf p_t$, take action $a_t$, observe $g_{t,a_t}$ \label{line-exp3s-feedback} \;
    $\mathbf {\tilde g}_t \gets \begin{cases} g_{t,a_t} / p_{t,a_t} & \text{the }a_t\text{-th coordinate}\\
    0 & \text{the other coordinates}\end{cases}$ \;
    $\mathbf x_{t+1} \gets \argmin_{\mathbf x' \in \triangle^{[K],\beta_t}} \left\langle - \eta_t \tilde {\mathbf g}_t, \mathbf x' \right\rangle + D_\Psi(\mathbf x', \mathbf x_t)$ \label{line-exp3s-argmax}
}
\end{algorithm2e}

\noindent \textbf{Remark.} The amplitude of feedback value $g_{t,a_t}$ in \Cref{line-exp3s-feedback} is crucial to the correctness of \texttt{EXP3.S}. The original \texttt{EXP3.S} algorithm in \cite{auer2002nonstochastic} uses a constant learning rate $\eta$ and a constant exploration rate $\gamma$ across all $T$ time steps. 
However, the algorithm can only support problems with feedback value no more than $\eta^{-1}\gamma$, and does not apply to our setting, where the queue length (which is the reward of \texttt{EXP3.S}) can go unbounded.  
%
For our purpose, in the presented algorithms, we feed $Q_{t-1,a_t}S_{t,a_t}$ into \texttt{EXP3.S+} as the reward value, which is a quantity that can be arbitrarily large (since $Q_{t-1,a_t}$ can be arbitrarily large). In \texttt{EXP3.S+}, the learning rates and exploration rates can both be time-varying, and the exploration rates can even be action-dependent (it allows specifying any $\mathbf e_t \in \Delta^{[K]}$ rather than $\mathbf 1/K$). 


The formal performance guarantee of \texttt{EXP3.S+} for $\sum_{t=1}^{T}g_{t,a_t}$ is given in \Cref{thm:extended-exp} below. 
\begin{theorem}[\texttt{EXP3.S+} Dynamic Regret Guarantee]
\label{thm-exp3s}
    During the execution of \Cref{exp3s}, for any fixed sequence $\vec \theta_1,\ldots, \vec \theta_T \in \Delta^{[K]}$, if w.p.1 the following events happen, 
    \begin{itemize}
        \item[(i)] $\mathbf x_1 \in \Delta^{[K],\beta_1}$,
        \item[(ii)] $\mathbf g_t  \le \eta_t^{-1}\gamma_t \mathbf e_t$ for all $1\le t \le T$,
        \item[(iii)] $\eta_1 \ge \eta_2 \ge \cdots \ge \eta_T$,
        \item[(iv)] $\beta_1 \ge \beta_2 \ge \cdots \ge \beta_T$,
        \item[(v)] $\vec \theta_t \in \Delta^{[K],\beta_t}$ for all $1\le t \le T$,
    \end{itemize}
    then let
    \begin{equation*}
        V \triangleq \sum_{t=1}^{T-1} \lVert \vec \theta_{t+1} - \vec \theta_t \rVert_1,
    \end{equation*}
    we will have
    \begin{equation*}
        \mathbb E\left[\sum_{t=1}^T \langle \mathbf g_t, \vec \theta_t \rangle \right] - \mathbb E\left[\sum_{t=1}^T g_{t,a_t}\right] \le (1+V) \mathbb E\left[\eta_T^{-1}\ln \frac 1 {\beta_T}\right] + e \mathbb E\left[ \sum_{t=1}^T \eta_t \lVert \mathbf g_t \rVert_2^2 \right] + \mathbb E\left[\sum_{t=1}^T\gamma_t \langle \mathbf g_t, \mathbf e_t\rangle \right].
    \end{equation*}
    \label{thm:extended-exp}
\end{theorem}

In \Cref{sec-apdx-exp3s}, we provide a formal proof for \Cref{thm-exp3s} using an analysis based on Online Mirror Descent \cite{hazan2016introduction}, which is much more suitable for handling time-varying learning rates compared to the classical sum-of-exp potential function approach in \cite{auer2002nonstochastic}.
%
We also discuss a practical implementation of the $\argmax$ calculation (at \Cref{line-exp3s-argmax}) in \Cref{sec-apdx-exp3s-argmax}. We note that \Cref{exp3s} and its analysis can be of independent interest and applied to problems other than stochastic network scheduling.

\subsection{Soft Max-Weight Scheduling using \texttt{EXP3.S+}}
\label{sec-softmw-real}

We now present our novel scheduling algorithm, \texttt{SoftMW}, in \Cref{softmw}. \texttt{SoftMW} is based on carefully designed feedback signals as well as parameters and learning rates in \texttt{EXP3.S+}. Its name refers to the computation in \texttt{EXP3.S+} (\Cref{exp3s}) that is heavily based on the softmax operation (see \Cref{sec-apdx-exp3s-argmax}). 

The intuitive reason why \Cref{softmw} works is as follows. We use \texttt{EXP3.S+} in a carefully designed way to drive the scheduling process, so that the effect of \Cref{softmw} is very closed to (or better than) the reference randomized policy given by Assumption~\ref{assumption-theta}, in the sense that under \Cref{softmw}, the queues' total quadratic Lyapunov drift is only slightly larger (or even smaller) than that under the reference randomized policy. Therefore, \Cref{softmw} has similar (or even stronger) capability of stabilizing the system. 

\begin{algorithm2e}
\caption{\texttt{SoftMW} (\textbf{Soft} \textbf{M}ax\textbf{W}eight)}
\label{softmw}
\LinesNumbered
\DontPrintSemicolon
\KwIn{One-step arrival/service upper-bound $M > 0$, Number of job types $K$, Problem instance smoothness parameter $\delta$ > 0}
\KwOut{A sequence of job types to serve $a_1, a_2,\ldots \in [K]$}
\BlankLine
Initialize an \texttt{EXP3.S+} instance with $K$ available actions and $\mathbf x_1 = \mathbf 1/K$\;
\For{$t=1,2,\ldots$}{
    Pick the following parameters of \texttt{EXP3.S+} for time slot $t$: \;
    \Indp$\beta_t \gets t^{-3}/ K$\;
    $\eta_t = \left( t^{-(\frac 1 4 - \frac \delta 2)} M \sqrt{ 86M^2 K^6 t^{\frac 3 2} + \sum_{s=0}^{t -1} \lVert \mathbf Q_s \rVert_2^2} \right)^{-1}$ \;
    $\mathbf e_t = \mathbf Q_{t - 1} / \lVert \mathbf Q_{t - 1} \rVert_1$\;
    $\gamma_t = M\eta_t \lVert \mathbf Q_{t - 1}\rVert_1 = \lVert \mathbf Q_{t - 1}\rVert_1 \left( t^{-(\frac 1 4 - \frac \delta 2)} \sqrt{ 86M^2 K^6 t^{\frac 3 2} + \sum_{s=0}^{t-1} \lVert \mathbf Q_s \rVert_2^2} \right)^{-1}$ \;
    \Indm Take a new action decision output $a_t$ from \texttt{EXP3.S+}, serve the $a_t$-th queue 
    , regard $Q_{t - 1,a_t}S_{t,a_t}$ as a new feedback $g_{t,a_t}$ and feed it into the current \texttt{EXP3.S+} instance 
}
\end{algorithm2e}

\Cref{softmw}'s average queue length bound on any finite time-horizon is given in \Cref{thm:softmw}. 
\begin{theorem}
\label{thm-softmw-stab}
    For problem instances satisfying Assumptions~\ref{assumption-theta} and \ref{assumption-delta}, \texttt{SoftMW} (\Cref{softmw}) guarantees
    \begin{equation*}
        \frac 1 T \mathbb E\left[ 
\sum_{t=1}^{T} \lVert \mathbf Q_t\rVert_1  \right] \le \frac {2(K+1)M^2 + 4C_W (KM^2 + \epsilon KM) } \epsilon + o(1).
    \end{equation*}
    In particular, the system is stable.
    \label{thm:softmw}
\end{theorem}

\noindent \textbf{Remark.} As a quick sanity check, for stationary problem instances, \Cref{thm-softmw-stab} gives $\O(KM^2/\epsilon)$ average queue length bound, which coincides with the classical result we can achieve in stationary problems (\cite{neely2010stochastic} Sec. 3.1 ).
In fact, one can show that for both (i) pretending to have accurate one-step forecasts for service rates and running vanilla \texttt{MaxWeight}, and (ii) running the reference randomized policy $\{\vec \theta_t\}$ specified in Assumption~\ref{assumption-theta}, the average queue length bounds via a standard quadratic Lyapunov analysis are $\O\left(\frac{(C_W K + K+1)M^2}{\epsilon}\right)$. Therefore, informally, in terms of queue length bound, the overhead due to \texttt{SoftMW} on problem instances satisfying Assumption~\ref{assumption-delta} is insignificant.



\subsection{Queue Stability Analysis Outline}
\label{sec-softmw-analysis}


In this section, we give a brief outline of how to formally establish the queue stability result (\Cref{thm-softmw-stab}). We first review the general procedure from quadratic Lyapunov drift analysis. Then, we show that \texttt{EXP3.S+} scheduling can lead to terminal Lyapunov function values close to the reference policy in Assumption \ref{assumption-theta}, differing by a term proportional to $\sqrt{\sum \lVert\mathbf Q_t\rVert_2^2}$. Finally, we relate this $\sqrt{\sum \lVert\mathbf Q_t\rVert_2^2}$ term with the queue lengths ($\sum \lVert\mathbf Q_t\rVert_1$) we want to bound, and show that the Lyapunov terminal value bound leads to an average queue length bound.

\subsubsection{Recap of Lyapunov Drift Analysis}
\label{sec-softmw-analysis-lyapunov}
\ \\
In our analysis, we use standard results from quadratic Lyapunov drift analysis \cite{neely2010stochastic}.
Conventionally, we define 
\begin{equation*}
    L_t \triangleq \frac 1 2 \lVert\mathbf Q_t\rVert_2^2 = \frac 1 2 \sum_{i=1}^K Q_{t,i}^2,
\end{equation*}
as the quadratic Lyapunov function of the queue lengths. We first have the following standard lemma regarding the drift upper bound.  
\begin{lemma}[General quadratic Lyapunov Drift Upper-bound \cite{neely2010stochastic}]
\label{lemma-quad-lyapunov}
    Consider \textit{any} scheduling policy for this queueing system and suppose that the policy randomly picks a job type $a_t$ according to a probability distribution $\mathbf p_t$ (which may depend on the system's history, i.e., $\mathbf p_t$ is an $\mathcal F_{t-1}$-measurable random vector supported on $\Delta^{[K]}$). Let $\mathbf Q_t$ denote the queue length vector under that policy. We have
    \begin{align*}
        \mathbb E\left[\left. L_t - L_{t-1} \right\rvert \mathcal F_{t-1}\right] & \le \frac{(k+1)M^2} 2 + \langle \mathbf Q_{t-1}, \vec \lambda_t - \vec \sigma_t \odot \mathbf p_t\rangle \\
        & = \frac{(k+1)M^2} 2 + \langle \mathbf Q_{t-1}, \vec \lambda_t\rangle  - \mathbb E\left[\left.Q_{t-1, a_t} S_{t,a_t}\right\rvert \mathcal F_{t-1}\right]
    \end{align*}
    for any $t \ge 1$. By summing the inequalities over $1\le t\le T$, taking total expectation and then rearranging the terms, we get
    \begin{equation}
    \label{eq-quad-lyapunov-total}
       \mathbb E\left[\sum_{t=1}^T Q_{t-1,a_t}S_{t,a_t} -\langle \mathbf Q_{t-1}, \vec \lambda_t\rangle\right] \le \frac{(K+1)M^2T} 2
    \end{equation}
    for any time horizon length $T\ge 1$.
\end{lemma}

Next, we have Lemma \ref{lemma-quad-lyapunov-theta} regarding the drift value under the reference policies. As in the standard Lyapunov drift analysis \cite{neely2010stochastic}, this bound will be useful for deriving queue length results for queue-based policies.  
\begin{lemma}[Negative Lyapunov Drift under Reference Policy] 
\label{lemma-quad-lyapunov-theta}
    Suppose Assumption~\ref{assumption-theta} holds. Consider \textit{any} scheduling policy for this queueing system, under which the queue length vectors are denoted by $\{\mathbf Q_t\}$. Let $\{\vec\theta_t : t\ge 1\}$ be the sequence of probabilities to serve each queue as defined in Assumption~\ref{assumption-theta}. Then, for any  time horizon length $T\ge 1$, we can find a constant $\mathcal T_T$ that depends only on $T$, such that $T \le \mathcal T_T \le T + \sqrt{\frac T {C_W}} + 1$ and
    \begin{align*}
       \mathbb E\left[\sum_{t=1}^{\mathcal T_T} \langle \mathbf Q_{t-1}, \vec \sigma_t \odot \vec \theta_t - \vec \lambda_t\rangle\right] & \ge \epsilon \mathbb E \left[\sum_{t=1}^{\mathcal T_T} \lVert\mathbf Q_{t-1}\rVert_1\right] - (KM^2 + \epsilon KM) C_W \mathcal T_T \\
       & \ge \epsilon \mathbb E \left[\sum_{t=1}^T \lVert\mathbf Q_{t-1}\rVert_1\right] - (KM^2 + \epsilon KM) C_W \mathcal T_T.
    \end{align*}
    Here $\odot$ is the element-wise product, i.e., $\vec a\odot \vec b = (a_1b_1,\ldots,a_Kb_K)$, and $C_W$ is the constant defined in Assumption~\ref{assumption-theta}.
\end{lemma}
\begin{proof}
    See \Cref{proof-lemma-quad-lyapunov-theta}.
\end{proof}


Combining \Cref{lemma-quad-lyapunov} and \Cref{lemma-quad-lyapunov-theta}, we obtain the following important proposition for our analysis. 
\begin{proposition}[Sufficiently-Large-Weight Implies Queue Stability] 
\label{prop-lyapunov-queue-length}
Suppose Assumption~\ref{assumption-theta} holds, also suppose a scheduling policy guarantees the following.
\begin{equation*}
    \mathbb E\left[\sum_{t=1}^T Q_{t-1,a_t}S_{t,a_t} \right] \ge  \mathbb E\left[\sum_{t=1}^T \langle \mathbf Q_{t-1}, \vec \sigma_t \odot \vec \theta_t \rangle \right] - f(T)
\end{equation*}
for all $T\ge \max\{\frac {4} {C_W}, C_W\}$, where $f(T)$ is some non-negative, increasing function of $T$.  Then, we have
\begin{equation*}
    \frac 1 T \mathbb E\left[ 
\sum_{t=1}^{T} \lVert \mathbf Q_t\rVert_1  \right] \le \frac {(K+1)M^2 + 2C_W (KM^2 + \epsilon KM) } \epsilon + \frac {f(2T)}{\epsilon T}.
\end{equation*}
In particular, if $f(T)$ is $\O(T)$, then this policy stabilizes the system.
\end{proposition}
\begin{proof}
    See \Cref{proof-prop-lyapunov-queue-length}.
\end{proof}

\noindent \textbf{Remark.} \Cref{prop-lyapunov-queue-length} implies that for problem instances satisfying Assumption~\ref{assumption-theta}, serving the queue according to either $\{\vec \theta_t\}$ (the reference randomized policy) or the vanilla \texttt{MaxWeight} algorithm (assuming that service rate forecasts are available to the algorithm at the time of decision making), the average queue length will be no more than $\frac {(K+1)M^2 + 2C_W (KM^2 + \epsilon KM) } \epsilon$, as claimed earlier in \Cref{sec-assumption-theta}. This is because in both cases, we have $\mathbb E\left[\left.Q_{t-1, a_t} S_{t,a_t}\right\rvert \mathcal F_{t-1}\right] \ge \langle \mathbf Q_{t-1}, \vec \sigma_t \odot \vec \theta_t \rangle$. Hence, the condition in \Cref{prop-lyapunov-queue-length} holds with $f(T) = 0$ for these two policies.

In the remaining of the analysis, we will derive the corresponding $f(T)$ for \texttt{SoftMW+}, so that we can conclude the queue stability via an argument similar to \Cref{prop-lyapunov-queue-length}.

\subsubsection{From \texttt{EXP3.S+} Regret Bound to Lyapunov Function Value Bound}
\ \\

To build the queue stability result for \texttt{SoftMW} (\Cref{softmw}), our high-level idea is to develop the required condition in \Cref{prop-lyapunov-queue-length} such that $f(T)$ can also be properly controlled. Since \texttt{SoftMW} makes decisions based on \texttt{EXP3.S+}, intuitively, we should utilize the regret upper-bound result \Cref{thm-exp3s}. In order to do that, we need to verify that the required conditions (i)-(iv)\footnote{The reference policy $\{\vec \theta_t\}$ itself may not satisfies condition (v), but we will project each $\vec \theta_t$ onto $\Delta^{[K],\beta_t}$ as $\vec \theta'_t$, and only use \Cref{thm-exp3s} to obtain a regret bound against the action sequence $\{\vec \theta'_t\}$.} in \Cref{thm-exp3s} hold. 

In fact, in \texttt{SoftMW}, our choices of $\eta_t$s and $\beta_t$s are obviously decreasing, hence condition (iii) and (iv) hold. We choose $\mathbf x_1 = (1/K,\ldots, 1/K)$ thus condition (i) also holds; the choice of $\gamma_t$ and $\mathbf e_t$ also guarantees condition (ii). The real issue is whether $\gamma_\tau$'s exceed $\frac 1 2$. This is established in the following proposition. 

\begin{proposition}[Feasibility of the Exploration Rates in \texttt{SoftMW}]
\label{prop-softmw-exp-rates}
For all $t\ge 1$, we have $\gamma_t \le \frac 1 2$ in \texttt{SoftMW}.
\end{proposition}

\Cref{proof-prop-softmw-exp-rates} gives a detailed proof of \Cref{prop-softmw-exp-rates}. Having confirmed that the algorithm is feasible, we can now safely apply \Cref{thm-exp3s}, resulting in the following property of \texttt{SoftMW}.
\begin{lemma}[\texttt{SoftMW} Large-Weight Guarantee ]
\label{lemma-softmw-regret}
Suppose Assumptions~\ref{assumption-theta} and \ref{assumption-delta} hold; then, running \Cref{softmw} guarantees
\begin{align}
    & \quad \sum_{t=1}^T \mathbb E\left[ \langle \mathbf Q_{t-1}, \mathbf S_t \odot \vec \theta_t \rangle - Q_{t-1,a_t}S_{t,a_t} \right] \nonumber \\
     & \le  \mathbb E\left\{ 9M(1 + C_V) T^{\frac 1 4 - \frac \delta 2} (3\ln T + \ln K) \sqrt{ 86M^2 K^6 T^{\frac 3 2} + \sum_{t=1}^{T} \lVert \mathbf Q_{t-1} \rVert_2^2} + 4M^2 \right\} \label{eq-lemma-softmw-regret}
\end{align}
for any time horizon length $T\ge 1$. Here $\{\vec \theta_t\}$ is the reference policy in Assumptions~\ref{assumption-theta} and \ref{assumption-delta}.
\end{lemma}
\begin{proof}
    See \Cref{proof-lemma-softmw-regret}.
\end{proof}
\Cref{lemma-softmw-regret} gives an upper-bound for $\mathbb E\left[\sum Q_{t-1,a_t}S_{t,a_t} - \sum \langle \mathbf Q_{t-1}, \vec \sigma_t \odot \vec \theta_t \rangle \right]$, which is closely related to the condition required by \Cref{prop-lyapunov-queue-length}. However, this upper-bound is not yet a quantity that depends solely on $T$; it still has a factor of $\sqrt{\mathbb E[\sum \lVert \mathbf Q_t \rVert_2^2]}$, depending on the actual queueing trajactory.
Therefore, we are unable to apply \Cref{prop-lyapunov-queue-length} directly to claim queue stability. Rather, we need to work with the $\sqrt{\mathbb E[\sum \lVert \mathbf Q_t \rVert_2^2]}$ factor, to convert it to the cumulative queue length $\mathbb E[\sum \lVert \mathbf Q_t \rVert_1]$, just as we did in \Cref{lemma-quad-lyapunov-theta} to convert $\mathbb E\left[\sum \langle \mathbf Q_{t-1}, \vec \sigma_t \odot \vec \theta_t - \vec \lambda_t\rangle\right]$ to queue lengths.

\subsubsection{Relate Regrets in $\sqrt{\mathbb E[\sum \lVert \mathbf Q_t \rVert_2^2]}$ to Queue Lengths $\mathbb E[\sum \lVert \mathbf Q_t \rVert_1]$ }
\ \\

Plugging \Cref{eq-lemma-softmw-regret} into \Cref{eq-quad-lyapunov-total} in \Cref{lemma-quad-lyapunov}, after further applying \Cref{lemma-quad-lyapunov-theta} and rearranging  terms, we get the following proposition, which offers an inequality connecting $\sqrt{\mathbb E[\sum \lVert \mathbf Q_t \rVert_2^2]}$  and  $\mathbb E[\sum \lVert \mathbf Q_t \rVert_1]$. 
\begin{proposition}
\label{prop-softmw-q12}
Given Assumptions~\ref{assumption-theta} and \ref{assumption-delta}, \Cref{softmw} gives us 
\begin{align}
   & \quad \mathbb E\left[ 
\sum_{t=1}^{\mathcal T_T} \lVert \mathbf Q_{t-1}\rVert_1  \right] \label{eq-prop-softmw-q12} \\
& \le \frac {(K+1)M^2 + 2C_W (KM^2 + \epsilon KM) } \epsilon \mathcal T_T + \frac{4M^2}{\epsilon} + \frac{g(\mathcal T_T)}\epsilon \cdot \sqrt{ 86M^2 K^6 \mathcal T_T^{\frac 3 2} + \mathbb E\left[\sum_{t=1}^{\mathcal T_T} \lVert \mathbf Q_{t-1} \rVert_2^2\right]} \nonumber \\
& \le \frac {(K+1)M^2 + 2C_W (KM^2 + \epsilon KM) } \epsilon \mathcal T_T + \frac{4M^2}{\epsilon} + \frac{\sqrt{86}MK^3\mathcal T_T^{\frac 3 4} g(\mathcal T_T)}{\epsilon} + \frac{g(\mathcal T_T)}\epsilon \cdot \sqrt{\mathbb E\left[\sum_{t=1}^{\mathcal T_T} \lVert \mathbf Q_{t-1} \rVert_2^2\right]} \nonumber 
\end{align}
for any $T\ge \max\{\frac {4} {C_W}, C_W\}$, where $\mathcal T_T$ is some constant no more than $2T$, and
\begin{equation*}
    g(T) = 9M(1 + C_V) T^{\frac 1 4 - \frac \delta 2} (3\ln T + \ln K) =\tilde \O(T^{\frac 1 4 - \frac \delta 2}).
\end{equation*}
\end{proposition}


Recall that all arrivals and departures are assumed to be bounded by a constant $M$. Therefore, each dimension of the queue length vectors $\{\mathbf Q_t\}$ is a sequence of non-negative numbers, where the difference between any two adjacent terms is within $\pm M$. We may then make use of the following lemma for such bounded-difference sequences.

\begin{lemma}
\label{lemma-bounded-diff-S}
Suppose $x_1=0$, $x_2,\ldots, x_n \ge 0$, $\lvert x_{i+1} - x_i \vert \le 1$ for all $1\le i < n$. Denote by $S = \sum_{i=1}^n x_i$; then we have 
\begin{equation*}
    \sum_{i=1}^n x_i^2 \le 4 S^{\frac 3 2}.
\end{equation*}
\end{lemma}
\begin{proof}
    See \Cref{proof-lemma-bounded-diff-S}.
\end{proof}
For our purposes, \Cref{lemma-bounded-diff-S} guarantees that 

\begin{equation}
\sum_{t=1}^T\lVert \mathbf Q_{t-1} \rVert_2^2 \le 4\sqrt M \sum_{i=1}^K \left(\sum_{t=1}^T Q_{t-1,i}\right)^{\frac 3 2} \le 4\sqrt M \left(\sum_{t=1}^T\lVert \mathbf Q_{t-1} \rVert_1\right)^{\frac 3 2}. \label{eq-lemma-bounded-diff-S-KM}
\end{equation}
Then, plugging \Cref{eq-lemma-bounded-diff-S-KM} into \Cref{eq-prop-softmw-q12}, we obtain the following inequality that depends \emph{entirely} on $\mathbb E\left[ 
\sum_{t=1}^{\mathcal T_T} \lVert \mathbf Q_{t-1}\rVert_1  \right]$: 
\begin{equation}
\label{eq-softmw-q34}
    \mathbb E\left[ 
\sum_{t=1}^{\mathcal T_T} \lVert \mathbf Q_{t-1}\rVert_1  \right] \le h(\mathcal T_T) + g(\mathcal T_T) 
\left(\mathbb E\left[\sum_{t=1}^{\mathcal T_T} \lVert \mathbf Q_{t-1}\rVert_1  \right]\right)^{\frac 3 4}
\end{equation}
where
\begin{equation*}
    g(T) = 18M^{\frac 5 4}(1 + C_V) T^{\frac 1 4 - \frac \delta 2} (3\ln T + \ln K) =\tilde \O\left(\frac{T^{\frac 1 4 - \frac \delta 2}}\epsilon\right),
\end{equation*}
\begin{equation*}
    h(T) = \frac {(K+1)M^2 + 2C_W (KM^2 + \epsilon KM) } \epsilon T + \tilde \O(T^{1 - \frac \delta 2}).
\end{equation*}
It remains to  solve \Cref{eq-softmw-q34}, in order to obtain an upper bound for $\mathbb E\left[ 
\sum_{t=1}^{\mathcal T_T} \lVert \mathbf Q_{t-1}\rVert_1  \right]$. To do so, we utilize the following lemma.
\begin{lemma}
\label{lemma-444}
    Let $y,f,g: \mathbb R_+ \rightarrow \left[1, \infty\right)$ be three non-decreasing functions. If
    \begin{equation*}
        y(x) \le f(x) + y(x)^{\frac 1 4}g(x)
    \end{equation*}
    for all $x\ge 0$, then we have
    \begin{equation*}
        y(x) \le \left( f(x)^{\frac 1 4} + g(x) \right)^4.
    \end{equation*}
\end{lemma}
\begin{proof}
    See \Cref{proof-lemma-444}.
\end{proof}
Finally, according to \Cref{lemma-444}, the solution of \Cref{eq-softmw-q34} gives us:
\begin{align*}
    \mathbb E\left[ 
\sum_{t=1}^{\mathcal T_T} \lVert \mathbf Q_{t-1}\rVert_1  \right] 
 & \le \left(h(\mathcal T_T)^{\frac 1 4} + g(\mathcal T_T)\right)^4 \\
 & \le \frac {(K+1)M^2 + 2C_W (KM^2 + \epsilon KM) } \epsilon \mathcal T_T + o(\mathcal T_T).
\end{align*}
Thus,
\begin{align*}
    \frac 1 T \mathbb E\left[ 
\sum_{t=1}^T \lVert \mathbf Q_{t-1}\rVert_1  \right] \le \frac 1 T \mathbb E\left[ 
\sum_{t=1}^{\mathcal T_T} \lVert \mathbf Q_{t-1}\rVert_1  \right] & \le \frac {(K+1)M^2 + 2C_W (KM^2 + \epsilon KM) } \epsilon \frac {\mathcal T_T} T + o(\mathcal T_T / T) \\
& \le  \frac {2(K+1)M^2 + 4C_W (KM^2 + \epsilon KM) } \epsilon + o(1)
\end{align*}
as desired.

\section{Taming Time-Homogeneous $\O(T^{1-\delta})$ Refernce Policy Total Variation }
\label{sec-ssmw}

In this section, we propose another novel algorithm capable of stabilizing our adversarial queueing system. Specifically, this algorithm is stable under a reference randomized policy with $O(T^{1-\delta})$ total variation, as long as that much total variation is to some extent ``evenly'' distributed throughout the infinite time horizon. This new condition is formalized as follows.

\begin{assumption}[Time-Homogeneous Reference Policy Stationarity]
\label{assumption-delta3}
For the sequence $\{\vec \theta_t\}$ in Assumption~\ref{assumption-theta}, there exist some $\delta > 0$ and $C_V > 0$ such that
\begin{equation*}
    \sum_{t=T_0 + 1}^{T_0 + T-1} \lVert \theta_{t+1} - \theta_t \rVert_1 \le C_V T^{1 - \delta}
\end{equation*}
for any $T_0 \ge 0$ and $T \ge 1$.
\end{assumption}

\noindent \textbf{Remark.} Assumption~\ref{assumption-delta3} can be viewed as a shift-invariant version of Assumption~\ref{assumption-delta}, with the degree of $T$ relaxed from $\frac 1 2 - \delta$ to $1 - \delta$. Roughly speaking, this assumption holds as long as there is only a finite number of time periods on which the reference policy variation accumulates at a linear rate. For example, if $\sum_{t=0}^T \lVert \theta_{t+1} - \theta_t \rVert_1 = \Theta(T^{1-\delta})$, then Assumption~\ref{assumption-delta3} is satisfied.


For problem instances where Assumptions~\ref{assumption-theta} and \ref{assumption-delta3} hold, we present a new algorithm to stabilize the system, namely \textbf{S}liding \textbf{S}oft\textbf{MW} (\texttt{SSMW}), which is detailed in \Cref{ssmw}.
\begin{figure}[H]
\centering
\includegraphics[width=0.95\textwidth]{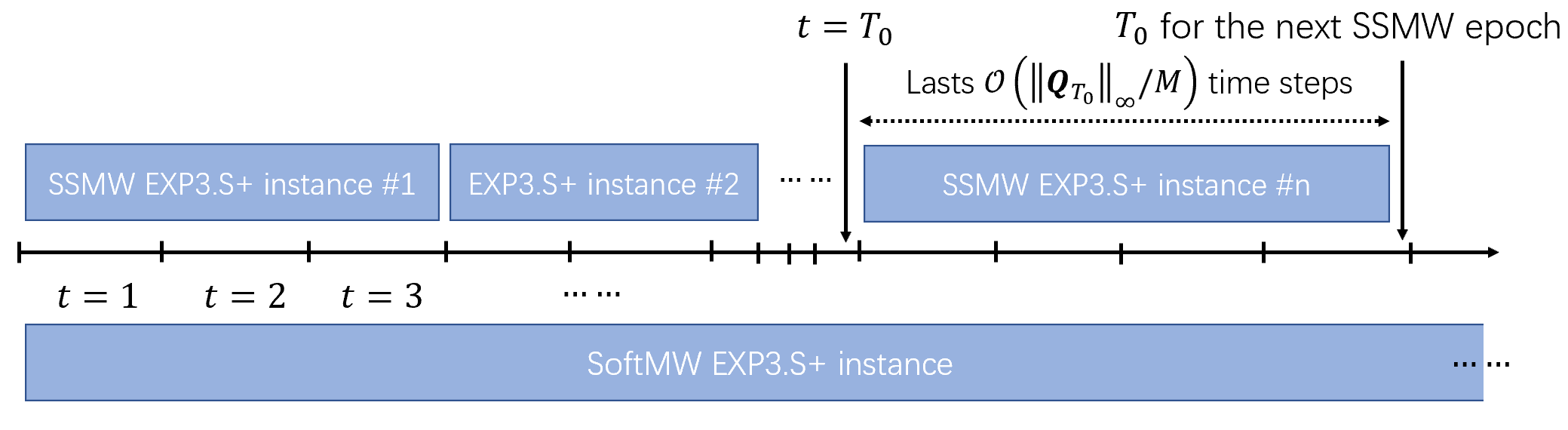}
\caption{Comparison between \texttt{SSMW} and \texttt{SoftMW} on \texttt{EXP3.S+} invocations}\label{fig:ssmw}
\end{figure}
\begin{algorithm2e}
\caption{\texttt{SSMW} (\textbf{S}liding-window \textbf{S}oft\textbf{MW})}
\label{ssmw}
\LinesNumbered
\DontPrintSemicolon
\KwIn{One-step arrival/service moment upper-bound parameter $M > 0$, Number of job types $K$, Problem instance smoothness parameter $\delta$ > 0}
\KwOut{A sequence of job types to serve $a_1, a_2,\ldots \in [K]$}
\BlankLine
\While{true}{
    $T_0 \gets \text{the latest time index $t$ at which we have made a new decision $a_t$}$\tcp*{for the first iteration, we should have $T_0 = 0$}
    $m \gets \max\left\{\lceil \frac {\lVert \mathbf Q_{T_0}\rVert_\infty} {2M} \rceil, 1\right\}$ \label{line-ssmw-m}\;
    Run a fresh \texttt{EXP3.S+} instance for $m$ time steps with the following configuration (below $\tau$ denotes the time index \textit{within} the epoch of length $m$, 1-based): \;
    \Indp 
    $\beta = m^{-2} / K$ \;
    $\mathbf x_1$ can be any element in $\Delta^{[K], \beta} \triangleq \{\mathbf x \in \Delta^{[K]} : \mathbf x_i \ge \beta
\hspace{1em} \forall i \in [K]\}$\;
    $\eta_\tau = \left( 6 M^2K m^{1 + \frac \delta 2} \right)^{-1}$ \;
    $\mathbf e_\tau = \mathbf Q_{T_0 + \tau - 1} / \lVert \mathbf Q_{T_0 + \tau - 1} \rVert_1$\;
    $\gamma_\tau = M\eta_\tau \lVert \mathbf Q_{T_0 + \tau - 1}\rVert_1 = \frac 1 6 K^{-1}M^{-1}m^{-1-\frac \delta 2}\lVert\mathbf Q_{T_0 + \tau - 1}\rVert_1$ \;
    \Indm Take a new action decision output from the current \texttt{EXP3.S+} instance, serve this type of jobs (recall we are at the $(T_0 + \tau)$-th time step of the whole time horizon), regard $Q_{T_0 + \tau - 1,a_{T_0 + \tau}}S_{T_0 + \tau,a_{T_0 + \tau}}$ as a new feedback $g_{\tau,a_\tau}$ and feed it into \texttt{EXP3.S+} 
}
\end{algorithm2e}
Compared to \texttt{SoftMW} (\Cref{softmw}), \texttt{SSMW} (\Cref{ssmw}) does not use historical queue lengths at the beginning to tune the \texttt{EXP3.S+} learning rates. Instead, \texttt{SSMW} starts with new \texttt{EXP3.S+} instances of lengths proportional to the current queue lengths (\Cref{line-ssmw-m}). As a result, \texttt{SSMW} initiates many more \texttt{EXP3.S+} instances throughout its execution, though each \texttt{EXP3.S+} period is likely to be short (demonstrated in Fig. \ref{fig:ssmw}). In this sense, \texttt{SSMW} is more similar to \texttt{MaxWeight}, since \texttt{MaxWeight} always uses the current queue length vector for making new decisions, and disregards how the system arrived at the current state. \Cref{thm:SSMW-stability} gives the queue stability result for \texttt{SSMW}.


\begin{theorem}
\label{thm-ssmw-stab}
    For problem instances satisfying Assumptions~\ref{assumption-theta} and \ref{assumption-delta3}, \texttt{SSMW} (\Cref{ssmw}) guarantees
    \begin{equation*}
        \frac 1 T \mathbb E\left[ 
\sum_{t=1}^{T} \lVert \mathbf Q_t\rVert_1  \right] \le \left[3KM^2m_0 + \frac{(K+1)M^2} 2 + (KM^2 + \epsilon KM)C_W + 6M^2\right] \cdot \frac {10} \epsilon
    \end{equation*}
for any time horizon of length $T \ge \frac 4 {C_W} + C_W$. In particular, the system is stable. Here $m_0$ is defined as
\begin{equation*}
    m_0 \triangleq \inf \left\{ m : m\ge 2,  f(m') \le \frac \epsilon 2 \space \forall m'\ge m\right\} \le \left((1+C_V)MK\ln K \epsilon^{-1}\right)^{\O(1/\delta)}
\end{equation*}
where
\begin{equation*}
    f(m) = 88(1+C_V)M Km^{- \frac \delta 2}(2\ln m + \ln K).
\end{equation*}
\label{thm:SSMW-stability}
\end{theorem}

\noindent \textbf{Remark.} Compared to the $\O(\epsilon^{-1})$ queue length bound of \texttt{SoftMW} (\Cref{thm-softmw-stab}), \Cref{thm-ssmw-stab} only gives an $\epsilon^{\O(1/\delta)}$ queue length guarantee. Nevertheless, the simulation results in \Cref{sec-apdx-experiments} show that the empirical performance of \texttt{SSMW} is comparable or even better than that of \texttt{SoftMW}.


\subsection{Queue Stability Analysis Outline for \texttt{SSMW}}
\label{sec-ssmw-analysis}
In this section, we provide a proof outline for Theorem \ref{thm:SSMW-stability}. 
First, applying \Cref{thm-exp3s}, we can obtain the result in \Cref{lemma-ssmw-epoch} regarding the performance of \texttt{SSMW} compared to the corresponding reference poilcy in each \texttt{EXP3.S+} instance of \texttt{SSMW}.
\begin{lemma}
\label{lemma-ssmw-epoch}
Suppose Assumptions~\ref{assumption-theta} and \ref{assumption-delta3} hold, then, let $T_0$ be some time step at which we start a new \texttt{EXP3.S+} instance of length $m$ in \Cref{ssmw}; then we have
\begin{align}
    & \quad \mathbbm 1[T_0\text{ ends an \texttt{EXP3.S+} instance, and the new \texttt{EXP3.S+} instance is of length }m] \nonumber \\
    &\quad \cdot \sum_{t=1}^m \mathbb E\left[\left. \langle \mathbf Q_{T_0 + t - 1} \odot \mathbf S_{T_0 + t}, \vec \theta_{T_0 + t} \rangle - Q_{T_0 + t - 1,a_{T_0 + t}}S_{T_0 + t,a_{T_0 + t}} \right\vert \mathcal F_{T_0} \right] \nonumber \\
& \le  6\left(1 + C_V\right) M^2Km^{2-\frac \delta 2} \cdot \left( 2\ln m + \ln K \right) + K^{-1}m^{-1-\frac \delta 2}\mathbb E\left[\left.\sum_{\tau=1}^m \lVert 
\mathbf Q_{T_0 + \tau-1}\rVert_2^2\right\rvert \mathcal F_{T_0}\right] + 6M^2. \label{eq-lemma-ssmw-epoch}
\end{align}
\end{lemma}
\begin{proof}
    See \Cref{proof-lemma-ssmw-epoch}.
\end{proof}

Similar to what we have done after obtaining \Cref{lemma-softmw-regret} when analysing \texttt{SoftMW}, we will relate this regret upper-bound in \Cref{lemma-ssmw-epoch} to the actual cumulative queue length $\mathbb E[\sum \lVert \mathbf Q_t \rVert_1]$. 
To achieve this, we  need several lemmas on sums of bounded-increment sequences.

\begin{lemma}
\label{lemma-ssmw-sum-l12norm}
For any $T_0 \ge 1$, suppose $M\ge 0$, $\lVert\mathbf Q_{T_0}\rVert_\infty \ge 4M$, then for any $T_0 + 1 \le t \le T_0 + \frac {\lVert \mathbf Q_{T_0}\rVert_\infty} {2M} + 1$,
we have
\begin{equation*}
    \frac 1 4 \lVert \mathbf Q_{T_0}\rVert_\infty \le \frac 1 2 \lVert \mathbf Q_{T_0}\rVert_\infty - M \le \lVert\mathbf Q_{t-1}\rVert_\infty \le \frac 3 2\lVert\mathbf Q_{T_0}\rVert_\infty + M \le 2 \lVert \mathbf Q_{T_0}\rVert_\infty.
\end{equation*}
Let $T = \lceil \frac {\lVert \mathbf Q_{T_0}\rVert_\infty} {2M} \rceil$; then we have
\begin{equation*}
    \frac 1 4 MT \le \lVert\mathbf Q_{T_0 + t-1}\rVert_\infty \le 4MT
\end{equation*}
for any $1 \le t \le T$. Moreover, 
\begin{equation*}
    \frac 1 {16} M^2T^3 \le \sum_{t=T_0 + 1}^{T_0 + T} \lVert \mathbf Q_{t-1} \rVert_2^2 \le 16KM^2T^3,
\end{equation*}
\begin{equation*}
    \frac 1 4 MT^2\le \sum_{t=T_0 + 1}^{T_0 + T} \lVert \mathbf Q_{t-1} \rVert_1 \le 4KMT^2.
\end{equation*}
\end{lemma}
\begin{proof}
    See \Cref{proof-lemma-ssmw-sum-l12norm}.
\end{proof}

We can now apply \Cref{lemma-ssmw-sum-l12norm} to relate the regret upper-bound in \Cref{lemma-ssmw-epoch} to the queue-length sums $\sum \lVert \mathbf Q_t \rVert_1 )$ to obtain the result in \Cref{lemma-ssmw-epoch-l1}.


\begin{lemma}
\label{lemma-ssmw-epoch-l1}
Suppose Assumptions~\ref{assumption-theta} and \ref{assumption-delta3} hold. Let $T_0$ be a time step at which we start a new \texttt{EXP3.S+} instance of length $m$ in \Cref{ssmw}. Then, we have   
\begin{align*}
    & \mathbbm 1\left[T_0\text{ ends an \texttt{EXP3.S+} instance, and the new \texttt{EXP3.S+} instance is of length }m = \left\lceil \frac {\lVert \mathbf Q_{T_0}\rVert_\infty} {2M}  \right\rceil, m\ge 2\right] \\
    &\cdot \sum_{t=1}^m \mathbb E\left[\left. \langle \mathbf Q_{T_0 + t - 1} \odot \mathbf S_{T_0 + t}, \vec \theta_{T_0 + t} \rangle - Q_{T_0 + t - 1,a_{T_0 + t}}S_{T_0 + t,a_{T_0 + t}} \right\vert \mathcal F_{T_0} \right] \\
& \le  6M^2 + \mathbb E\left[\left. f(m)\cdot \sum_{t=1}^m \lVert \mathbf Q_{T_0 + t - 1} \rVert_1 \right\vert \mathcal F_{T_0} \right],
\end{align*}
where
\begin{equation*}
    f(m) = 88(1+C_V)M Km^{- \frac \delta 2}(2\ln m + \ln K).
\end{equation*}
\end{lemma}
\begin{proof}
    See \Cref{proof-lemma-ssmw-epoch-l1}.
\end{proof}

Compared to \texttt{SoftMW}, an \texttt{SSMW} execution contains multiple \texttt{EXP3.S+} executions, and their starting times and ending times are all stochastic quantities. To rigorously handle these stochastic \texttt{EXP3.S+} epochs in the analysis, we will introduce a few more notations. Denote by $\tau_i$ ($i \ge 0$) the time at which the $i$-th \texttt{EXP3.S+} instance finishes. Then, $\tau_0 = 0$ and $\{\tau_i\}$ is a sequence of non-decreasing $\{\mathcal F_t\}$-adapted stopping-times. Furthermore, each $\tau_{i+1}$ is $\mathcal F_{\tau_i}$-measurable. Fix any $T\ge 1$ and define
\begin{equation*}
    \tau'_i \triangleq \begin{cases}
        0 & \text{if }i = 0 \\
        \tau_i & \text{if } i > 0 \text{ and } \tau'_{i-1} < T \\
        \tau'_{i-1} & \text{otherwise}
    \end{cases},
\end{equation*}
i.e., $\tau'_i$ can be regarded as the epoch end time $\tau_i$, but truncated at $T$, and it will be more convenient than $\tau_i$ when we consider the cumulative regret up to time $T$. Then, $\{\tau'_i\}$ is also a sequence of non-decreasing $\{\mathcal F_t\}$-adapted stopping-times, each $\tau'_{i+1}$ is $\mathcal F_{\tau'_i}$-measurable, and $\tau'_{i+1} = \tau'_i$ if any only if $\tau'_i \ge T$. Thus, we can restate \Cref{lemma-ssmw-epoch-l1} as \Cref{lemma-ssmw-epoch-l1-stopping-time}.
\begin{lemma}
\label{lemma-ssmw-epoch-l1-stopping-time}
    Suppose Assumptions~\ref{assumption-theta} and \ref{assumption-delta3} hold.  Then, we have
    \begin{align*}
        & \quad\mathbbm 1\left[ \lVert \mathbf Q_{\tau'_i}\rVert_\infty \ge 4M\right] \sum_{t=1}^{\tau'_{i+1} - \tau'_i} \mathbb E\left[\left. \langle \mathbf Q_{\tau'_i + t - 1} \odot \mathbf S_{\tau'_i + t}, \vec \theta_{\tau'_i + t} \rangle - Q_{\tau'_i + t - 1,a_{\tau'_i + t}}S_{\tau'_i + t,a_{\tau'_i + t}}\right\vert \mathcal F_{\tau'_i}\right] \\
        & \le h(\tau'_{i+1} - \tau'_i) + g(\tau'_{i+1} - \tau'_i) \cdot \sum_{t=1}^{\tau'_{i+1} - \tau'_i} \mathbb E\left[\left. 
\lVert \mathbf Q_{\tau'_i + t - 1} \rVert_1 \right\vert \mathcal F_{\tau'_i}\right]
    \end{align*}
for any $i\ge 0$, where 
\begin{align}
    g(m) & = 88(1+C_V)M Km^{- \frac \delta 2}(2\ln m + \ln K),\label{eq:lemma65-f} \\
    h(m) & = \mathbbm 1[m > 0]\cdot 6M^2.\label{eq:lemma65-h}
\end{align}
\end{lemma}

Compared to \Cref{lemma-ssmw-epoch-l1}, now \Cref{lemma-ssmw-epoch-l1-stopping-time} allows us to sum the regret bounds for each epoch freely without caring about the subtleties caused by the finite time-horizon length $T$. 

Next, we will make use of Lemma \ref{lemma-ssmw-epoch-l1-stopping-time} to bound $\mathbb E[\sum_{t=1}^{\mathcal T_1} \lVert\mathbf Q_{t - 1}\rVert_1]$. 
To begin, fix some $T \ge 1$, let $\mathcal T_0 \triangleq \sup \{\tau_i : i\ge 0, \tau_i < T\}$, and $\mathcal T_1 \triangleq \inf \{\tau_i : i\ge 0, \tau_i \ge T\}$. We see that $\mathcal T_0$ and $\mathcal T_1$ are both $\{\mathcal F_t\}$-adapted stopping-times, and $\mathcal T_1$ is $\mathcal F_{\mathcal T_0}$-measurable. Note that $\mathcal T_0 < T \le \mathcal T_1$. Furthermore, since 
$\mathcal T_1 - \mathcal T_0$ is the length of the last epoch in the first $T$ time steps, we have $\mathcal T_1 - \mathcal T_0 \le \frac {\lVert \mathbf Q_{\mathcal T_0}\rVert_\infty}{2M} + 1  \le \frac{\mathcal T_0 \cdot M}{2M} + 1 = \frac{\mathcal T_0} 2 + 1 \le \frac{\mathcal T_0} 2 + T$, thus we can see $\mathcal T_1 \le \frac 5 2 T$.

In the remainder of this section, we combine \Cref{lemma-quad-lyapunov} and \Cref{lemma-ssmw-epoch-l1-stopping-time} to bound $\mathbb E[\sum_{t=1}^{\mathcal T_1} \lVert\mathbf Q_{t - 1}\rVert_1]$ in $\O(\mathbb E[\mathcal T_1]) = \O(T)$ in order to conclude that $\mathbb E[\sum_{t=1}^T \lVert\mathbf Q_{t - 1}\rVert_1]$ is also $\O(T)$.

To this end, recall that $\epsilon > 0$ is the lower-bound of the ``average advantage of departure against arrival'' of the reference policy $\{\vec \theta_t\}$ in Assumption~\ref{assumption-theta}. Define
\begin{equation*}
    m_0 \triangleq \inf \left\{ m : m\ge 2,  g(m') \le \frac \epsilon 2 \space \forall m'\ge m\right\}.
\end{equation*}
Then $m_0$ is a constant that only depends on $\delta$ and $\epsilon$; in fact, it solves to 
\begin{equation*}
    m_0 \le \left((1+C_V)MK\ln K \epsilon^{-1}\right)^{\O(1/\delta)}.
\end{equation*}

By considering whether each epoch length $\tau'_{i+1} - \tau'_i$ is greater than $m_0$ or not, we conclude from \Cref{lemma-ssmw-epoch-l1-stopping-time} that for all $i\ge 0$, 

    \begin{align}
        & \quad\sum_{t=1}^{\tau'_{i+1} - \tau'_i} \mathbb E\left[\left. \langle \mathbf Q_{\tau'_i + t - 1} \odot \mathbf S_{\tau'_i + t}, \vec \theta_{\tau'_i + t} \rangle - Q_{\tau'_i + t - 1,a_{\tau'_i + t}}S_{\tau'_i + t,a_{\tau'_i + t}}\right\vert \mathcal F_{\tau'_i}\right] \nonumber \\
        & \le \underbrace{h(\tau'_{i+1} - \tau'_i) + \frac \epsilon 2 \sum_{t=1}^{\tau'_{i+1} - \tau'_i} \mathbb E\left[\left. \lVert \mathbf Q_{\tau'_i + t - 1} \rVert_1 \right\vert \mathcal F_{\tau'_i}\right]}_{\text{when }\tau'_{i+1} - \tau'_i > m_0\text{, apply \Cref{lemma-ssmw-epoch-l1-stopping-time}}} + \underbrace{\mathbbm 1[\tau'_{i+1} - \tau'_i \le m_0]\sum_{t=1}^{\tau'_{i+1} - \tau'_i} \mathbb E\left[\left. \langle \mathbf Q_{\tau'_i + t - 1} \odot \mathbf S_{\tau'_i + t}, \vec \theta_{\tau'_i + t} \rangle \right\vert \mathcal F_{\tau'_i}\right]}_{\text{when }\tau'_{i+1} - \tau'_i \le m_0\text{, simply drop the minus-signed term}} \nonumber \\
        & \stackrel{(a)}\le h(\tau'_{i+1} - \tau'_i) + \frac \epsilon 2 \sum_{t=1}^{\tau'_{i+1} - \tau'_i} \mathbb E\left[\left. \lVert \mathbf Q_{\tau'_i + t - 1} \rVert_1 \right\vert \mathcal F_{\tau'_i}\right] + \mathbbm 1[\tau'_{i+1} - \tau'_i \le m_0] M\sum_{t=1}^{\tau'_{i+1} - \tau'_i} \mathbb E\left[\left. \lVert\mathbf Q_{\tau'_i + t - 1}\rVert_1\right\vert \mathcal F_{\tau'_i}\right] \nonumber \\
        & \stackrel{(b)}\le h(\tau'_{i+1} - \tau'_i) + \frac \epsilon 2 \sum_{t=1}^{\tau'_{i+1} - \tau'_i} \mathbb E\left[\left. \lVert \mathbf Q_{\tau'_i + t - 1} \rVert_1 \right\vert \mathcal F_{\tau'_i}\right] + \mathbbm 1[\tau'_{i+1} - \tau'_i \le m_0] KM\sum_{t=1}^{\tau'_{i+1} - \tau'_i} \mathbb E\left[\left .\lVert \mathbf Q_{\tau'_i} \rVert_\infty + Mm_0 \right\vert \mathcal F_{\tau'_i}\right] \nonumber \\
        & \stackrel{(c)}\le h(\tau'_{i+1} - \tau'_i) + \frac \epsilon 2 \sum_{t=1}^{\tau'_{i+1} - \tau'_i} \mathbb E\left[\left. \lVert \mathbf Q_{\tau'_i + t - 1} \rVert_1 \right\vert \mathcal F_{\tau'_i}\right] + \mathbbm 1[\tau'_{i+1} - \tau'_i \le m_0] KM\sum_{t=1}^{\tau'_{i+1} - \tau'_i} \mathbb E\left[\left .3Mm_0 \right\vert \mathcal F_{\tau'_i}\right] \nonumber \\
        & \le h(\tau'_{i+1} - \tau'_i) + \frac \epsilon 2 \sum_{t=1}^{\tau'_{i+1} - \tau'_i} \mathbb E\left[\left. \lVert \mathbf Q_{\tau'_i + t - 1} \rVert_1 \right\vert \mathcal F_{\tau'_i}\right] + 3KM^2m_0(\tau'_{i+1} - \tau'_i) \label{eq-ssmw-epoch-regret-stopping-time-m0}
    \end{align} 
Here $h(m)$ is the function in \eqref{eq:lemma65-h}. In steps $(a)$ and $(b)$ we simply leverage the assumption of bounded queue length increments. Step $(c)$ is due to $\lVert \mathbf Q_{\tau'_i} \rVert_\infty\le 2M(\tau'_{i+1} - \tau'_i) \le 2Mm_0$ as long as $\tau'_{i+1} - \tau'_i > 0$, i.e., $i$ is not the index of an epoch after the first $T$ time slots. 


Summing \Cref{eq-ssmw-epoch-regret-stopping-time-m0} over all $i\ge 0$ and then taking total expectations, we obtain 
\begin{align*}
        \mathbb E\left[\sum_{t=1}^{\mathcal T_1}  \langle \mathbf Q_{t - 1} \odot \mathbf S_t, \vec \theta_t \rangle - Q_{t - 1,a_t}S_{t,a_t} \right]  \le \frac \epsilon 2 \mathbb E\left[ \sum_{t=1}^{\mathcal T_1} \lVert \mathbf Q_{t - 1} \rVert_1 \right] + 3KM^2m_0 \mathbb E\left[\mathcal T_1\right] + \mathbb E\left[\sum_{i=0}^\infty h(\tau'_{i+1} - \tau'_i) \right].
\end{align*}
From \eqref{eq:lemma65-h}, we see that in any sample path,  $\sum_{i=0}^\infty h(\tau'_{i+1} - \tau'_i) \le 6M^2\mathcal T_1$. Therefore, 
\begin{align}
        \mathbb E\left[\sum_{t=1}^{\mathcal T_1}  \langle \mathbf Q_{t - 1} \odot \mathbf S_t, \vec \theta_t \rangle - Q_{t - 1,a_t}S_{t,a_t} \right] & \le \frac \epsilon 2 \mathbb E\left[ \sum_{t=1}^{\mathcal T_1} \lVert \mathbf Q_{t - 1} \rVert_1 \right] + (3KM^2m_0 + 6M^2) \mathbb E\left[\mathcal T_1\right]. \label{eq-ssmw-analysis-1}
\end{align}
According to \Cref{lemma-quad-lyapunov-theta}, we can also find a constant $\mathcal T_2$ depending on $\mathcal T_1$, such that  $\mathcal T_2 \le \mathcal T_1 + \sqrt{\frac {\mathcal T_1} {C_W}} + 1$, and
\begin{equation}
     -\mathbb E\left[\sum_{t=1}^{\mathcal T_2} \langle \mathbf Q_{t-1}, \vec \sigma_t \odot \vec \theta_t - \vec \lambda_t\rangle\right] \le  -\epsilon \mathbb E\left[ \sum_{t=1}^{\mathcal T_1} \lVert \mathbf Q_{t - 1} \rVert_1 \right] + (KM^2 + \epsilon KM)C_W \mathbb E[\mathcal T_2]. \label{eq-ssmw-analysis-2}
\end{equation}
If $T \ge \frac 4 {C_W} + C_W$, we have $\mathcal T_1 \ge \max\{C_W, \frac 4 {C_W}\}$; hence $\mathcal T_2 \le 2\mathcal T_1 \le 5T$. Also, \Cref{lemma-quad-lyapunov} guarantees that
\begin{equation}
    \mathbb E\left[\sum_{t=1}^{\mathcal T_2} Q_{t-1,a_t}S_{t,a_t} -\langle \mathbf Q_{t-1}, \vec \lambda_t\rangle\right] \le \frac{(K+1)M^2\mathbb E[\mathcal T_2]} 2. \label{eq-ssmw-analysis-3}
\end{equation}

Combining \Cref{eq-ssmw-analysis-1,eq-ssmw-analysis-2,eq-ssmw-analysis-3} by simply summing them up, we obtain 
\begin{align*}
    \frac \epsilon 2 \mathbb E\left[ \sum_{t=1}^{\mathcal T_1} \lVert \mathbf Q_{t - 1} \rVert_1 \right] & \le \left[\frac{(K+1)M^2} 2 + (KM^2 + \epsilon KM)C_W + 3KM^2m_0 + 6M^2\right] \mathbb E[\mathcal T_2] \\
    & \le \left[\frac{(K+1)M^2} 2 + (KM^2 + \epsilon KM)C_W + 3KM^2m_0 + 6M^2\right] \cdot 5T.
\end{align*}

Thus, when $T \ge \frac 4 {C_W} + C_W$, we have
\begin{equation*}
     \frac 1 T \mathbb E\left[ \sum_{t=1}^T \lVert \mathbf Q_{t - 1} \rVert_1 \right] \le \frac 1 T \mathbb E\left[ \sum_{t=1}^{\mathcal T_1} \lVert \mathbf Q_{t - 1} \rVert_1 \right] \le \left[3KM^2m_0 + \frac{(K+1)M^2} 2 + (KM^2 + \epsilon KM)C_W + 6M^2\right] \cdot \frac {10} \epsilon, 
\end{equation*}
which completes the proof of Theorem \ref{thm-ssmw-stab}.

\section{Relaxing the Boundedness Assumption for Queue Increments}
\label{sec-moment}


In this section, we further extend \texttt{SoftMW} and \texttt{SSMW} to 
settings where the queue lengths increments (individual arrivals and departures) are not necessarily bounded in a known range, but have bounded moments. 
Formally, we will replace the bounded-arrival-and-service assumption in our problem setting by the following new assumption.
\begin{assumption}[Queue length increments with bounded moments]
\label{assumption-moment}
    The arrival and service distributions $\{\mathcal A_1, \mathcal A_2, \ldots\}$ and $\{\mathcal S_1, \mathcal S_2, \ldots\}$ are supported on $\mathbb R_+^K$, but there exists constants $\alpha \ge 2$ and $M > 0$, both known to the system scheduler before-hand, such that
    \begin{equation*}
        \mathbb E\left[\left.A_{t,i}^\alpha \right\rvert \mathcal F_{t-1}\right], \mathbb E\left[\left.A_{t,i}^\alpha \right\rvert \mathcal F_{t-1}\right] \le M^\alpha
    \end{equation*}
    for all $t \ge 1$ and $i\in [K]$. As an immediate implication, we also have for all $t \ge 1$ and $i\in [K]$ that 
\begin{equation*}
     \mathbb E\left[\left.\lvert Q_{t,i} - Q_{t-1,i}\rvert^\alpha \right\rvert \mathcal F_{t-1}\right] \le 2M^\alpha.
\end{equation*}
\end{assumption}
Note that while this assumption is often not difficult in the standard Lyapunov analysis \cite{neely2010stochastic}, it poses a new challenge in the learning-augmented control analysis, especially when the algorithm uses UCB bonuses (e.g., \cite{choudhury2021job,hsu2022integrated,krishnasamy2021learning,yang2022maxweight}), primarily due to the impact it brings in estimation.\footnote{For example, UCB-based estimators usually require the distributions to be sub-Gaussian, which is a much more restricted assumption compared to our Assumption~\ref{assumption-moment}.}

Below, we present new variants of \texttt{SoftMW} and \texttt{SSMW} that are capable of stabilizing the system under Assumption~\ref{assumption-moment}. We explain the high-level design ideas in \Cref{sec-adpx-moment-idea}, and we put detailed queue stability proofs in \Cref{sec-apdx-softmw-moment,sec-apdx-ssmw-moment}.

\begin{algorithm2e}
\caption{\texttt{SoftMW} for queue-length increments with bounded moments (\texttt{SoftMW+})}
\label{softmw-moment}
\LinesNumbered
\DontPrintSemicolon
\KwIn{Queue-length increment moment upper-bound parameter $M > 0$, $\alpha > 14$, Number of job types $K$, Problem instance smoothness parameter $0 < \delta \le \frac 1 2$}
\KwOut{A sequence of job types to serve $a_1, a_2,\ldots \in [K]$}
\BlankLine
$L_0 \gets M$\;
Initialize an extended \texttt{EXP3.S+} instance.\;
\For{$t=1,2,\ldots$}{
    Pick the following parameters of \texttt{EXP3.S+} for time slot $t$: \;
    \Indp$\beta_t \gets t^{-4}/ K$\;
    $\eta_t = \left( t^{-\left(\frac 1 4 - \frac \delta 2\right)} L_{t-1} \sqrt{ 86L_{t-1}^2 K^6 t^{\frac 3 2} + \sum_{s=0}^{t -1} \lVert \mathbf Q_s \rVert_2^2} \right)^{-1}$ \;
    $\mathbf e_t = \mathbf Q_{t - 1} / \lVert \mathbf Q_{t - 1} \rVert_1$\;
    $\gamma_t = Mt^{\frac \delta 4}\eta_t \lVert \mathbf Q_{t - 1}\rVert_1 = t^{\frac \delta 4} M L_{t-1}^{-1 } \lVert \mathbf Q_{t - 1}\rVert_1 \left( t^{-(\frac 1 4 - \frac \delta 2)} \sqrt{ 86L_{t-1}^2 K^6 t^{\frac 3 2} + \sum_{s=0}^{t-1} \lVert \mathbf Q_s \rVert_2^2} \right)^{-1}$ \label{line-softmw-moment-exploration-rate} \; 
    \Indm Take a new action decision output $a_t$ from \texttt{EXP3.S+}, serve the $a_t$-th queue, receive feedback $S_{t,a_t}$\; 
    $S'_t \triangleq \begin{cases}
        S_{t,a_t} & \text{if } S_{t,a_t} \le Mt^{\frac \delta 4 } \\
        0 & \text{otherwise}
    \end{cases}$ \label{line-softmw-moment-clip}\;
    Regard $Q_{t - 1,a_t}S'_t$ as a new feedback $g_{t,a_t}$ and feed it into the current \texttt{EXP3.S+} instance\;
    $L_t \gets \max\left\{L_{t-1}, \lVert \mathbf Q_t - \mathbf Q_{t-1} \rVert_\infty\right\}$
}
\end{algorithm2e}

\texttt{SoftMW+} (\Cref{softmw-moment}) is a generalized version of \texttt{SoftMW} for Assumption~\ref{assumption-moment}.  Compared to \texttt{SoftMW} (\Cref{softmw}), \texttt{SoftMW+} (\Cref{softmw-moment}) explicitly tracks $L_t$ (the maximum queue length increment we have encountered up to time $t$) and most occurrences of the queue length increment that upper-bound $M$ in \Cref{softmw} are replaced by $L_t$ in the new algorithm. We also increase the explicit exploration rate by $t^{\frac \delta 4}$ times (\Cref{line-softmw-moment-exploration-rate}) and clip the service feedback before we feed it into \texttt{EXP3.S+} (\Cref{line-softmw-moment-clip}).

\Cref{thm-softmw-moment-stab} below gives the corresponding average queue length bound on any finite time-horizon.

\begin{theorem}
\label{thm-softmw-moment-stab}
    For problem instances satisfying Assumptions~\ref{assumption-theta}, \ref{assumption-delta} and \ref{assumption-moment}, \texttt{SoftMW+} (\Cref{softmw-moment}) guarantees
    \begin{equation*}
        \frac 1 T \mathbb E\left[ 
\sum_{t=1}^{T} \lVert \mathbf Q_t\rVert_1  \right] \le \frac {2(K+1)M^2 + 4C_W (KM^2 + \epsilon KM) } \epsilon + o(1)
    \end{equation*}
    as long as $\delta \cdot \alpha > 7$. In particular, the system is stable.
\end{theorem}
\begin{proof}
    See \Cref{sec-apdx-softmw-moment}.
\end{proof}

We next present \texttt{SSMW+} (\Cref{ssmw-moment}), the generalized version of \texttt{SSMW}. 
Compared to \texttt{SoftMW+} (\Cref{softmw-moment}), \texttt{SSMW+}  does not require the actual value of $\alpha$ in Assumption~\ref{assumption-moment}, and can stabilize any problem instance where new arrivals and service have bounded second moments (any $\alpha \ge 2$ in Assumption~\ref{assumption-moment} also implies this). 
By contrast, \Cref{softmw-moment} needs to know $\alpha$ beforehand, and requires that the product of $\delta$ and $\alpha$ is not too small.

\begin{algorithm2e}
\caption{\texttt{SSMW} for queue-length increments with bounded moments (\texttt{SSMW+})}
\label{ssmw-moment}
\LinesNumbered
\DontPrintSemicolon
\KwIn{Queue-length increment moment upper-bound parameter $M > 0$, Number of job types $K$, Problem instance smoothness parameter $\delta > 0$}
\KwOut{A sequence of job types to serve $a_1, a_2,\ldots \in [K]$}
\BlankLine
\While{true}{
    $T_0 \gets \text{the latest time index $t$ at which we have made a new decision $a_t$}$\tcp*{for the first iteration, we should have $T_0 = 0$}
    $m \gets \max\left\{\left\lceil \frac {\lVert \mathbf Q_{T_0}\rVert_\infty} {2M} \right\rceil, 1\right\}$ \;
    Run a fresh \texttt{EXP3.S+} instance for $m$ time steps with the following configuration (below $\tau$ denotes the time index \textit{within} the epoch of length $m$, 1-based): \;
    \Indp 
    $\beta = m^{-3} / K$ \;
    $\mathbf x_1$ can be any element in $\Delta^{[K], \beta} \triangleq \{\mathbf x \in \Delta^{[K]} : \mathbf x_i \ge \beta
\hspace{1em} \forall i \in [K]\}$\;
    $\eta_\tau = \left( 4 M^3K m^{1 + \frac 2 3 \delta} \right)^{-1}$ \;
    $\mathbf e_\tau = \mathbf 1 / K$ \;
    $\gamma_\tau = m^{\frac \delta 3}KM\eta_\tau \lVert \mathbf Q_{T_0}\rVert_\infty = \frac 1 4 M^{-2}m^{-1-\frac \delta 3}\lVert\mathbf Q_{T_0}\rVert_\infty$ \;
    \Indm Take a new action decision output $a_t$ from \texttt{EXP3.S+}, serve the $a_t$-th queue, receive feedback $S_{T_0 + \tau,a_{T_0 + \tau}}$\; 
    $g_{T_0 + \tau} \triangleq \begin{cases}
        Q_{T_0 + \tau - 1, a_{T_0 + \tau}}S_{T_0 + \tau,a_{T_0 + \tau}} & \text{if } Q_{T_0 + \tau - 1, a_{T_0 + \tau}}S_{T_0 + \tau,a_{T_0 + \tau}} \le m^{\frac \delta 3} M Q_{T_0,a_{T_0 + \tau}} \\
        0 & \text{otherwise}
    \end{cases}$\;
    Regard $g_{T_0 + \tau}$ as a new feedback and feed it into the \texttt{EXP3.S+} instance
}
\end{algorithm2e}


\Cref{thm-ssmw-moment-stab} gives the corresponding average queue length bound on any finite time-horizon.
\begin{theorem}
\label{thm-ssmw-moment-stab}
    For problem instances satisfying Assumptions~\ref{assumption-theta}, \ref{assumption-delta3} and \ref{assumption-moment}, \texttt{SSMW+} (\Cref{ssmw-moment}) guarantees
    \begin{equation*}
        \frac 1 T \mathbb E\left[ 
\sum_{t=1}^{T} \lVert \mathbf Q_t\rVert_1  \right] \le \left[3KM^2m_0 + \frac{(K+1)M^2} 2 + (KM^2 + \epsilon KM)C_W + 4M^2\right] \cdot \frac {10} \epsilon
    \end{equation*}
for any time horizon length $T \ge  \frac 4 {C_W} + C_W$. In particular, the system is stable. Here $m_0$ is defined as
\begin{equation*}
    m_0 \triangleq \inf \left\{ m : m\ge 2,  f(m') \le \frac \epsilon 2 \space \forall m'\ge m\right\} \le \left((1+C_V)M^2K\ln K \epsilon^{-1}\right)^{\O(1/\delta)}
\end{equation*}
where
\begin{equation*}
    f(m) = 42(1+C_V)M^2Km^{- \frac \delta 3} \cdot \left( 3\ln m + \ln K \right).
\end{equation*}
\end{theorem}
\begin{proof}
    See \Cref{sec-apdx-ssmw-moment}.
\end{proof}


\section{Related Work}

Recent literature includes learning-based scheduling policies that require little prior-knowledge and can gather channel statistics at run-time.

Learning-based approaches to scheduling queueing systems without perfect channel state knowledge require substantial exploration, to probe for more information of all the channels inside the system instead of merely exploiting the statistics at hand (e.g., via a \texttt{MaxWeight} style planning). Typical ways to introduce adequate exploration include epsilon-greedy, which explicitly allocates a small probability to serve each channel unconditionally \cite{neely2012max,krishnasamy2018augmenting,krishnasamy2021learning}; here the exploration is blind to the queue sizes and historical channel statistics and thus almost decoupled from exploitation. 
By contrast, optimistic exploration works by adding bonus terms to current channel statistics, so that  exploration and exploitation are naturally coupled during scheduling \cite{choudhury2021job,krishnasamy2021learning,stahlbuhk2018learning,yang2022maxweight}. Upper confidence bound (UCB) \cite{auer2002using} is a classical method for designing a bonus term. 

Existing works on scheduling in non-stationary queueing systems include \cite{yang2022maxweight}, which uses discounted UCB estimators for an up-to-date service rate of each link to replace the actual mean services rate in classical \texttt{MaxWeight}. The resulting policy can stabilize problem instances where the difference of each link's arrival (and service) rates between any two time steps in any time window of length $W$ is sufficiently small, and this window length $W$ needs to match with the discounting factor $\gamma$ used in discounted UCB estimators. Compared to \cite{yang2022maxweight}, our smoothness assumption is on the reference randomized policies rather than the true service rates.  

\section{Conclusions and Future Work}

In this paper, we propose a novel approach to apply adversarial bandit learning techniques to schedule queueing systems with unknown, time-varying channel states. The presented new algorithms \texttt{SoftMW} and \text{SSMW} are capable of stabilizing the system whenever the system can be stabilized by some (possibly unknown) sequence of randomized policies, and their time-variation satisfies some mild condition. We further generalize our results to the setting where  arrivals and departures only have bounded moments and develop two stablizing algorithms \texttt{SoftMW+}  and \texttt{SSMW+}. 

We believe our approach can be generalized to more complex stochastic networks (e.g., multi-hop networks), and to achieve other tasks such as utility optimization subject to queue stability. It is also an interesting future work to design distributed network scheduling algorithms using  adversarial bandit learning techniques.

\bibliographystyle{ACM-Reference-Format}
\bibliography{references}

\newpage
\appendix
\renewcommand{\appendixpagename}{\centering \LARGE Supplementary Materials}
\appendixpage

\startcontents[section]
\printcontents[section]{l}{1}{\setcounter{tocdepth}{2}}

\section{Simulation Results}
\label{sec-apdx-experiments}

In this section, we evaluate performance of the presented algorithms \texttt{SoftMW} (\Cref{softmw}) and \texttt{SSMW} (\Cref{ssmw})  based on synthetic problem instances. In our experiments, there are $K=5$ queues, the arrival and departure processes of each queue are all Bernoulli. Specifically, we assign the arrival rates and departure rates as follows:
\begin{itemize}
    \item The arrival processes of each queue are fixed Bernoulli distributions, i.e., $\vec \lambda_t = \vec \lambda^{(0)}$ for all $t\ge 1$ for some fixed $\vec \lambda^{(0)}\in [0,1]^K$. 
    \item We set the Bernoulli service rates for each queue to \\
    $\vec \sigma_t = \min\{\max\{\vec \sigma^{(0)} + \vec \zeta_t, \mathbf 0\}, \mathbf 1\}$, \\
    where $\vec \sigma^{(0)}\in [0,1]^K$ is a fixed vector, and $\vec \zeta_t$ is generated from some $K$-dimensional stochastic process such that $\mathbb E[\vec \zeta_t] = \mathbf 0$, i.e., they are marginally zero-mean.
\end{itemize}

Therefore, each considered problem instance can be regarded as the superposition of a stationary Bernoulli arrival and service process pair and a $0$-mean noise process on the service. 

\subsection*{The Pre-noising Stationary Problem}
In the experiments, we consider the following stationary problem instance:
\begin{align}
\vec \lambda & = \left(0.25, 0.2, 0.15, 0.1, 0.05\right), \nonumber \\
\vec \sigma & = \left(0.9, 0.85, 0.8, 0.59, 0.39\right). \label{eq-apdx-expriment-stationary}
\end{align}
For this stationary problem, we solve the following linear programming problem
\begin{align}
\max_{\vec \theta \in \mathbb R^K, \epsilon} \quad & \epsilon \label{eq-apdx-expriments-LP} \\
\textrm{s.t.} \quad & \vec \theta \in \Delta^{[K]} \nonumber \\
& \epsilon + \lambda_i \le \theta_i \cdot \sigma_i \quad \forall i \in [K] \nonumber
\end{align}
to get $\vec \theta$, referred to as the LP-based randomized policy (see e.g., \cite{neely2010stochastic}) and $\epsilon$, the distance from the problem instance to the boundary of the capacity region. The solution to \Cref{eq-apdx-expriments-LP} is
\begin{align*}
        \vec \theta & = \left(0.27802, 0.23556, 0.18778, 0.16987, 0.12877\right),\\
    \epsilon & = 2.2207\cdot 10^{-4}.
\end{align*}

\subsection*{The Noise Process}

We use the following $AR(1)$ autoregressive process to generate the noise to be added to the stationary service rates:
\begin{equation}
    \zeta_{t,i} = \begin{cases} 
    0 & t = 0 \\
    0.999 \cdot \zeta_{t-1,i} + \mathcal N(0,0.005^2) & t \ge 1
    \end{cases}, \label{eq-apdx-experiments-ar1}
\end{equation}
for all $t\ge 1$ and $i\in[K]$. i.e., the noise value at time $t$ is the noise at $t-1$ times a discounting factor $0.999$, then add an independently sampled $0$-mean normal-distributed random variable.

For this noise mechanism, we generate one trajectory of $\{\vec \zeta_t : 0\le t \le 2\cdot 10^{6}\}$ (i.e., generate $K=5$ i.i.d. samples of the $1$-dimensional AR(1) process defined in \Cref{eq-apdx-experiments-ar1}). We \textit{save} this  trajectory, and use it in the subsequent repeated queueing simulations. In other words, when we simulate the system repeatedly, the noise vectors $\vec \zeta_t$'s will \textit{not} be resampled, hence the post-noising service rates $\vec \sigma_t$'s are \textit{fixed} before-hand, only the Bernoulli trial outcomes will be resampled in each new run. In \Cref{figure-apdx-experiments-noise}, we plot the first $50000$ items of the generated service noise for the first queue.

\begin{figure}
     \centering
         \centering\includegraphics[width=0.75\textwidth]{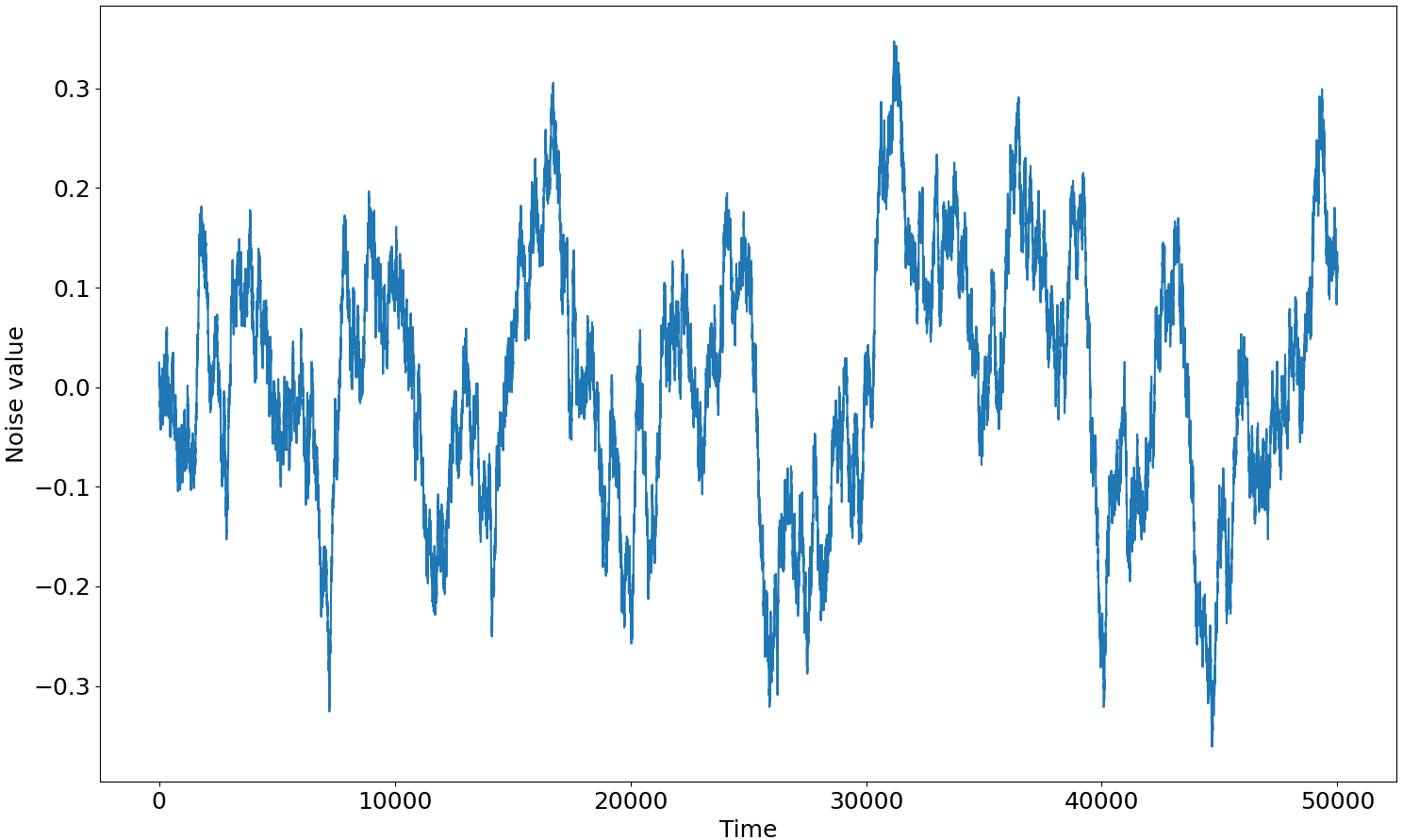}
\caption{Generated service noise sequence for Queue 1}
        \label{figure-apdx-experiments-noise}
\end{figure}

\subsection*{Algorithms Evaluated}
In the experiment, we evaluate the following algorithm instances:
\begin{itemize}
    \item (Baseline) The \texttt{MaxWeight} algorithm, taking the pre-noising service rate vector $\vec \sigma^{(0)}$ as input parameters. It schedules the system according to\\
    $a_t = \argmax_{i\in[K]} \left\{ Q_{t-1, i}\cdot \sigma^{(0)}_i\right\}$;
    \item (Oracle baseline when there are noises) The \texttt{MaxWeight} algorithm, taking all post-noising service rate vectors $\{\vec \sigma_t : 1\le t \le T\}$ as as input parameters, referred to as \texttt{MaxWeightGT}. It schedules the system according to\\
    $a_t = \argmax_{i\in[K]} \left\{ Q_{t-1, i}\cdot \sigma_{t,i}\right\}$;
    \item (Baseline) The LP-based randomized policy (referred to as \texttt{Randomized}). It schedules the system by sampling $a_t$ according to the probability vector $\vec \theta$ independently at each time step $t$.
    \item (Ours)  \texttt{SoftMW}  (\Cref{softmw}) with parameters $M=1$, $\delta = 0.5$ (referred to as \texttt{SoftMW-0.5});
    \item (Ours) \texttt{SoftMW} (\Cref{softmw})  with parameters $M=1$, $\delta = 0.1$ (referred to as \texttt{SoftMW-0.1});
    \item (Ours) \texttt{SoftMW} (\Cref{softmw}) with parameters $M=1$, $\delta = 0$ (referred to as \texttt{SoftMW-0});
    \item (Ours) \texttt{SSMW} (\Cref{ssmw}) with parameters $M=1$, $\delta = 0.5$ (referred to as \texttt{SSMW-0.5});
    \item (Ours) \texttt{SSMW} (\Cref{ssmw}) with parameters $M=1$, $\delta = 0.1$ (referred to as \texttt{SSMW-0.1});
    \item (Ours) \texttt{SSMW} (\Cref{ssmw}) with parameters $M=1$, $\delta = 0$ (referred to as \texttt{SSMW-0}).
\end{itemize}
\textbf{Remark. } Our main theoretical results, \Cref{thm-softmw-stab,thm-ssmw-stab}, cannot give queue stability guarantee for both algorithms running with $\delta = 0$. Nevertheless, with $\delta = 0$, both algorithms are proper and feasible queueing policies. Therefore, we also include them in the list of algorithm instances to evaluate, to illustrate the limit behavior of \texttt{SoftMW} and \texttt{SSMW} when $\delta$ is sufficiently small.

When we implement \texttt{SSMW}, we take the $\mathbf x_\tau$ value of the last step in the previous \texttt{EXP3.S+} epoch as the initial mixed action $\mathbf x_1$ in the new \texttt{EXP3.S+} instance.


\subsection*{Simulation Results}

We do the numerical evaluation on two problem instances, both based on the stationary problem specified in \Cref{eq-apdx-expriment-stationary}. In one instance, we set all service noise vectors $\vec \zeta_t$s to $\mathbf 0$; in the other instance, we set $\vec \zeta_t$ according to \Cref{eq-apdx-experiments-ar1}. In both problem instances, we simulate for $T = 2\cdot 10^6$ time steps.

In \Cref{figure-apdx-experiments-curve}, we plot the simulation results for the two problem instances, on which there are $8$ and $9$ algorithm instances being evaluated (in the no-noise problem \texttt{MaxWeightGT} and \texttt{MaxWeight} are identical), the curve for each algorithm is obtained by taking average over $20$ independent simulations. Below, we briefly summarize the simulation results.
\begin{itemize}
    \item In the no-noise problem instance  (\Cref{figure-apdx-experiments-curve-a}), \texttt{MaxWeight} leads to very small average queue lengths, surpassing other evaluated algorithms. This is not suprising, since \texttt{MaxWeight} is known to perform well given the precise channel condition. It is  worth noting that  our algorithms also stablize the network, despite not having such information. 
    
    \item In the noisy problem instance (\Cref{figure-apdx-experiments-curve-b}), \texttt{MaxWeight} experiences significant performance degrade, and the system seems unstable. While the oracle baseline \texttt{MaxWeightGT} still gives rather small queue lengths. Our \texttt{SoftMW} and \texttt{SSMW} algorithms all perform better than \texttt{Randomized}. In particular, \texttt{SSMW-0.1} and \texttt{SSMW-0} lead to average queue lengths close to \texttt{MaxWeightGT}.
\end{itemize}

\begin{figure}[H]
     \centering
     \begin{subfigure}{\textwidth}
         \centering
         \includegraphics[width=0.75\textwidth]{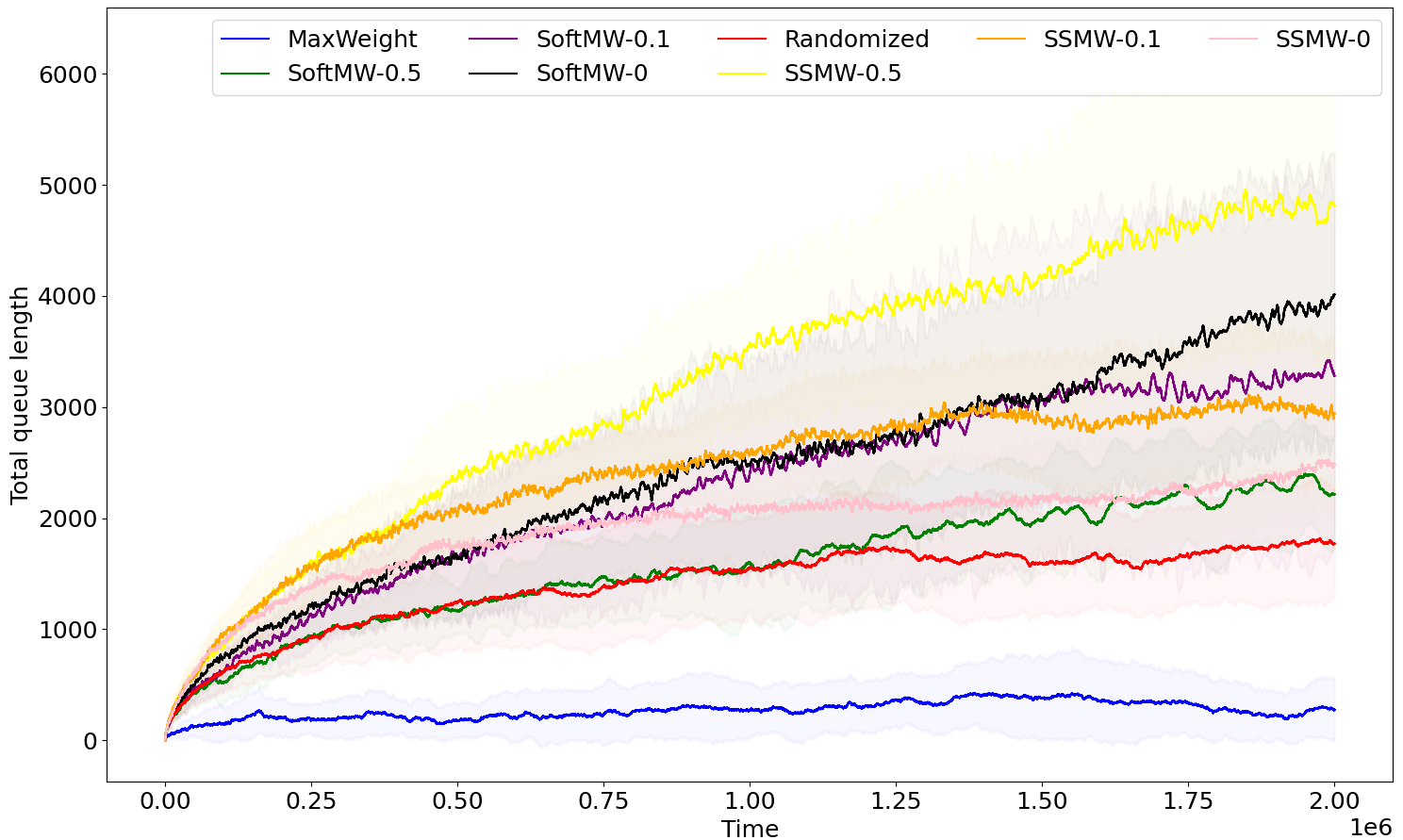}
         \caption{The Problem without noise. Here \texttt{MaxWeight} has access to the accurate channel information.}
         \label{figure-apdx-experiments-curve-a}
     \end{subfigure}
\end{figure}
\begin{figure}[H]\ContinuedFloat
     \centering
     \begin{subfigure}{\textwidth}
         \centering
         \includegraphics[width=0.75\textwidth]{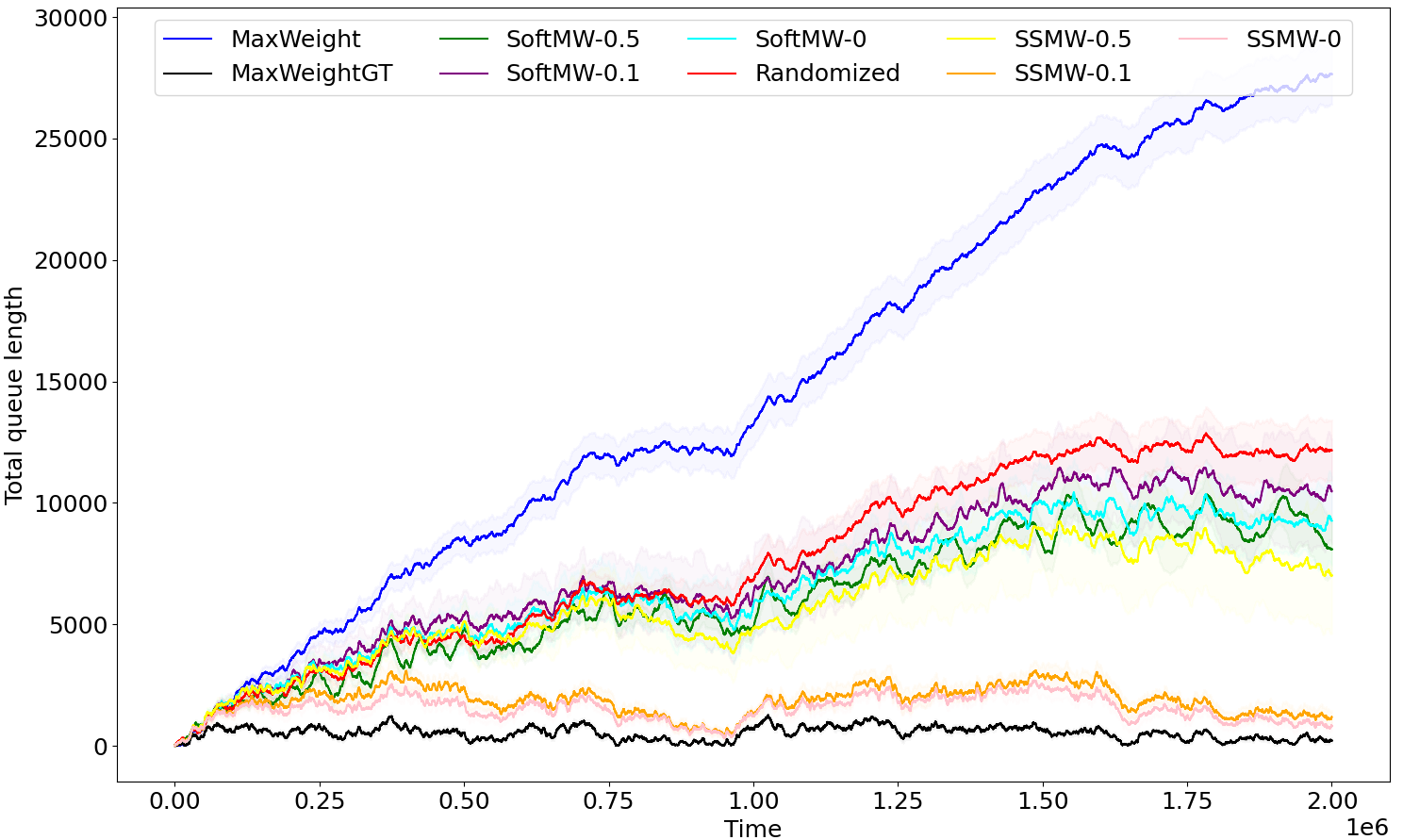}
         \caption{The Problem with AR(1) noise. Here \texttt{MaxWeightGT} has access to the accurate channel information, \texttt{MaxWeight} only has access to imperfect channel information that does not count in noises. }
         \label{figure-apdx-experiments-curve-b}
     \end{subfigure}
     \caption{Total Queue Lengths Plot}
     \label{figure-apdx-experiments-curve}
\end{figure}

\section{Proofs for \texttt{EXP3.S+} (Algorithm 1)}
\label{sec-apdx-exp3s}

Unless stated otherwise, for any strictly convex function $f$, we use $\bar{f}$ to denote its restriction on some $\triangle^{[K],\beta}$, where $\beta$ should be inferred from the context. Specifically,
\begin{equation*}
    \bar{f}(\mathbf x) \triangleq  \begin{cases}
    f(\mathbf x), & \mathbf x \in \triangle^{[K], \beta} \\
    \infty, & \mathbf x \notin \triangle^{[K], \beta}.
    \end{cases}
\end{equation*}

We list some important properties of Legendre functions and Bregman divergences below. The proof can be found in many literature on convex analysis, e.g., \citep{rockafellar2015convex}.
\begin{lemma}
\label{lemma-legendre}
Let $\mathcal{C}\subseteq \mathbb{R}^n$ be a convex set, $f:\mathcal{C} \rightarrow \mathbb{R}$ be a Legendre function. Then,
\begin{enumerate}
    \item $\nabla f$ is a bijection between $\mathop{int}(\mathcal{C})$ and $\mathop{int}(\mathop{dom}(f^*))$ with the inverse $(\nabla f)^{-1} = \nabla f^*$.
    \item $D_f(\mathbf y, \mathbf x) = D_{f^*}(\nabla f(\mathbf x), \nabla f(\mathbf y))$ for all $\mathbf x, \mathbf y \in \mathop{int}(\mathcal{C})$.
    \item The convex conjugate $f^*$ is Legendre.
    \item (Generalized Pythagorean Theorem) Let $W$ be a closed convex subset of $\mathcal C$, for any $\mathbf x \in \mathcal C$ denote by $\Pi_f(\mathbf x, W) \triangleq \argmin_{\mathbf x' \in W} D_f(\mathbf x', \mathbf x)$. The minimizer is guaranteed to exist uniquely, and for any $\mathbf y \in W$ we have $D_f(\mathbf y, \mathbf x) \ge D_f(\mathbf y, \Pi_f(\mathbf x, W)) + D_f(\Pi_f(\mathbf x, W), \mathbf x)$.
\end{enumerate}
\end{lemma}

Below is a technical lemma relating the single-step regret to single-step OMD update.
\begin{lemma}
\label{lemma-omd-single-step}
For any $\beta\in [0,1/K]$, $\eta > 0$, $\mathbf x,\mathbf y \in \triangle^{[K],\beta}$, $\mathbf g \in \mathbb{R}_+^K$ and Legendre function $\Psi: \mathbb{R}_+^{K}\ \rightarrow \mathbb{R}$, we have
\begin{equation}
    \left\langle \mathbf g, \mathbf y - \mathbf x \right\rangle \le \eta^{-1} D_\Psi(\mathbf y, \mathbf x) - \eta^{-1} D_\Psi(\mathbf y, \mathbf z) + \eta^{-1} D_\Psi(\mathbf x, \tilde{\mathbf z}) \label{eq-lemma-omd-single-step}
\end{equation}
where
\begin{equation}
    \mathbf z = \argmin_{\mathbf x' \in \triangle^{[K],\beta}} \left\langle - \eta \mathbf g, \mathbf x' \right\rangle + D_\Psi(\mathbf x', \mathbf x),\quad
    \tilde{\mathbf z} = \argmin_{x' \in \mathbb{R}_+^K} \left\langle - \eta \mathbf g, x' \right\rangle + D_\Psi(\mathbf x', \mathbf x), \label{eq-lemma-omd-single-step-z-def}
\end{equation}
or equivalently,
\begin{equation}
    \tilde{\mathbf z} = \nabla\Psi^*(\nabla\Psi(x) + \eta \mathbf g)\label{eq-lemma-omd-single-step-z-def-alt}
\end{equation}
\end{lemma}
\begin{proof}
    We first prove that \Cref{eq-lemma-omd-single-step-z-def} and \Cref{eq-lemma-omd-single-step-z-def-alt} are equivalent. $\Psi$ is Legendre hence $\nabla\Psi$ explodes on $\partial \mathbb{R}_+^K$, and the minimizer $\tilde{\mathbf z}$ defined by \Cref{eq-lemma-omd-single-step-z-def} will lie in $\mathop{int}(\mathbb{R}_+^K)$, furthermore, we must have $\frac{\partial}{\partial \tilde{\mathbf z}'}[\langle -\eta \mathbf g, \tilde{\mathbf z}\rangle + D_\Psi(\tilde{\mathbf z}, \mathbf x)] 
    = -\eta \mathbf g + \nabla\Psi(\tilde{\mathbf z}) - \nabla\Psi(\mathbf x) = \mathbf 0$, thus $\nabla\Psi(\tilde{\mathbf z}) = \nabla\Psi(\mathbf x) + \eta \mathbf g$. Then the bijection property in \Cref{lemma-legendre} states that $\tilde{\mathbf z} = \nabla \Psi^*(\nabla \Psi(\tilde{\mathbf z})) = \nabla\Psi^*(\nabla\Psi(x) + \eta \mathbf g)$, which is just \Cref{eq-lemma-omd-single-step-z-def-alt}.

    We know from \Cref{eq-lemma-omd-single-step-z-def-alt} that \begin{equation}
    \mathbf g = \eta^{-1}(\nabla \Psi(\tilde {\mathbf z}) - \nabla \Psi(\mathbf x)) \label{eq-lemma-omd-single-step-g}
    \end{equation} The first order optimality condition of $\tilde {z}$ in \Cref{eq-lemma-omd-single-step-z-def} implies that 
    $\langle -\eta \mathbf g + \nabla\Psi(\mathbf z) - \nabla\Psi(\mathbf x), \mathbf y - \mathbf z \rangle \ge 0$ for any $y \in \triangle^{[K], \beta}$, thus
    \begin{equation}
        \langle \mathbf g, \mathbf y - \mathbf z \rangle \le \eta^{-1}\langle \nabla\Psi(\mathbf z) - \nabla\Psi(\mathbf x), \mathbf y - \mathbf z \rangle\label{eq-lemma-omd-single-step-tilde-z-optimality}
    \end{equation}
    Therefore, we can write
    \begin{align*}
        \langle \mathbf g, \mathbf y - \mathbf x \rangle & = \langle \mathbf g, \mathbf y - \mathbf z \rangle + \langle \mathbf g, \mathbf z - \mathbf x \rangle \\
        & \stackrel{(a)}\le \eta^{-1}\langle \nabla\Psi(\mathbf z) - \nabla\Psi(\mathbf x), \mathbf y - \mathbf z \rangle + \eta^{-1}\langle \nabla \Psi(\tilde {\mathbf z}) - \nabla \Psi(\mathbf x), \mathbf z - \mathbf x \rangle \\
        & \stackrel{(b)}= \eta^{-1}(D_\Psi(\mathbf y, \mathbf z) + D_\Psi(\mathbf z, \mathbf x) - D_\Psi(\mathbf y, \mathbf x)) + \eta^{-1}(D_\Psi(\mathbf z, \mathbf x) + D_\Psi(\mathbf x, \tilde {\mathbf z}) - D_\Psi(\mathbf z, \tilde {\mathbf z})) \\
        & = \eta^{-1}(D_\Psi(\mathbf y, \mathbf x) - D_\Psi(\mathbf y, \mathbf z) + D_\Psi(\mathbf x, \tilde {\mathbf z}) - D_\Psi(\mathbf z, \tilde {\mathbf z})) \\
        & \le \eta^{-1}(D_\Psi(\mathbf y, \mathbf x) - D_\Psi(\mathbf y, \mathbf z) + D_\Psi(\mathbf x, \tilde {\mathbf z})).
    \end{align*}
    Here in step $(a)$, we plug in \Cref{eq-lemma-omd-single-step-tilde-z-optimality} for the first term and \Cref{eq-lemma-omd-single-step-g} for the second term, in step $(b)$, we use the following ``three-point identity'' of Bregman divergences:
\begin{equation*}
    D_\Psi(\mathbf a, \mathbf b) + D_\Psi(\mathbf b, \mathbf c) - D_\Psi(\mathbf a, \mathbf c) = \langle \nabla\Psi(\mathbf c) - \nabla\Psi(\mathbf b), \mathbf a - \mathbf b \rangle,
\end{equation*}
    which can be verified by expanding all Bregman divergences.
\end{proof}

Intuitively, \Cref{lemma-omd-single-step} gives an upper-bound in Bregman divergences for the \textit{sample-path} single step regret of an mixed action $\mathbf x$ against a base-line mixed action $\mathbf y$, where the reward vector is $\mathbf g$. Our high-level idea towards \Cref{thm-exp3s} is to sum up the bound in \Cref{lemma-omd-single-step} and then appropriately take expectations and control the expectations. Firstly, we can get a sample-path bound for the total dynamic regret.

\begin{lemma}
\label{lemma-exp3s-sample-path}
For any fixed sequence $\vec \theta_1 \in \Delta^{[K], \beta_1},\ldots, \vec \theta_T \in \Delta^{[K], \beta_T}$, \Cref{exp3s} guarantees that
\begin{equation}
    \sum_{t=1}^T \langle \tilde{\mathbf g}_t, \vec \theta_t - \mathbf x_t \rangle \le \left(1 + \sum_{t=1}^{T-1} \lVert \vec \theta_{t+1} - \vec \theta_t\rVert_1 \right)\eta_T^{-1}\ln \frac 1 {\beta_T} - \eta_T^{-1}D_\Psi(\vec \theta_T, \mathbf z_T) + \sum_{t=1}^T \eta_t^{-1}D_\Psi(\mathbf x_t, \tilde {\mathbf z_t}) \label{eq-lemma-exp3s-sample-path}
\end{equation}
where $\Psi(\mathbf x) \triangleq \sum_{i=1}^K (x_i \ln x_i - x_i)$, $\tilde {\mathbf z}_t \triangleq \nabla \Psi^*(\nabla \Psi( \mathbf x_t) + \eta_t \tilde{\mathbf g}_t)$, $\mathbf z_t \triangleq \argmin_{\mathbf x' \in \triangle^{[K],\beta_t}} \left\langle - \eta_t \tilde {\mathbf g}_t, \mathbf x' \right\rangle + D_\Psi(\mathbf x', \mathbf x_t)$.
\end{lemma}
\begin{proof}
    We will prove \Cref{eq-lemma-exp3s-sample-path} by induction on $T$. \Cref{eq-lemma-exp3s-sample-path} trivially holds when $T=0$ (here we use the convention that $\eta_0^{-1} = 0$). Then it suffice to verify \Cref{lemma-exp3s-sample-path}, while assuming that
    \begin{equation*}
        \sum_{t=1}^{T-1} \langle \tilde{\mathbf g}_t, \vec \theta_t - \mathbf x_t \rangle \le \left(1 + \sum_{t=1}^{T-2} \lVert \vec \theta_{t+1} - \vec \theta_t\rVert_1 \right)\eta_{T-1}^{-1}\ln \frac 1 {\beta_{T-1}} - \eta_{T-1}^{-1}D_\Psi(\vec \theta_{T-1}, \mathbf z_{T-1}) + \sum_{t=1}^{T-1} \eta_t^{-1}D_\Psi(\mathbf x_t, \tilde {\mathbf z_t}).
    \end{equation*}

In \Cref{exp3s}, the $\Psi$ function is chosen to the negative entropy, and we have
\begin{equation}
    D_\Psi(\mathbf y, \mathbf x) = \sum_{t=1}^K y_i \ln \frac{y_i} {x_i}. \label{eq-lemma-exp3s-sample-path-bd}
\end{equation}
Hence for any $\mathbf y\in \Delta^{[K]}, \mathbf x \in \Delta^{[K], \beta}$, since $x_i \ge \beta$ for all $i \in [K]$, we have
\begin{equation}
    D_\Psi(\mathbf y, \mathbf x) \le \sum_{i=1}^K y_i \ln \frac 1 {x_i} \le \ln \frac 1 \beta. \label{lemma-exp3s-sample-path-bd-upper-bound}
\end{equation}

Note that the choice of $\mathbf x_t$ in \Cref{exp3s} is just $\mathbf z_{t-1}$, hence 
    we can write
    \begin{align*}
        & \quad \eta_T^{-1} D_\Psi(\vec \theta_T, \mathbf x_T) \\
        & = \eta_T^{-1} D_\Psi(\vec \theta_T, \mathbf z_{T-1}) \\
        & = \eta_{T-1}^{-1} D_\Psi(\vec \theta_{T-1}, \mathbf z_{T-1}) + \eta_{T-1}^{-1}\left(D_\Psi(\vec \theta_T, \mathbf z_{T-1}) - D_\Psi(\vec \theta_{T-1}, \mathbf z_{T-1}) \right) + (\eta_T^{-1} - \eta_{T-1}^{-1}) D_\Psi(\vec \theta_T, \mathbf z_{T-1}) \\
        & \stackrel{(a)}\le \eta_{T-1}^{-1} D_\Psi(\vec \theta_{T-1}, \mathbf z_{T-1}) + \eta_{T-1}^{-1}\sum_{t=1}^K \left(\theta_{T,i}\ln\frac{\theta_{T,i}}{z_{T-1,i}} - \theta_{T-1,i}\ln\frac{\theta_{T-1,i}}{z_{T-1,i}}\right) + (\eta_T^{-1} - \eta_{T-1}^{-1})\ln \frac 1 {\beta_{T-1}} \\
        & \stackrel{(b)}\le \eta_{T-1}^{-1} D_\Psi(\vec \theta_{T-1}, \mathbf z_{T-1}) + \eta_{T-1}^{-1}\sum_{i=1}^K \left( (1-z_{T-1,i})\ln \frac 1 {z_{T-1,i}} \lvert \theta_{T,i} - \theta_{T-1, i} \rvert\right) + (\eta_T^{-1} - \eta_{T-1}^{-1})\ln \frac 1 \beta \\
        & \le \eta_{T-1}^{-1} D_\Psi(\vec \theta_{T-1}, \mathbf z_{T-1}) + \eta_{T-1}^{-1} \lVert \vec \theta_T - \vec \theta_{T-1}\rVert_1 \ln \frac 1 {\beta_{T-1}} + (\eta_T^{-1} - \eta_{T-1}^{-1})\ln \frac 1 {\beta_{T-1}},
    \end{align*}
    where in step $(a)$ we plug in \Cref{eq-lemma-exp3s-sample-path-bd} for the second term and \Cref{lemma-exp3s-sample-path-bd-upper-bound} for the third term, in step $(b)$ we use the following fact: let $f(x)= x\ln \frac x y$, then $f'(x) = (1 - y) \ln \frac x y \le (1 - y) \ln \frac 1 y$ as long as $x,y \in (0,1)$.

    Then, we can apply \Cref{lemma-omd-single-step} for the $T$-th time step and write
    \begin{align*}
        \langle \tilde{\mathbf g}_T, \vec \theta_T - \mathbf x_T \rangle & \le \eta_T^{-1} D_\Psi(\vec \theta_T, \mathbf x_T) - \eta_T^{-1} D_\Psi(\vec \theta_T, \mathbf z_T) + \eta_T^{-1} D_\Psi(\mathbf x_T, \tilde{\mathbf z}_T) \\
        & \le \eta_{T-1}^{-1} D_\Psi(\vec \theta_{T-1}, \mathbf z_{T-1}) - \eta_T^{-1} D_\Psi(\vec \theta_T, \mathbf z_T) + \eta_T^{-1} D_\Psi(\mathbf x_T, \tilde{\mathbf z}_T) \\
        & \quad + \eta_{T-1}^{-1} \lVert \vec \theta_T - \vec \theta_{T-1}\rVert_1 \ln \frac 1 {\beta_{T-1}} + (\eta_T^{-1} - \eta_{T-1}^{-1})\ln \frac 1 {\beta_{T-1}},
    \end{align*}
    which is adequate for the induction step, recalling that $\eta_t$'s and $\beta_t$'s are both decreasing.
\end{proof}

\begin{lemma}
    \label{lemma-exp3s-immediate-cost}
    For each term $\eta_t^{-1}D_\Psi(\mathbf x_t, \tilde {\mathbf z_t})$ defined in \Cref{lemma-exp3s-sample-path}, we have
    \begin{equation*}
        \mathbb E\left[\left. \eta_t^{-1}D_\Psi(\mathbf x_t, \tilde {\mathbf z_t}) \right\rvert \mathcal F_{t-1}\right] \le e \eta_t \lVert\mathbf g_t\rVert_2^2.
    \end{equation*}
\end{lemma}
\begin{proof}
    In fact, for any choice of $\Psi$, we have
\begin{align}
    \eta_t^{-1}D_\Psi(x_t, \tilde{z}_t) & \stackrel{(a)}= \eta_t^{-1} D_{\Psi^*}(\nabla\Psi(\tilde{\mathbf z}_t), \nabla\Psi(\mathbf x_t)) \nonumber \\
    & = \eta_t^{-1} D_{\Psi^*}(\nabla\Psi(\mathbf x_t) + \eta_t \tilde{\mathbf g}_t, \nabla\Psi(\mathbf x_t)) \nonumber \\
    & = \eta_t^{-1}\left(\Psi^*(\nabla\Psi(\mathbf x_t) + \eta_t \tilde{\mathbf g}_t) - \Psi^*(\nabla\Psi(\mathbf x_t)) - \langle \mathbf x_t,  \eta_t \tilde{\mathbf g}_t\rangle\right) \nonumber \\
    & \stackrel{(b)}= \frac {\eta_t} 2 \Vert \tilde{\mathbf g}_t \Vert_{\nabla^2\Psi^*(\mathbf w_t)}^2, \label{eq-immediate-cost-matrix-norm}
\end{align}
where in step $(a)$ we use the duality property in \Cref{lemma-legendre}, in step $(b)$ we regard the Bregman divergence as a second order Lagrange remainder, $\mathbf w_t$ is some element inside the line segment connecting $\nabla\Psi(\mathbf x_t) + \eta_t\tilde{\mathbf g}_t$ and $\nabla\Psi(\mathbf x_t)$. i.e., $w_{t,i} = \nabla\Psi(\mathbf x_t)_i$ for all $i\ne a_t$, $w_{t,a_t} \in [\nabla\Psi(\mathbf x_t)_{a_t}, \nabla\Psi(\mathbf x_t)_{a_t} + \eta_t \frac {g_{t,a_t}}{p_{t,a_t}}]$. Since we have assumed that $\eta_t^{-1} \gamma_t e_{t,i} \ge g_{t,i}$ for all $i$, and we use an explicit exploration mechanism to guarantee that $p_{t,i} = (1 - \gamma_t)x_{t,i} +  \gamma_t e_{t,i} \ge \gamma_t e_{t,i}$, hence now we have $\eta_t\frac{g_{t,a_t}}{p_{t,a_t}} \le 1$.

Now for the particular choice of $\Psi$ being the negative entropy function, we have $\nabla \Psi(\mathbf x) = (\ln x_1, \ldots, \ln x_K)$ and $\nabla^2 \Psi(\mathbf x) = \mathrm{Diag}(x_1^{-1},\ldots, x_K^{-1})$. Using the property (1) in \Cref{lemma-legendre} we can see that
\begin{align*}
    \nabla^2 \Psi^*(\mathbf w_t) & = \left( \nabla^2 \Psi(\nabla \Psi^*(\mathbf w_t))\right)^{-1} \\
    & = \mathrm{Diag}(\exp(w_{t,1})^{-1},\ldots, \exp(w_{t,K})^{-1})^{-1} \\
    & = \mathrm{Diag}(\exp(w_{t,1}),\ldots, \exp(w_{t,K})),
\end{align*}
therefore, we have
\begin{equation}
    \mathrm{Diag}(x_{t,1},\ldots, x_{t,K}) \preceq \nabla^2 \Psi^*(\mathbf w_t) \preceq e\mathrm{Diag}(x_{t,1},\ldots, x_{t,K}) \label{eq-immediate-cost-matrix-norm-2}
\end{equation}
We can then plug \Cref{eq-immediate-cost-matrix-norm-2} into \Cref{eq-immediate-cost-matrix-norm} to get
\begin{align*}
    \eta_t^{-1}D_\Psi(x_t, \tilde{z}_t) & \le \frac {e\eta_t}2 \tilde g_{t,a_t}^2 x_{t,a_t} \\
    & = \frac {e\eta_t}2 \frac{g_{t,a_t}^2}{p_{t,a_t}^2}x_{t,a_t}.
\end{align*}
Taking expectation, we can see
\begin{align*}
    \mathbb E\left[\left.\eta_t^{-1}D_\Psi(x_t, \tilde{z}_t) \right\rvert \mathcal F_{t-1}\right] & \le \frac {e\eta_t}2 \sum_{i=1}^K g_{t,i}^2 \frac{x_{t,i}}{p_{t,i}} \\
    & \le e\eta_t \sum_{i=1}^K g_{t,i}^2
\end{align*}
where the last step is due to $\eta_t \le \frac 1 2$, hence $p_{t,i} \ge (1-\gamma_i)x_{t,i} \ge \frac 1 2x_{t,i}$.
\end{proof}

Now we are ready to prove \Cref{thm-exp3s}.
\begin{proof}[Proof of \Cref{thm-exp3s}]
    Take expectaion on both sides of \Cref{eq-lemma-exp3s-sample-path} in \Cref{lemma-exp3s-sample-path}, note that
    \begin{equation*}
        \mathbb E\left[\left. \langle \tilde{\mathbf g}_t, \vec \theta_t - \mathbf x_t \rangle \right\rvert \mathcal F_{t-1}\right] = \langle \mathbf g_t, \vec \theta_t - \mathbf x_t \rangle,
    \end{equation*}
    and
    \begin{align*}
        \mathbb E\left[\left. g_{t,a_t} \right\rvert \mathcal F_{t-1}\right] & = \langle \mathbf g_t, \mathbf p_t \rangle \\
        & = \langle 
        \mathbf g_t, \mathbf x_t \rangle + \langle 
        \mathbf g_t, \mathbf p_t - x_t \rangle \\
        & = \langle 
        \mathbf g_t, \mathbf x_t \rangle + \langle 
        \mathbf g_t, \mathbf -\gamma_t \mathbf x_t + \gamma_t \mathbf e_t \rangle \\
        & \le \langle 
        \mathbf g_t, \mathbf x_t \rangle + \gamma_t \langle 
        \mathbf g_t, \mathbf e_t \rangle.
    \end{align*}
    Then apply \Cref{lemma-exp3s-immediate-cost}.
\end{proof}

\section{On the Implementation of \texttt{EXP3.S+} (Algorithm 1)}
\label{sec-apdx-exp3s-argmax}

The core operation when implementing our  \texttt{EXP3.S+} (\Cref{exp3s}) is the $\argmin$ expression in \Cref{line-exp3s-argmax}. In other words, we need to solve the following optimization problem:
\begin{align}
\min_{\mathbf y\in \mathbb R^K} \quad & -\langle \mathbf g, \mathbf y \rangle + D_\Psi(\mathbf y, \mathbf x) \label{eq-apdx-exp3s-opt} \\
\textrm{s.t.} \quad & y_i \ge \beta \quad  \forall 1 \le i \le K \nonumber \\
& \sum_{i=1}^K y_i = 1 \nonumber
\end{align}
where $\mathbf g\in \mathcal R_+^K$ and $x\in \Delta^{[K]}$ are two fixed vectors, $0 \le \beta < \frac 1 K$ is a constant.  
For this optimization problem, we can introduce a set of Lagrange multipliers $\lambda_1,\ldots,\lambda_K,\mu$ and the Lagrangian
\begin{equation*}
    L(\mathbf y; \vec \lambda, \mu) \triangleq -\sum_{i=1}^K g_i y_i + D_\Psi(\mathbf y, \mathbf x) + \sum_{i=1}^K \lambda_i(y_i - \beta) + \mu\left(\sum_{i=1}^K y_i - 1\right).
\end{equation*}
Let $\mathbf y^*\in \Delta^{[K],\beta}$ be an optimizer for the optimization problem \Cref{eq-apdx-exp3s-opt}, and let $\vec \lambda^*, \mu^*$ be the corresponding multipliers. 
Then we can write down the KKT condition for $\mathbf y^*$:
\begin{align*}
  \left.\frac{\partial L}{\partial y_i}\right\rvert_{\mathbf y^*;\vec \lambda^*, \mu^*} = -g_i + \ln(y^*_i) -\ln(x_i) + \lambda^*_i +\mu^* & = 0 \quad \forall 1 \le i \le K,\\
    y^*_i > \beta \Rightarrow\lambda^*_i & = 0 \quad \forall 1 \le i \le K.
\end{align*}
Denote by $\mathcal I\triangleq \left\{i\in [K]: y^*_i > \beta \right\}$, then the KKT conditions implies that
\begin{equation}
    \ln(y^*_i) - \ln(x_i) -g_i = \ln(y^*_j) - \ln(x_j) - g_j \label{eq-apdx-exp3s-1}
\end{equation}
for all $i,j\in \mathcal I$. On the other hand, we know that
\begin{equation}
    \sum_{i\in\mathcal I} y^*_i = 1 - \beta \cdot (K - \left\lvert\mathcal I\right\rvert) \label{eq-apdx-exp3s-2}
\end{equation}
since we have assumed that $y^*_i = \beta$ for all $i\in [K] \setminus \mathcal I$. In fact, \Cref{eq-apdx-exp3s-1,eq-apdx-exp3s-2} have an explicit  ``soft-max'' solution
\begin{equation*}
    y^*_i \propto x_i\exp(g_i)
\end{equation*}
and thus
\begin{equation*}
    y^*_i = \left[1 - \beta \cdot (K - \left\lvert\mathcal I\right\rvert)
\right]\cdot \frac{x_i\exp(g_i)}{\sum_{j\in\mathcal I}x_j\exp(g_j)}
\end{equation*}
for all $i\in \mathcal I$.
Therefore, it remains to determine $\mathcal I$ to complete depict $y^*$. Note that the optimization problem \Cref{eq-apdx-exp3s-opt} satisfies Slater's condition, hence its optimizers are completely determined by the KKT conditions. One can verify that the following problem
\begin{align}
\min_{\mathbf y\in \mathbb R^K} \quad &  D_\Psi(\mathbf y, \mathbf x') \label{eq-apdx-exp3s-opt2} \\
\textrm{s.t.} \quad & y_i \ge \beta \quad  \forall 1 \le i \le K \nonumber \\
& \sum_{i=1}^K y_i = 1 \nonumber
\end{align}
where $x'_i = x_i \exp(g_i)$, has the exactly same KKT conditions as \Cref{eq-apdx-exp3s-opt}. Furthermore, without loss of generality, we assume that $x'_1 \le x'_2 \le \cdots \le x'_K$ in \Cref{eq-apdx-exp3s-opt2} (otherwise we can permute and rearrange the coordinates of $\mathbf x$ since both $D_\Psi(\cdot,\cdot)$ and the constraints are symmetric in all coordinates). Then, we can claim that $y^*_1 \le y^*_2 \le \cdots \le y^*_K$, i.e., the coordinates of $y^*$ are also in a sorted order. The reason is, we can expand the objective of \Cref{eq-apdx-exp3s-opt2} as:
\begin{align*}
    OBJ & = \Psi(\mathbf y) - \Psi(\mathbf x') - \langle \nabla \Psi(\mathbf x'), \mathbf y - \mathbf x' \rangle \\
    & = \sum_{i=1}^K y_i \ln(y_i) - \sum_{i=1}^K x'_i \ln(x'_i) - \sum_{i=1}^K \ln(x'_i)(y_i - x'_i) - 1 + \lVert\mathbf x'\rVert_1 \\
    & = \sum_{i=1}^K y_i \cdot (-\ln(x'_i)) + \text{some permutation-invariant quantity of }\mathbf y
\end{align*}
Then, one can see if $\mathbf y$ has a pair of coordinates $i<j$ but $y_i > y_j$, then swapping $y_i$ and $y_j$ makes the objective function smaller. Therefore, after reducing the optimization problem to \Cref{eq-apdx-exp3s-opt2} with ascending $x'_i$'s, there exists $0\le i\le K$ such that $\beta = y^*_1 = \cdots = y^*_i = \beta < y^*_{i+1} \le y^*_{i+2} \le \cdots \le y^*_K$. To solve \Cref{eq-apdx-exp3s-opt2}, it suffices to enumerate this boundary index $i=0\ldots K$, then calculate $y^*_{i+1},\ldots,y^*_K$ by
\begin{equation*}
    y^*_j = (1-\beta\cdot i)\cdot \frac{x'_j}{\sum_{k=i+1}^K x'_k}.
\end{equation*}
If $y^*_{i+1} \ge \beta$, we can accept the current $i$ as the optimal boundary index, and conclude that $(\beta,\ldots,\beta, y^*_{i+1},\ldots,y^*_K)$ is the optimizer of \Cref{eq-apdx-exp3s-opt2}.

\section{Proofs for General Qurdratic Lyapunov Analysis}
\begin{proof}[Proof of \Cref{lemma-quad-lyapunov}]
\label{proof-lemma-quad-lyapunov}
    Recall that in this paper we denote by $\mathbf Q_t$ the queue lengths at the end of $t$-th time slot, by $\mathbf A_t$, $\mathbf S_t$ the arrivals and (maximum) services respectively at time $t$. Define
    \begin{equation*}
        \mathbf U_t = \mathbf \max\{ -\mathbf Q_{t-1} - \mathbf A_t + \mathbf S_t \odot \mathbf 1_{a_t}, \mathbf 0 \}
    \end{equation*}
    to be the unused services at time $t$, then it is easy to see we have
    \begin{equation}
        \mathbf Q_t = \mathbf Q_{t-1} + \mathbf A_t - \mathbf S_t \odot \mathbf 1_{a_t} + \mathbf U_t \label{eq-proof-lemma-quad-lyapunov-QASU}
    \end{equation}
    and
    \begin{align}
        \langle \mathbf Q_t, \mathbf U_t \rangle & = 0, \label{eq-proof-lemma-quad-lyapunov-QUdot} \\
        \langle \mathbf Q_{t-1} + \mathbf A_t -\mathbf S_t \odot \mathbf 1_{a_t}, \mathbf U_t \rangle & \le 0 \label{eq-proof-lemma-quad-lyapunov-QUdot-2}
    \end{align}
    for all $t\ge 1$.

    Using \Cref{eq-proof-lemma-quad-lyapunov-QASU} to express $L_t$, we can see
    \begin{align*}
        L_t - L_{t-1} & = \frac 1 2\lVert \mathbf Q_{t-1} + \mathbf A_t -\mathbf S_t \odot \mathbf 1_{a_t} + \mathbf U_t\rVert_2^2 - \frac 1 2 \lVert \mathbf Q_{t-1}\rVert_2^2 \\
& = \frac 1 2\lVert \mathbf Q_{t-1} + \mathbf A_t -\mathbf S_t \odot \mathbf 1_{a_t}\rVert_2^2 + \frac 1 2 \lVert \mathbf U_t\rVert_2^2 + \langle \mathbf Q_{t-1} + \mathbf A_t - \mathbf S_t \odot \mathbf 1_{a_t}, \mathbf U_t \rangle - \frac 1 2 \lVert \mathbf Q_{t-1}\rVert_2^2 \\
& \stackrel{(a)}\le \frac 1 2\lVert \mathbf Q_{t-1} + \mathbf A_t -\mathbf S_t \odot \mathbf 1_{a_t}\rVert_2^2 + \frac 1 2 \lVert \mathbf U_t\rVert_2^2 - \frac 1 2 \lVert \mathbf Q_{t-1}\rVert_2^2 \\
& = \frac 1 2 \lVert \mathbf A_t - \mathbf S_t \odot \mathbf 1_{a_t}\rVert_2^2 + \frac 1 2 \lVert \mathbf U_t\rVert_2^2 + \langle \mathbf Q_{t-1}, \mathbf A_t - \mathbf S_t \odot \mathbf 1_{a_t}\rangle \\
& \stackrel{(b)}\le \frac {(K+1) M^2} 2 + \langle \mathbf Q_{t-1}, \mathbf A_t - \mathbf S_t \odot \mathbf 1_{a_t}\rangle,
    \end{align*}
where step $(a)$ is due to equation \Cref{eq-proof-lemma-quad-lyapunov-QUdot-2}, $(b)$ is due to $\lVert \mathbf A_t - \mathbf S_t \odot \mathbf 1_{a_t}\rVert_\infty, \lVert \mathbf U_t\rVert_\infty \le M$, and $\mathbf U_t$ can have at most one positive entry. Thus
    \begin{align*}
        \mathbb E\left[\left. 
L_t - L_{t-1} \right\rvert \mathcal F_{t-1}\right] & \le \frac {(K+1) M^2} 2 + \mathbb E\left[\left. 
\langle \mathbf Q_{t-1}, \mathbf A_t - \mathbf S_t \odot \mathbf 1_{a_t}\rangle \right\rvert \mathcal F_{t-1}\right] \\
& \le \frac {(K+1) M^2} 2 + \langle \mathbf Q_{t-1}, \vec \lambda_t - \vec \sigma_t \odot \mathbf p_t\rangle,
    \end{align*}
    where the last step is because $\mathbf Q_{t-1}$ is $\mathcal F_{t-1}$-measurable, and $\mathbf A_t, \mathbf S_t$ are both independent to $a_t$ and $\mathcal F_{t-1}$, $\mathbb P[a_t = i\rvert \mathcal F_{t-1}] = p_{t,i}$.
\end{proof}

\begin{proof}[Proof of \Cref{lemma-quad-lyapunov-theta}]
\label{proof-lemma-quad-lyapunov-theta}
For each interval of time steps $W_j$ in Assumption~\ref{assumption-theta}, denote by $T_0$ the index of the first time slot in $W_j$, then we can write
\begin{align*}
    \sum_{t \in W_j} \langle \mathbf Q_{t-1}, \vec \theta_t \odot \vec \sigma_t - \vec \lambda_t \rangle & =  \sum_{t \in W_j} \langle \mathbf Q_{T_0 - 1}, \vec \theta_t \odot \vec \sigma_t - \vec \lambda_t \rangle + \sum_{t \in W_j} \langle \mathbf Q_{t-1} - \mathbf Q_{T_0 - 1}, \vec \theta_t \odot \vec \sigma_t - \vec \lambda_t \rangle,
\end{align*}
then bound the two sums one by one. First the sum with $\mathbf Q_{T_0 -1}$ factors, we have
\begin{align*}
    \sum_{t \in W_j} \langle \mathbf Q_{T_0 - 1}, \vec \theta_t \odot \vec \sigma_t - \vec \lambda_t \rangle & \stackrel{(a)}\ge \epsilon \lvert W_j\rvert \cdot \lVert \mathbf Q_{T_0 - 1}\rVert_1 \\
    & \ge \epsilon \sum_{t \in W_j} \left(\lVert 
\mathbf Q_{t-1} \rVert_1 - \lVert \mathbf Q_{t-1} - \mathbf Q_{T_0 - 1} \rVert_1\right) \\
    & \stackrel{(b)}\ge \epsilon \sum_{t \in W_j} \left(\lVert 
\mathbf Q_{t-1} \rVert_1 - KM(t - T_0)\right) \\
    & \ge \epsilon \sum_{t \in W_j} \lVert 
\mathbf Q_{t-1} \rVert_1 - \epsilon KM \left(\lvert 
W_j\rvert - 1\right)^2
\end{align*}
where step $(a)$ is due to \Cref{eq-assumption-theta-eps} in Assumption~\ref{assumption-theta}, $(b)$ is because the queue length increments are bounded by $M$.

As for the sum with $\mathbf Q_{t-1} - \mathbf Q_{T_0 -1}$, we have 
\begin{align*}
    \sum_{t \in W_j} \langle \mathbf Q_{t-1} - \mathbf Q_{T_0 - 1}, \vec \theta_t \odot \vec \sigma_t - \vec \lambda_t \rangle & \ge -\sum_{t \in W_j} \lVert \mathbf Q_{t-1} - \mathbf Q_{T_0 - 1} \rVert_1 \cdot \lVert \vec \theta_t \odot \vec \sigma_t - \vec \lambda_t \rVert_\infty \\
    & \ge -\sum_{t \in W_j}KM(t - T_0) \cdot M \\
    & \ge - KM^2 \left( \lvert W_j \rvert - 1 \right)^2.
\end{align*}
Sum up the two parts, we get
\begin{equation}
    \sum_{t \in W_j} \langle \mathbf Q_{t-1}, \vec \theta_t \odot \vec \sigma_t - \vec \lambda_t \rangle \ge \sum_{t \in W_j} \lVert 
\mathbf Q_{t-1} \rVert_1 + (KM^2 + \epsilon KM) \left( \lvert W_j \rvert - 1 \right)^2. \label{eq-proof-lemma-quad-lyapunov-theta-W_j}
\end{equation}
For any given time horizon $T \ge 1$, we can find a smallest $n$ such that $\bigcup_{j=0}^n W_j \supseteq [T]$, then the left most endpoint of $W_n$ is smaller than $T$. According to \Cref{eq-assumption-theta-C_W} in Assumption~\ref{assumption-theta}, we have $(\lvert W_n\rvert - 1)^2 \le C_W T$, hence $\lvert W_n\rvert \le \sqrt{\frac T {C_W}} + 1$. Let $\mathcal T_T \triangleq \min_{t\in W_n}$ be the rightmost endpoint of $W_n$, then $W_0,\ldots, W_n$ is a partition of $[\mathcal T_T]$, and we have
\begin{equation*}
    \mathcal T_T \le T + \lvert W_n \rvert \le T + \sqrt{\frac T {C_W}} + 1.
\end{equation*}

Summing \Cref{eq-proof-lemma-quad-lyapunov-theta-W_j} over $j=0\ldots n$, we get
\begin{align*}
    \sum_{t=1}^{\mathcal T_T} \langle \mathbf Q_{t-1}, \vec \theta_t \odot \vec \sigma_t - \vec \lambda_t \rangle & \ge \sum_{t=1}^{\mathcal T_T} \lVert 
\mathbf Q_{t-1} \rVert_1 + (KM^2 + \epsilon KM) \sum_{j=0}^n \left( \lvert W_j \rvert - 1 \right)^2 \\
& \stackrel{(a)}\ge \sum_{t=1}^{\mathcal T_T} \lVert 
\mathbf Q_{t-1} \rVert_1 + (KM^2 + \epsilon KM) C_W \mathcal T_T \\
& \ge \sum_{t=1}^T \lVert 
\mathbf Q_{t-1} \rVert_1 + (KM^2 + \epsilon KM) C_W \mathcal T_T
\end{align*}
where in $(a)$ we applied \Cref{eq-assumption-theta-C_W}.
\end{proof}

\begin{proof}[Proof of \Cref{prop-lyapunov-queue-length}]
\label{proof-prop-lyapunov-queue-length}
If suffices to note that when $T \ge \max\{\frac 4 {C_W}, C_W\}$, we have $\sqrt{\frac T{C_W}} \ge \sqrt{\frac{C_W}{C_W}} = 1$. Also, we have $T^2 \ge \frac{4T}{C_W}$ hence $T \ge 2\sqrt{\frac T {C_W}}$. Therefore, the constant $\mathcal T_T$ in \Cref{lemma-quad-lyapunov-theta} satisfies
\begin{align*}
    \mathcal T_T & \le T + \sqrt{\frac T{C_W}} + 1 \\
    & \le T + 2\sqrt{\frac T{C_W}} \\
    & \le 2T.
\end{align*}

Then, applying \Cref{lemma-quad-lyapunov} with the time horizon length $\mathcal T_T$, we get
\begin{equation}
    \mathbb E\left[\sum_{t=1}^{\mathcal T_T} Q_{t-1,a_t}S_{t,a_t} -\langle \mathbf Q_{t-1}, \vec \lambda_t\rangle\right] \le \frac{(K+1)M^2\mathcal T_T} 2. \label{eq-proof-prop-lyapunov-queue-length-1}
\end{equation}
By rearranging the terms in the bound given by \Cref{lemma-quad-lyapunov-theta}, we get
\begin{align}
    \epsilon \mathbb E \left[\sum_{t=1}^T \lVert\mathbf Q_{t-1}\rVert_1\right] & \le \mathbb E\left[\sum_{t=1}^{\mathcal T_T} \langle \mathbf Q_{t-1}, \vec \sigma_t \odot \vec \theta_t - \vec \lambda_t\rangle\right] + (KM^2 + \epsilon KM) C_W \mathcal T_T \nonumber \\
    & = \mathbb E\left[\sum_{t=1}^{\mathcal T_T} Q_{t-1,a_t}S_{t,a_t} -\langle \mathbf Q_{t-1}, \vec \lambda_t\rangle\right] + (KM^2 + \epsilon KM) C_W \mathcal T_T \nonumber \\
    & \quad - \mathbb E\left[\sum_{t=1}^{\mathcal T_T} \langle \mathbf Q_{t-1}, \vec \sigma_t \odot \vec \theta_t \rangle - Q_{t-1,a_t}S_{t,a_t}\right] \nonumber \\
    & \le \mathbb E\left[\sum_{t=1}^{\mathcal T_T} Q_{t-1,a_t}S_{t,a_t} -\langle \mathbf Q_{t-1}, \vec \lambda_t\rangle\right] + (KM^2 + \epsilon KM) C_W \mathcal T_T + f(\mathcal 
T_T), \label{eq-proof-prop-lyapunov-queue-length-2}
\end{align}
where the last inequality is due to the assumption that $\mathbb E\left[\sum_{t=1}^T Q_{t-1,a_t}S_{t,a_t} \right] \ge  \mathbb E\left[\sum_{t=1}^T \langle \mathbf Q_{t-1}, \vec \sigma_t \odot \vec \theta_t \rangle \right] - f(T)$.

Plugging \Cref{eq-proof-prop-lyapunov-queue-length-1} into \Cref{eq-proof-prop-lyapunov-queue-length-2}, we can see that
\begin{equation*}
    \epsilon \mathbb E \left[\sum_{t=1}^T \lVert\mathbf Q_{t-1}\rVert_1\right] \le \frac{(K+1)M^2\mathcal T_T} 2 + (KM^2 + \epsilon KM) C_W \mathcal T_T + f(\mathcal T_T)
\end{equation*}
and thus
\begin{align*}
    \frac 1 T \mathbb E\left[ 
\sum_{t=1}^{T} \lVert \mathbf Q_t\rVert_1  \right] & \le \frac {(K+1)M^2 + 2C_W (KM^2 + \epsilon KM) } {2\epsilon} \cdot \frac{\mathcal T_T}{T} + \frac {f(\mathcal T_T)}{\epsilon T} \\
& \le \frac {(K+1)M^2 + 2C_W (KM^2 + \epsilon KM) } {\epsilon} + \frac {f(2T)}{\epsilon T}
\end{align*}
where the last step is due to $\mathcal T_T \le 2T$ and the assumption that $f(\cdot)$ is increasing.
    
\end{proof}

\section{Proofs for \texttt{SoftMW} (Algorithm 2) Stability Analysis}

Here is a lemma similar to \Cref{lemma-bounded-diff-S} but in the opposite direction.
\begin{lemma}
\label{lemma-bounded-diff-S-rev}
Suppose $x_1=0$, $x_2,\ldots, x_n \ge 0$, $\lvert x_{i+1} - x_i \vert \le 1$ for all $1\le i < n$, then we have 
\begin{equation*}
    \sum_{i=1}^n x_i^2 \ge \frac 1 3 x_n^3.
\end{equation*}
\end{lemma}
\begin{proof}
    For $0\le i \le \lceil x_n \rceil$, we have $x_{n - i} \ge x_n - i \ge 0$, thus
    \begin{align*}
        \sum_{i=1}^n x_i^2 & \ge \sum_{i=0}^{\lceil x_n \rceil} x_{n - i}^2 \ge \sum_{i=0}^{\lceil x_n \rceil} (x_n-i)^2 \\
        & = \frac 1 6 x_n(x_n + 1)(2x_n + 1) - \frac 1 6 \{x_n\}(\{x_n\} + 1)(2\{x_n\} + 1) \\
        & \ge \frac 1 3 x_n^3
    \end{align*} 
    where $\{x_n\}\triangleq x_n - \lceil x_n \rceil$.
\end{proof}
\begin{corollary}
\label{corollary-bounded-diff-S-rev}
Let $\{\mathbf Q_t\}$ be queue lengths where $\mathbf Q_0 = \mathbf 0$, and the increments are bounded by $M$, then
\begin{equation*}
    \sum_{t=0}^{T-1} \lVert \mathbf Q_t\rVert_2^2 \ge \frac 1 {3MK^3} \lVert \mathbf Q_{T-1} \rVert_1^3
\end{equation*}
for any $T \ge 1$.
\end{corollary}
\begin{proof}
    Let $i^* \in \argmax_{i \in[K]} Q_{T-1,i}$, then
    \begin{equation*}
        \frac {Q_{0,i^*}} M, \frac {Q_{1,i^*}} M, \ldots, \frac {Q_{T-1,i^*}} M
    \end{equation*}
    is a non-negative sequence starting from $0$ with differences bounded by $1$. Applying \Cref{lemma-bounded-diff-S-rev} to this sequence, we can see
    \begin{align*}
        \sum_{t=0}^{T-1} \lVert \mathbf Q_t\rVert_2^2 & \ge M^2 \sum_{t=0}^{T-1} \left( \frac{Q_{t,i^*}} M \right) ^2 \\
        & \ge \frac {M^2} 3 \left( \frac{Q_{T-1,i^*}} M \right) ^3 \\
        & = \frac 1 {3M} Q_{T-1,i^*} ^3 \\
        & \ge \frac 1 {3M} \left(\frac {\lVert \mathbf Q_{T-1} \rVert_1} K\right)^3 \\
        & = \frac 1 {3MK^3} \lVert \mathbf Q_{T-1} \rVert_1^3
    \end{align*}
    as claimed.
\end{proof}

\label{sec-apdx-softmw}
\begin{proof}[Proof of \Cref{prop-softmw-exp-rates}]
\label{proof-prop-softmw-exp-rates}
    It suffices to verify that at each time step
    \begin{align*}
        2\lVert \mathbf Q_{t - 1}\rVert_1 & \le t^{-(\frac 1 4 - \frac \delta 2)} \sqrt{ 86M^2 K^6 t^{\frac 3 2} + \sum_{s=0}^{t-1} \lVert \mathbf Q_s \rVert_2^2}, \\
    & \Updownarrow \\
        4\lVert \mathbf Q_{t - 1}\rVert_1^2 & \le t^{-\frac 1 2 + \delta} \left( 86M^2 K^6 t^{\frac 3 2} + \sum_{s=0}^{t-1} \lVert \mathbf Q_s \rVert_2^2\right), \\
    & \Uparrow (a) \\
        4\lVert \mathbf Q_{t - 1}\rVert_1^2 & \le t^{-\frac 1 2 + \delta} \left(  86M^2 K^6 t^{\frac 3 2} + \frac 1 {3MK^3} \lVert \mathbf Q_{t -1} \rVert_1^3 \right) \\
    & \Updownarrow \\
        4\lVert \mathbf Q_{t - 1}\rVert_1^2 & \le t^{-\frac 1 2 + \delta} \left( \frac 1 3 \cdot 
 258 M^2 K^6 t^{\frac 3 2} + \frac 2 3 \cdot \frac 1 {2MK^3} \lVert \mathbf Q_{t -1} \rVert_1^3\right), \\
    & \Uparrow (b) \\
        4\lVert \mathbf Q_{t - 1}\rVert_1^2 & \le t^{-\frac 1 2 + \delta} \cdot 258^{\frac 1 3} 2^{-\frac 2 3} t^{\frac 1 2} \lVert \mathbf Q_{t - 1}\rVert_1^2 \\
    & \Updownarrow \\
        4 & \le t^\delta \left(\frac {258} 4\right)^{\frac 1 3},
    \end{align*}
    and the last statement trivially holds. Here in step $(a)$ we apply \Cref{corollary-bounded-diff-S-rev}, in step $(b)$ we apply AM-GM inequality $\frac 1 3 x + \frac 2 3 y \ge x^{\frac 1 3} y ^{\frac 2 3}$.
\end{proof}

\begin{proof}[Proof of \Cref{lemma-softmw-regret}]
\label{proof-lemma-softmw-regret}
    For each $\vec \theta_t$, we can find a $\vec\theta'_t \in \Delta^{[K], \beta_t}$ so that $\lVert \vec \theta_t - \vec \theta'_t \rVert_1 \le 2K \beta_t$ and $\lVert 
\vec \theta'_s - \vec \theta'_t \rVert_1 \le \lVert 
\vec \theta_s - \vec \theta_t \rVert_1$. For example, we can choose \begin{equation*}
    \vec \theta'_t = (1 - \beta_t) \vec \theta_t + \beta_t \mathbf 1.
\end{equation*} 
Then, we can write
\begin{align*}
    \langle \mathbf Q_{t-1}, \mathbf S_t \odot \vec \theta_t \rangle & = \langle \mathbf Q_{t-1}, \mathbf S_t \odot \vec \theta'_t \rangle + \langle \mathbf Q_{t-1} \odot \mathbf S_t, \vec \theta_t - \vec \theta'_t \rangle \\
    & \le \langle \mathbf Q_{t-1}, \mathbf S_t \odot \vec \theta'_t \rangle + \lVert 
\mathbf Q_{t-1} \odot \mathbf S_t \rVert_\infty \cdot \lVert \vec \theta_t - \vec \theta'_t \rVert_1 \\
& \le \langle \mathbf Q_{t-1}, \mathbf S_t \odot \vec \theta'_t \rangle + M^2t \cdot 2K\beta_t \\
& = \langle \mathbf Q_{t-1}, \mathbf S_t \odot \vec \theta'_t \rangle + 2M^2t^{-2}.
\end{align*}
Then, one can see the quantity
\begin{equation*}
    \sum_{t=1}^T \mathbb E\left[ \langle \mathbf Q_{t-1}, \mathbf S_t \odot \vec \theta'_t \rangle - Q_{t-1,a_t}S_{t,a_t} \right] = \sum_{t=1}^T \mathbb E\left[ \langle \mathbf Q_{t-1} \odot \mathbf S_t, \vec \theta'_t \rangle - (\mathbf Q_{t-1} \odot \mathbf S_t)_{a_t} \right]
\end{equation*}
satisfies the condition to apply \Cref{thm-exp3s}. \Cref{thm-exp3s} asserts that
\begin{align}
    & \quad \sum_{t=1}^T \mathbb E\left[ \langle \mathbf Q_{t-1}, \mathbf S_t \odot \vec \theta'_t \rangle - Q_{t-1,a_t}S_{t,a_t} \right] \nonumber \\
    & \le \left(1 + \sum_{t=1}^{T-1} \lVert \vec\theta'_{t+1} - \vec\theta'_t \rVert_1\right)\mathbb E\left[ \eta_T^{-1} \ln \frac 1 {\beta_T}\right] + e\mathbb E\left[\sum_{t=1}^T \eta_t \lVert 
\mathbf Q_{t-1}\odot \mathbf S_t \rVert_2^2\right] \nonumber \\
& \quad + \mathbb E\left[\sum_{t=1}^T\gamma_t \left\langle \mathbf Q_{t-1}\odot \mathbf S_t, \frac{\mathbf Q_{t-1}}{\lVert \mathbf Q_{t-1}\rVert_1} \right\rangle\right].\label{eq-proof-lemma-softmw-regret-1}
\end{align}
Below, we bound each term in the RHS of \Cref{eq-proof-lemma-softmw-regret-1}. Firstly, we have $\sum_{t=1}^{T-1} \lVert \vec\theta'_{t+1} - \vec\theta'_t \rVert_1 \le \sum_{t=1}^{T-1} \lVert \vec\theta_{t+1} - \vec\theta_t \rVert_1 \le C_V T^{\frac 1 2 - \delta}$ according to Assumption~\ref{assumption-delta}, hence
\begin{align*}
    & \quad \left(1 + \sum_{t=1}^{T-1} \lVert \vec\theta'_{t+1} - \vec\theta'_t \rVert_1\right)\mathbb E\left[ \eta_T^{-1} \ln \frac 1 {\beta_T}\right] \\ 
    & \le \left(1 + C_V T^{\frac 1 2 -\delta}\right)  \mathbb E\left[ \eta_T^{-1} \ln \frac 1 {\beta_T}\right] \\
    & \le \left(1 + C_V T^{\frac 1 2 -\delta}\right) \left( 3\ln T + \ln K\right) \mathbb E\left[ T^{-(\frac 1 4 - \frac \delta 2)} M \sqrt{ 86M^2 K^6 T^{\frac 3 2} + \sum_{s=0}^{T -1} \lVert \mathbf Q_s \rVert_2^2}\right] \\
    & \le \left(1 + C_V\right) T^{\frac 1 4 - \frac \delta 2} \left( 3\ln T + \ln K\right) M \mathbb E\left[\sqrt{ 86M^2 K^6 T^{\frac 3 2} + \sum_{s=0}^{T -1} \lVert \mathbf Q_s \rVert_2^2}\right].
\end{align*}

For the second term, we have
\begin{align*}
    & \quad \sum_{t=1}^T \eta_t \lVert 
\mathbf Q_{t-1}\odot \mathbf S_t \rVert_2^2 \\
& \le M^2 \sum_{t=1}^T \eta_t \lVert 
\mathbf Q_{t-1} \rVert_2^2 \\
& = M \sum_{t=1}^T t^{\frac 1 4 - \frac \delta 2} \left(\sqrt{ 86M^2 K^6 t^{\frac 3 2} + \sum_{s=0}^{t -1} \lVert \mathbf Q_s \rVert_2^2}\right)^{-1} \cdot \lVert 
\mathbf Q_{t-1} \rVert_2^2 \\
& \le M \sum_{t=1}^T t^{\frac 1 4 - \frac \delta 2} \left(\sqrt{ 1 + \sum_{s=0}^{t -1} \lVert \mathbf Q_s \rVert_2^2}\right)^{-1} \cdot \lVert 
\mathbf Q_{t-1} \rVert_2^2 \\
& \le M T^{\frac 1 4 - \frac \delta 2}  \sum_{t=1}^T \left(\sqrt{ 1 + \sum_{s=0}^{t -1} \lVert \mathbf Q_s \rVert_2^2}\right)^{-1} \cdot \lVert 
\mathbf Q_{t-1} \rVert_2^2 \\
& \le 2 M T^{\frac 1 4 - \frac \delta 2} \sqrt{1 + \sum_{s=0}^{T-1}\lVert \mathbf Q_s \rVert_2^2},
\end{align*}
where in the last step, we use the fact that
\begin{equation*}
    \sum_{i=1}^n \frac {x_i}{\sqrt{1 + \sum_{j=1}^i x_j}} \le 2 \sqrt{1 + \sum_{i=1}^n x_i}
\end{equation*}
for non-negative $x_1\ldots, x_n$.

For the third term, we have
\begin{align*}
    & \quad \sum_{t=1}^T\gamma_t \left\langle \mathbf Q_{t-1}\odot \mathbf S_t, \frac{\mathbf Q_{t-1}}{\lVert \mathbf Q_{t-1}\rVert_1} \right\rangle \\
    & \le  M \sum_{t=1}^T \gamma_t \lVert \mathbf Q_{t-1}\rVert_1^{-1} \lVert \mathbf Q_{t-1}\rVert_2^2 \\
    & = M \sum_{t=1}^T t^{\frac 1 4 - \frac \delta 2} \left(\sqrt{ 1 + \sum_{s=0}^{t -1} \lVert \mathbf Q_s \rVert_2^2}\right)^{-1} \cdot \lVert 
\mathbf Q_{t-1} \rVert_2^2 
\end{align*}
then it can be bounded in the way exactly same as the second term, hence it is also no more than
\begin{equation*}
    2 M T^{\frac 1 4 - \frac \delta 2} \sqrt{1 + \sum_{s=0}^{T-1}\lVert \mathbf Q_s \rVert_2^2}.
\end{equation*}
Combining everything together then taking expectation, we can see that
\begin{align*}
    & \quad \sum_{t=1}^T \mathbb E\left[ \langle \mathbf Q_{t-1}, \mathbf S_t \odot \vec \theta'_t \rangle - Q_{t-1,a_t}S_{t,a_t} \right] \\
    & \le T^{\frac 1 4 - \frac \delta 2} M \mathbb E\left[\sqrt{ 86M^2 K^6 T^{\frac 3 2} + \sum_{s=0}^{T -1} \lVert \mathbf Q_s \rVert_2^2}\right] \cdot \left\{\left(1 + C_V\right) \left( 3\ln T + \ln K\right) + 2e + 2\right\} + 2M^2 \cdot \frac{\pi^2} 6 \\
    & \le (2e + 3)\cdot(1 + C_V) T^{\frac 1 4 - \frac \delta 2} (3\ln T + \ln K) M \mathbb E\left[\sqrt{ 86M^2 K^6 T^{\frac 3 2} + \sum_{s=0}^{T -1} \lVert \mathbf Q_s \rVert_2^2}\right] + 4M^2 \\
    & \le 9(1 + C_V) T^{\frac 1 4 - \frac \delta 2} (3\ln T + \ln K) M \mathbb E\left[\sqrt{ 86M^2 K^6 T^{\frac 3 2} + \sum_{s=0}^{T -1} \lVert \mathbf Q_s \rVert_2^2}\right] + 4M^2.
\end{align*}
\end{proof}

\begin{proof}[Proof of \Cref{lemma-bounded-diff-S}]
\label{proof-lemma-bounded-diff-S}
Sort $x_1,\ldots, x_n$ in acsending order to $y_1 \le y_2 \le \cdots \le y_n$. We then claim that $y_1,\ldots,y_n$ is also a sequence of non-negative numbers with increments bounded by $1$, i.e., $\lvert y_{i+1} - y_i \rvert \le 1$ for all $1 \le i < n$.

Let $\sigma: [n] \rightarrow [n]$ be a permutation over $[n]$ such that $y_i = x_{\sigma(i)}$ for any $1\le i \le n$. For any $1 \le i < n$, without loss of generality, assume $\sigma(i+1) > \sigma(i)$, let
\begin{equation*}
    j \triangleq \min\left\{k : \sigma(i) < k \le \sigma(i+1), x_k \ge x_{\sigma(i)}\right\}.
\end{equation*}
Since we have assumed $\sigma(i+1) > \sigma(i)$, $j$ is guaranteed to be properly defined. According to the definition of $j$ and the fact that $\{y_i\}$ is a sorted list of $\{x_i\}$, we have
\begin{equation*}
    x_{j-1} \stackrel{(a)}\le x_{\sigma(i)} = y_i \le y_{i+1} = x_{\sigma(i+1)} \stackrel{(b)}\le x_j,
\end{equation*}
where the inequality $(a)$ is due to the definition of $j$ ($j$ is the index of the first element in $\{x_i\}$ after $x_{\sigma(i)}$ that is above $y_i$), and $(b)$ is due to $x_j \ge y_i$, and$y_{i+1}$ is the smallest element among the numbers above $y_i$. Therefore,
\begin{equation*}
    \lvert y_{i+1} - y_i \rvert = y_{i+1} - y_i \le x_j - x_{j-1} \le \lvert x_j - x_{j-1} \rvert \le 1. 
\end{equation*}
The case where $\sigma(i+1) < \sigma(i)$ can be similarly verified.

Define
\begin{equation*}
    s \triangleq \lceil y_n \rceil = \left\lceil \max_{1\le i \le n} x_i \right\rceil,
\end{equation*}
then we have $y_{n - i} \ge y_n - i \ge 0$ for all $0 \le i < s$, hence
\begin{equation*}
    S = \sum_{i=1}^n y_i \ge \sum_{i=0}^{s-1} y_{n-i} \ge \sum_{i=0}^{s-1} (y_n - i) = \frac {s(2y_n - s + 1)} 2 \ge \frac {y_n^2} 2,
\end{equation*}
where the last inequality is due to $y_n \le s \le y_n + 1$, hence $2y_n - s + 1\ge y_n$. Therefore $y_n \le \sqrt {2S}$. Then, it is easy to conclude that
\begin{align*}
    \sum_{i=1}^n x_i^2 \le y_n \sum_{i=1}^n x_i \le \sqrt{2S} \cdot S = \sqrt 2 S^{\frac 3 2}.
\end{align*}
\end{proof}

\begin{proof}[Proof of \Cref{lemma-444}]
\label{proof-lemma-444}
    Let $z(x) \triangleq y(x)^{\frac 1 4}$, then we have
    \begin{equation*}
        z(x)^4 \le f(x)+z(x)g(x)
    \end{equation*}
    for all $x\ge 0$. Note that for any fixed choice of $x\ge 0$, the function
    \begin{equation*}
        h(z)\triangleq z^4 - g(x)z - f(x)
    \end{equation*}
    is increasing on $\left[\left(\frac {g(x)} 4\right)^{\frac 1 3}, \infty\right)$.
    In order to prove \Cref{lemma-444}, it suffices to show that the particular choice of $z = f(x)^{\frac 1 4} + g(x)$ guarantees $z \ge \left(\frac {g(x)} 4\right)^{\frac 1 3}$ and $h(z) \ge 0$.
    The condition $z \ge \left(\frac {g(x)} 4\right)^{\frac 1 3}$. First, the condition $z \ge \left(\frac {g(x)} 4\right)^{\frac 1 3}$ holds trivially since we have assumed that $f(x)$ and $g(x)$ are both no less than $1$. Then, we can do a direct calculation
    \begin{align*}
        z^4 & = \left( 
f(x)^{\frac 1 4} + g(x) \right)^4 \\
    & = f(x) + 4f(x)^{\frac 3 4}g(x) + 6f(x)^{\frac 1 2}g(x)^2 + 4f(x)^{\frac 1 4}g(x)^3 + g(x)^4,
    \end{align*}
    thus,
    \begin{align*}
        h(z) & = z^4 - g(x)\left( 
f(x)^{\frac 1 4} + g(x) \right) -f(x)\\
& = \left(4f(x)^{\frac 3 4}g(x) - f(x)^{\frac 1 4}g(x)\right) + \left(6f(x)^{\frac 1 2}g(x)^2 - g(x)^2\right) + 4f(x)^{\frac 1 4}g(x)^3 + g(x)^4
    \end{align*}
    where we can see all terms are non-negative since $f(x),g(x) \ge 1$.
\end{proof}

\section{Proofs for \texttt{SSMW} (Algorithm 3) Stability Analysis}

\begin{proof}[Proof of \Cref{lemma-ssmw-epoch}]
\label{proof-lemma-ssmw-epoch}

Similar to the proof of \Cref{lemma-softmw-regret}, it suffices to apply \Cref{thm-exp3s}, but we need to verify that all $\gamma_\tau$'s are no more than $\frac 1 2$ first.

Our choice of $m$ guarantees that $m\ge \frac{\lVert 
\mathbf Q_{T_0} \rVert_\infty}{2M}$, hence $\lVert 
\mathbf Q_{T_0} \rVert_\infty \le 2Mm$, and for any $\tau$ inside this epoch, we have $\lVert 
\mathbf Q_{T_0 + \tau} \rVert_\infty \le \lVert 
\mathbf Q_{T_0} \rVert_\infty + M \tau \le 3Mm$. Therefore
\begin{align*}
    \gamma_\tau & = \frac 1 6 K^{-1}M^{-1}m^{-1-\frac \delta 2}\lVert\mathbf Q_{T_0 + \tau - 1}\rVert_1 \\
    & \le \frac 1 6 K^{-1}M^{-1}m^{-1-\frac \delta 2} \cdot 3KMm \\
    & = \frac 1 2 m^{-\frac \delta 2} \\
    & \le \frac 1 2.
\end{align*}

For each $\vec \theta_{T_0 + \tau}$, we can find a $\vec\theta'_{T_0 + \tau} \in \Delta^{[K], \beta_\tau}$ so that $\lVert \vec \theta_{T_0 + \tau} - \vec \theta'_{T_0 + \tau} \rVert_1 \le 2K \beta_\tau$ and $\lVert 
\vec \theta'_s - \vec \theta'_t \rVert_1 \le \lVert 
\vec \theta_{T_0 + s} - \vec \theta_{T_0 + t} \rVert_1$ for any $1\le s,t\le m$. For example, we can choose \begin{equation*}
    \vec \theta'_{T_0 + \tau} = (1 - \beta_\tau) \vec \theta_{T_0 + \tau} + \beta_\tau \mathbf 1.
\end{equation*} 
Then, we can write
\begin{align*}
    \langle \mathbf Q_{T_0 + \tau - 1}, \mathbf S_{T_0 + \tau} \odot \vec \theta_{T_0 + \tau} \rangle & = \langle \mathbf Q_{T_0 + \tau-1}, \mathbf S_{T_0 + \tau} \odot \vec \theta'_{T_0 + \tau} \rangle + \langle \mathbf Q_{T_0 + \tau-1} \odot \mathbf S_{T_0 + \tau}, \vec \theta_{T_0 + \tau} - \vec \theta'_{T_0 + \tau} \rangle \\
    & \le \langle \mathbf Q_{T_0 + \tau-1}, \mathbf S_{T_0 + \tau} \odot \vec \theta'_{T_0 + \tau} \rangle + \lVert 
\mathbf Q_{T_0 + \tau-1} \odot \mathbf S_{T_0 + \tau} \rVert_\infty \cdot \lVert \vec \theta_{T_0 + \tau} - \vec \theta'_{T_0 + \tau} \rVert_1 \\
& \le \langle \mathbf Q_{T_0 + \tau-1}, \mathbf S_{T_0 + \tau} \odot \vec \theta'_{T_0 + \tau} \rangle + 3Mm \cdot M \cdot 2K\beta_\tau \\
& = \langle \mathbf Q_{T_0 + \tau-1}, \mathbf S_{T_0 + \tau} \odot \vec \theta'_{T_0 + \tau} \rangle + 6M^2m^{-1}.
\end{align*}

Then, one can see the quantity
\begin{align*}
    & \quad \sum_{\tau=1}^m \mathbb E\left[\left. \langle \mathbf Q_{T_0 + \tau-1}, \mathbf S_{T_0 + \tau} \odot \vec \theta'_{T_0 + \tau} \rangle - Q_{T_0 + \tau-1,a_{T_0 + \tau}}S_{T_0 + \tau,a_{T_0 + \tau}}\right \rvert \mathcal F_{T_0} \right] \\
    & = \sum_{\tau=1}^m \mathbb E\left[\left. \langle \mathbf Q_{T_0 + \tau-1} \odot \mathbf S_{T_0 + \tau}, \vec \theta'_{T_0 + \tau} \rangle - (\mathbf Q_{T_0 + \tau-1} \odot \mathbf S_{T_0 + \tau})_{a_{T_0 + \tau}}\right \rvert \mathcal F_{T_0} \right]
\end{align*}
satisfies the condition to apply \Cref{thm-exp3s}. \Cref{thm-exp3s} asserts that
\begin{align}
    & \quad \sum_{\tau=1}^m \mathbb E\left[\left. \langle \mathbf Q_{T_0 + \tau-1}, \mathbf S_{T_0 + \tau} \odot \vec \theta'_{T_0 + \tau} \rangle - Q_{T_0 + \tau-1,a_{T_0 + \tau}}S_{T_0 + \tau,a_{T_0 + \tau}}\right \rvert \mathcal F_{T_0} \right]\nonumber \\
    & \le \left(1 + \sum_{\tau=1}^{m-1} \lVert \vec\theta'_{T_0 + \tau+1} - \vec\theta'_{T_0 + \tau} \rVert_1\right)\mathbb E\left[\left. \eta_m^{-1} \ln \frac 1 {\beta_m}\right \rvert \mathcal F_{T_0}\right] + e\mathbb E\left[\left.\sum_{\tau=1}^m \eta_\tau \lVert 
\mathbf Q_{T_0 + \tau-1}\odot \mathbf S_{T_0 + \tau} \rVert_2^2\right\rvert \mathcal F_{T_0}\right] \nonumber \\
& \quad + \mathbb E\left[\left.\sum_{\tau=1}^m\gamma_\tau \left\langle \mathbf Q_{T_0 + \tau-1}\odot \mathbf S_{T_0 + \tau}, \frac{\mathbf Q_{T_0 + \tau-1}}{\lVert \mathbf Q_{T_0 + \tau-1}\rVert_1} \right\rangle\right\vert \mathcal F_{T_0}\right]\nonumber \\
& \le \left(1 + \sum_{\tau=1}^{m-1} \lVert \vec\theta'_{T_0 + \tau+1} - \vec\theta'_{T_0 + \tau} \rVert_1\right) \cdot 6 M^2Km^{1+\frac \delta 2} \cdot \left( 2\ln m + \ln K \right) \nonumber \\
& \quad + (e+1) \frac 1 6 K^{-1}m^{-1-\frac \delta 2}\mathbb E\left[\left.\sum_{\tau=1}^m \lVert 
\mathbf Q_{T_0 + \tau-1}\rVert_2^2\right\rvert \mathcal F_{T_0}\right]\nonumber \\
& \le 6 \left(1 + C_V m^{1-\delta}\right) \cdot M^2Km^{1+\frac \delta 2} \cdot \left( 2\ln m + \ln K \right) \nonumber \\
& \quad + (e+1) \frac 1 6 K^{-1}m^{-1-\frac \delta 2}\mathbb E\left[\left.\sum_{\tau=1}^m \lVert 
\mathbf Q_{T_0 + \tau-1}\rVert_2^2\right\rvert \mathcal F_{T_0}\right]\nonumber,
\end{align}
where in the last step, we use the bound for $\sum_{\tau=1}^{m-1} \lVert \vec\theta'_{T_0 + \tau+1} - \vec\theta'_{T_0 + \tau} \rVert_1$ in Assumption~\ref{assumption-delta3}.

Therefore, we have
\begin{align*}
    & \quad \sum_{\tau=1}^m \mathbb E\left[\left. \langle \mathbf Q_{T_0 + \tau-1}, \mathbf S_{T_0 + \tau} \odot \vec \theta_{T_0 + \tau} \rangle - Q_{T_0 + \tau-1,a_{T_0 + \tau}}S_{T_0 + \tau,a_{T_0 + \tau}}\right \rvert \mathcal F_{T_0} \right] \\
    & \le 6\left(1 + C_V\right) M^2Km^{2-\frac \delta 2} \cdot \left( 2\ln m + \ln K \right) + K^{-1}m^{-1-\frac \delta 2}\mathbb E\left[\left.\sum_{\tau=1}^m \lVert 
\mathbf Q_{T_0 + \tau-1}\rVert_2^2\right\rvert \mathcal F_{T_0}\right] + 6M^2.
\end{align*}
\end{proof}

\begin{proof}[Proof of \Cref{lemma-ssmw-sum-l12norm}]
\label{proof-lemma-ssmw-sum-l12norm}
    These inequalities are all direct implications of the bounded-increment assumption. Denote by $i^\ast \in \argmax_{i\in [K]} Q_{T_0, i}$, then we have
    \begin{align*}
        Q_{t,i^*} & \stackrel{(a)}\ge Q_{T_0, i^*} - (t-T_0)\cdot M \\
        & \ge Q_{T_0, i^*} - \left(\frac {\lVert \mathbf Q_{T_0}\rVert_\infty} {2M} + 1\right) \cdot M \\
        & = \frac {\lVert \mathbf Q_{T_0}\rVert_\infty} 2 - M \\
        & \stackrel{(b)}\ge \frac {\lVert \mathbf Q_{T_0}\rVert_\infty} 4
    \end{align*}
    for all $T_0 + 1\le t \le T_0 + \frac {\lVert \mathbf Q_{T_0}\rVert_\infty} {2M} + 1$. Here step $(a)$ is due to the boundedness of queue length increments, step $(b)$ is due to the assumption that $\lVert \mathbf Q_{T_0}\rVert_\infty \ge 4M$. Thus, for these $t$'s we have
    \begin{equation*}
        \lVert \mathbf Q_t\rVert_\infty \ge Q_{t,i^*} \ge \frac {\lVert \mathbf Q_{T_0}\rVert_\infty} 4.
    \end{equation*}

    Also, for any $T_0 + 1\le t \le T_0 + \frac {\lVert \mathbf Q_{T_0}\rVert_\infty} {2M} + 1$ and $i\in [K]$, we have
        \begin{align*}
        Q_{t,i} & \stackrel{(a)}\le Q_{T_0, i} + (t-T_0)\cdot M \\
        & \le \lVert \mathbf Q_{T_0}\rVert_\infty + \left(\frac {\lVert \mathbf Q_{T_0}\rVert_\infty} {2M} + 1\right) \cdot M \\
        & = \frac 3 2 \lVert \mathbf Q_{T_0}\rVert_\infty + M \\
        & \stackrel{(b)}\le 2 \lVert \mathbf Q_{T_0}\rVert_\infty.
    \end{align*}
    Here similarly, step $(a)$ is due to the boundedness of queue length increments, step $(b)$ is due to the assumption that $\lVert \mathbf Q_{T_0}\rVert_\infty \ge 4M$. Therefore, 
    \begin{equation*}
        \lVert \mathbf Q_t\rVert_\infty = \max_{i\in [K]} Q_{t,i} \le 2 \lVert \mathbf Q_{T_0}\rVert_\infty.
    \end{equation*}
\end{proof}

\begin{proof}[Proof of \Cref{lemma-ssmw-epoch-l1}]
\label{proof-lemma-ssmw-epoch-l1}
According to \Cref{lemma-ssmw-sum-l12norm}, when $m = \lceil \frac {\lVert \mathbf Q_{T_0}\rVert_\infty} {2M} \rceil$ we have
\begin{equation}
   Mm^2 \le 4\sum_{\tau=1}^m \lVert \mathbf Q_{T_0 + \tau -1} \rVert_1 \label{eq-proof-lemma-ssmw-sum-l12norm-1}
\end{equation}
and
\begin{equation}
    \sum_{\tau=1}^m \lVert \mathbf Q_{T_0 + \tau-1} \rVert_2^2 \le 16KM^2m^3 = 16KMm \cdot Mm^2\le 64KMm \sum_{\tau=1}^m \lVert \mathbf Q_{T_0 + \tau -1} \rVert_1 \label{eq-proof-lemma-ssmw-sum-l12norm-2}.
\end{equation}
Then, we can apply the two inequalities to \Cref{eq-lemma-ssmw-epoch} of \Cref{lemma-ssmw-epoch}, to obtain an upper-bound in $\sum_{\tau=1}^m \lVert \mathbf Q_{T_0 + \tau -1} \rVert_1$. Specifically, for the $6\left(1 + C_V\right) M^2Km^{2-\frac \delta 2} \cdot \left( 2\ln m + \ln K \right)$ term in the RHS of \Cref{eq-lemma-ssmw-epoch}, we can apply \Cref{eq-proof-lemma-ssmw-sum-l12norm-1} to upper-bound one $Mm^2$ factor by $4\sum_{\tau=1}^m \lVert \mathbf Q_{T_0 + \tau -1} \rVert_1$, hence the term is no more than
\begin{equation*}
    24\left(1 + C_V\right) MKm^{-\frac \delta 2} \cdot \left( 2\ln m + \ln K \right) \cdot \mathbb E\left[\left. \sum_{\tau=1}^m \lVert \mathbf Q_{T_0 + \tau -1} \rVert_1\right\rvert \mathcal F_{T_0}\right].
\end{equation*}
For the $K^{-1}m^{-1-\frac \delta 2}\mathbb E\left[\left.\sum_{\tau=1}^m \lVert 
\mathbf Q_{T_0 + \tau-1}\rVert_2^2\right\rvert \mathcal F_{T_0}\right]$ term, we can upper-bound the $\sqrt{\sum_{\tau=1}^{m} \lVert \mathbf Q_{T_0 + \tau-1} \rVert_2^2}$ factor (in each sample path) by $64KMm \sum_{\tau=1}^m \lVert \mathbf Q_{T_0 + \tau -1} \rVert_1$. Hence, the term is no more than
\begin{equation*}
    64Mm^{-\frac \delta 2}\mathbb E\left[\left.\sum_{\tau=1}^m \lVert \mathbf Q_{T_0 + \tau -1} \rVert_1\right\rvert \mathcal F_{T_0}\right].
\end{equation*}

Putting the upper-bounds for the two terms, we can get the claimed upper-bound.
\end{proof}

\section{High-Level Ideas to Handle Queue Length Increments with Bounded Moments}
\label{sec-adpx-moment-idea}

Here, we explain the intuition behind the design of \texttt{SoftMW+} and \texttt{SSMW+}. Recall that in the design and analysis of \texttt{SoftMW} and \texttt{SSMW}, the original assumption of bounded increments is mainly utilized in the following two steps:
\begin{itemize}
    \item \textbf{(\texttt{EXP3.S+} feedback signal scale)} In order to apply \Cref{thm-exp3s}, it must be guaranteed that each feedback value $g_{t,a_t}$ to be fed into \texttt{EXP3.S+} is no more than $\eta_t^{-1}\gamma_t e_{t,a_t}$. This is the main reason for us to set $\mathbf e_t$ to the normalized current queue length vector, and set $\gamma_t = \eta_t \cdot M\lVert \mathbf Q_t \rVert_1$ in the bounded increment case.
    \item \textbf{(Bounding \texttt{EXP3.S+} regret by $\sum \lVert \mathbf Q_t \rVert_1$)} After successfully applying \Cref{thm-exp3s} (and possibly summing over \texttt{EXP3.S+} instances), we obtain an upper-bound on the expected quadratic Lynapunov function value at the end of time $T$. Compared to the bound we can obtain in the stationary setting with \texttt{Max-Weight} via a standard Lyapunov analysis, our bound contains an additional term with a $\sqrt{\sum \lVert\mathbf Q_t\rVert_2^2}$ factor. We then apply \Cref{lemma-bounded-diff-S} for \texttt{SoftMW} and \Cref{lemma-ssmw-sum-l12norm} for \texttt{SSMW} to upper-bound the $\sqrt{\sum \lVert\mathbf Q_t\rVert_2^2}$ by expressions in $\sum \lVert\mathbf Q_t\rVert_1$.
\end{itemize}
Therefore, if we want to handle potentially unbounded queue length increments with only bounded moments, we need to generalize the arguments of the above two steps. Below, we introduce the high-level ideas to overcome the challenge introduced by unbounded increments.

\vspace{0.6ex}
\noindent \textbf{Skipping large feedback values from \texttt{EXP3.S+} }

\noindent Suppose we are now at time $t$ and choose $\mathbf e_t = \mathbf Q_{t-1} / \lVert 
\mathbf Q_{t-1} \rVert_1$ and $\gamma_t = C_t\cdot M \eta_t \lVert 
\mathbf Q_{t-1} \rVert_1$ in the next \texttt{EXP3.S+} step, where $C_t > 0$ is some constant to be determined at the end of time $t-1$. Then \texttt{EXP3.S+} can handle a new feedback signal no more than $C_t\cdot M Q_{t-1,a_t}$, where $a_t$ is the index of the queue we finally choose to serve in this time step. In other words, we can feed the signal to \texttt{EXP3} only if $S_{t,a_t} \le C_t M$. Since our goal of applying \texttt{EXP3.S+} is to make the sum of $\langle \mathbf Q_{t-1} \odot \mathbf S_t, \vec \theta_t\rangle - Q_{t-1,a_t} S_{t,a_t}$ small, we can write
\begin{align*}
    & \quad \langle \mathbf Q_{t-1} \odot \mathbf S_t, \vec \theta_t\rangle - Q_{t-1,a_t} S_{t,a_t} \le \mathbbm 1[S_{t,a_t} > C_tM] \langle \mathbf Q_{t-1} \odot \mathbf S_t, \vec \theta_t\rangle +  \langle \mathbf Q_{t-1} \odot \mathbf S'_t, \vec \theta_t\rangle - Q_{t-1,a_t} S'_{t,a_t} 
\end{align*}
where
\begin{equation*}
    S'_{t,i} = \begin{cases}
        S_{t,i} & \text{if } S_{t,i} \le C_tM \\
        0 & \text{otherwise}
    \end{cases}
\end{equation*}
is a truncation of $S_{t,i}$. Then, $\langle \mathbf Q_{t-1} \odot \mathbf S'_t, \vec \theta_t\rangle - Q_{t-1,a_t} S'_{t,a_t}$ is a quantity with an expected cumulative sum that can be regarded as the regret of a MAB problem, and thus can be controlled by \texttt{EXP3.S+}, 
if we take $Q_{t-1,a_t}S'_{t,a_t}$ as the feedback value. Now there will be a multiplicative factor $\O(\max_{1\le t\le T} C_t)$ in the obtained regret upper-bound. 
Also note that
\begin{align*}
    & \quad \mathbb E\left[\left.\mathbbm 1[S_{t,a_t} > C_tM] \langle \mathbf Q_{t-1} \odot \mathbf S_t, \vec \theta_t\rangle\right\rvert 
\mathcal F_{t-1} \right] \\
    & = \mathbb E\left[\left.\mathbbm 1[S_{t,a_t} > C_tM] \sum_{i\ne a_t} Q_{t-1,i}S_{t,i}\theta_{t,i}\right\rvert 
\mathcal F_{t-1} \right] + \mathbb E\left[\left.\mathbbm 1[S_{t,a_t} > C_tM] Q_{t-1,a_t}S_{t,a_t}\theta_{t,a_t} \right\rvert \mathcal F_{t-1} \right]\\
    & \stackrel{(a)}= \mathbb E\left[\left.\mathbb P\left[\left.S_{t,a_t} > C_tM\right\rvert \mathcal F_{t-1},a_t\right] \sum_{i\ne a_t} Q_{t-1,i}\sigma_{t,i}\theta_{t,i}\right\rvert 
\mathcal F_{t-1} \right] + \mathbb E\left[\left.\mathbbm 1[S_{t,a_t} > C_tM] Q_{t-1,a_t}S_{t,a_t}\theta_{t,a_t} \right\rvert \mathcal F_{t-1} \right]\\
    & \le \mathbb P\left[\left.S_{t,a_t} > C_tM\right\rvert \mathcal F_{t-1}\right]\langle \mathbf Q_{t-1}\odot \vec \sigma_t, \vec\theta_t \rangle + \mathbb E\left[\left.\mathbbm 1[S_{t,a_t} > C_tM] Q_{t-1,a_t}S_{t,a_t}\theta_{t,a_t} \right\rvert \mathcal F_{t-1} \right] \\
    & \le \mathbb P\left[\left.S_{t,a_t} > C_tM\right\rvert \mathcal F_{t-1}\right]\langle \mathbf Q_{t-1}\odot \vec \sigma_t, \vec\theta_t \rangle + \lVert \mathbf Q_{t-1}\rVert_\infty \mathbb E\left[\left.\mathbbm 1[S_{t,a_t} > C_tM] S_{t,a_t}  \right\rvert \mathcal F_{t-1}\right] \\
    & \stackrel{(b)}\le C_t^{-\alpha}\langle \mathbf Q_{t-1}\odot \vec \sigma_t, \vec\theta_t \rangle + C_t^{-\alpha + 1}M\lVert \mathbf Q_{t-1}\rVert_\infty \\
    & \le 2C_t^{-\alpha + 1}M\lVert \mathbf Q_{t-1}\rVert_\infty.
\end{align*}
Here in step $(a)$, we take the conditional expectation of $\mathbbm 1[S_{t,a_t} > C_tM] \sum_{i\ne a_t} Q_{t-1,i}S_{t,i}\theta_{t,i}$ with respect to $a_t$ before taking the conditional expectation with respect to $\mathcal F_{t-1}$. In step $(b)$, we make use of the fact that $\mathbb P\left[\left.S_{t,a_t} > C_tM\right\rvert \mathcal F_{t-1}\right] \le C_t^{-\alpha}$ and $\mathbb E\left[\left.\mathbbm 1[S_{t,a_t} > C_tM] S_{t,a_t}  \right\rvert \mathcal F_{t-1}\right] \le \mathbb E[\mathbbm 1[S_{t,a_t} > C_tM] S_{t,a_t}^\alpha / (C_tM)^{\alpha - 1}\rvert \mathcal F_{t-1}] \le (C_tM)^{-\alpha + 1}\mathbb E[S_{t,a_t}^\alpha \rvert \mathcal F_{t-1}] \le MC_t^{-\alpha + 1}$. Thus, the expected cumulative sum of $\mathbbm 1[S_{t,a_t} > C_tM] \langle \mathbf Q_{t-1} \odot \mathbf S_t, \vec \theta_t\rangle$ can also be well-controlled as long as we choose sufficiently small $C_t$s.

By choosing appropriate $C_t$s, we can trade-off the clipping error $\sum \mathbb E\left[\left.\mathbbm 1[S_{t,a_t} > C_tM] \langle \mathbf Q_{t-1} \odot \mathbf S_t, \vec \theta_t\rangle\right. \right]$ and the post-clipping MAB regret $\sum \mathbb E\left[\langle \mathbf Q_{t-1} \odot \mathbf S'_t, \vec \theta_t\rangle - Q_{t-1,a_t} S'_{t,a_t}\right]$ well to make their sum, i.e., the pre-clipping MAB regret $\sum \mathbb E\left[\langle \mathbf Q_{t-1} \odot \mathbf S_t, \vec \theta_t\rangle - Q_{t-1,a_t} S_{t,a_t}\right]$, not too large.


\vspace{0.6ex}
\noindent \textbf{General conversion lemmas between $\sum \lVert \mathbb Q_t \rVert_1$ and $\sum \lVert \mathbb Q_t \rVert_2^2$}

So far we have dealt with issues of unbounded feedback values when applying regret bounds of \texttt{EXP3.S+}. It remains to bound $\sum \lVert \mathbb Q_t \rVert_2^2$ by $\sum \lVert \mathbb Q_t \rVert_1$ so that it can finally solve to an upper-bound of $\sum \lVert \mathbb Q_t \rVert_1$ on the whole time-horizon. 

In fact, \Cref{lemma-bounded-diff-S} and \Cref{lemma-ssmw-sum-l12norm} are both \textit{sample-path} bounds for $\sum \lVert \mathbb Q_t \rVert_2^2$ which hold under \textit{arbitrary} scheduling policies. Specifically, \Cref{lemma-bounded-diff-S} states that $\sum \lVert \mathbb Q_t \rVert_2^2$ is $\O(\sum \lVert \mathbb Q_t \rVert_1)^{\frac 3 2}$ on the entire time-horizon, whereas \Cref{lemma-ssmw-sum-l12norm} describes how $\sum \lVert \mathbb Q_t \rVert_1$ and $\sum \lVert \mathbb Q_t \rVert_2^2$ can be used to bound each other on any successive $m=\Theta(\lVert\mathbf Q_{T_0}\rVert_\infty /M)$ time steps (where $T_0$ is the first time step of the $m$ time steps considered). On the other hand, our analysis actually only needs a relationship between the expectations $\mathbb E\left[\sqrt{\sum \lVert \mathbb Q_t \rVert_2^2}\right]$ and $\mathbb E[\sum \lVert \mathbb Q_t \rVert_1]$. Therefore, there may be room for generalizing the result for queue size increments with bounded moments.

Surprisingly, when one tries to develop expectation versions of \Cref{lemma-bounded-diff-S} and \Cref{lemma-ssmw-sum-l12norm}, they behave quite differently. \Cref{lemma-ssmw-sum-l12norm} has a direct expectation alternative but this does not seem to be the case for \Cref{lemma-bounded-diff-S}. Specifically, we can generalize \Cref{lemma-ssmw-sum-l12norm} to \Cref{lemma-ssmw-sum-l12norm-unbounded}.

\begin{lemma}
\label{lemma-ssmw-sum-l12norm-unbounded}
Under Assumption~\ref{assumption-moment}, for any $T_0 \ge 0$, let $m = \lceil \frac {\lVert\mathbf Q_{T_0}\rVert_\infty}{2M}\rceil$; further, suppose that $\lVert\mathbf Q_{T_0}\rVert_\infty \ge 8M$; then
\begin{equation*}
    \frac M 2 m^2 \le \mathbb E\left[\left.\sum_{t=1}^m \lVert\mathbf Q_{T_0 + t - 1}\rVert_1 \right\rvert \mathcal F_{T_0} \right] \le 3KMm^2,
\end{equation*}
\begin{equation*}
    M^2 m^3 \le \mathbb E\left[\left.\sum_{t=1}^m \lVert\mathbf Q_{T_0 + t - 1}\rVert_2^2 \right\rvert \mathcal F_{T_0} \right] \le 5KM^2m^3.
\end{equation*}
Thus, we have
\begin{equation*}
    \mathbb E\left[\left.\sqrt{\sum_{t=1}^T \lVert\mathbf Q_{T_0 + t - 1}\rVert_2^2 }\right\rvert \mathcal F_{T_0} \right] \le \sqrt{\mathbb E\left[\left.\sum_{t=1}^T \lVert\mathbf Q_{T_0 + t - 1}\rVert_2^2 \right\rvert \mathcal F_{T_0} \right]} \le \frac{2\sqrt{5K}} {\sqrt m} \mathbb E\left[\left.\sum_{t=1}^T \lVert\mathbf Q_{T_0 + t - 1}\rVert_1 \right\rvert \mathcal F_{T_0} \right].
\end{equation*}
\end{lemma}
\begin{proof}
    See \Cref{proof-lemma-ssmw-sum-l12norm-unbounded}.
\end{proof}

However, we cannot expect to generalize \Cref{lemma-bounded-diff-S} in the same way. Formally speaking, now our goal would be to obtain
\begin{equation}
    \mathbb E\left[\sqrt{\sum_{t=1}^T \lVert \mathbf Q_{t-1}\rVert_2^2}\right] \le \O\left(T^a \left(\mathbb E\left[\sqrt{\sum_{t=1}^T \lVert \mathbf Q_{t-1}\rVert_1}\right]\right)^b\right) \label{eq-lemma-bounded-diff-S-expected-ver.}
\end{equation}
for some $a+b < 1$. Recall that to obtain the last statement in \Cref{lemma-ssmw-sum-l12norm-unbounded}, we simply apply Jensen's inequality to get $\mathbb E\left[\sqrt{\sum \lVert \mathbf Q_t\rVert_2^2}\right] \le \sqrt{ \mathbb E\left[\sum \lVert \mathbf Q_t\rVert_2^2\right]}$. But now we want \Cref{eq-lemma-bounded-diff-S-expected-ver.} to hold under any scheduling policy; if we similarly apply Jensen's inequality, the task would reduce to bounding $\mathbb E\left[\sum \lVert \mathbf Q_t\rVert_2^2\right]$ \textit{unconditionally}, which is impossible since under trivial scheduling policies that cannot stabilize the system, we can have $\lVert \mathbf Q_t \rVert_2^2 = \Theta(t)$ and thus $\mathbb E\left[\sum \lVert \mathbf Q_t\rVert_2^2\right]$ is $\Theta(T^2)$.

It seems quite difficult to directly obtain an expectation bound for $\sqrt{\sum \lVert \mathbf Q_t\rVert_2^2}$ without interchanging the square root and the expectation. Nevertheless, we can still utilize \Cref{lemma-bounded-diff-S} in a somewhat naive way. Let
\begin{equation*}
    L \triangleq \max_{1 \le t \le T, i\in [K]} \lvert 
Q_{t, i} - Q_{t-1, i} \rvert,
\end{equation*}
i.e., $L$ is the maximum queue length increment in this sample path; hence, it is a random variable. Then, we can write
\begin{align*}
    \mathbb E\left[\sqrt{\sum_{t=1}^T \lVert \mathbf Q_{t-1}\rVert_2^2}\right] & \le \mathbb E\left[\mathbbm 1\left[L > C\right]\sqrt{\sum_{t=1}^T \lVert \mathbf Q_{t-1}\rVert_2^2}\right] + \mathbb E\left[\mathbbm 1\left[L \le C\right]\sqrt{\sum_{t=1}^T \lVert \mathbf Q_{t-1}\rVert_2^2}\right] \\
    & \le \mathbb E\left[\mathbbm 1\left[L > C\right]\sqrt{KT^3L^2}\right] + \mathbb E\left[\mathbbm 1\left[L \le C\right]\sqrt{\sum_{t=1}^T \lVert \mathbf Q_{t-1}\rVert_2^2}\right] \\
    & \le \sqrt K T^{\frac 3 2} \mathbb E\left[\mathbbm 1\left[L > C\right]L\right] + \mathbb E\left[\mathbbm 1\left[L \le C\right]\sqrt{\sum_{t=1}^T \lVert \mathbf Q_{t-1}\rVert_2^2}\right].
\end{align*}
For the $\mathbb E\left[\mathbbm 1\left[L > C\right]L\right]$ term, recall that $L$ is the maximum of $KT$ random variables (and they are martingale differences) with bounded $\alpha$-th moments; it can be shown that $L$'s $\alpha$-th moment is no more than $2KTM^\alpha$, and $\mathbb E\left[\mathbbm 1\left[L > C\right]L\right] \le 2KTM^{\alpha}C^{1-\alpha}$. As for the other term, we can apply \Cref{lemma-bounded-diff-S} to obtain a sample-path bound $\mathbbm 1\left[L \le C\right]\sqrt{\sum_{t=1}^T \lVert \mathbf Q_{t-1}\rVert_2^2} \le \O(C^{\frac 1 4} (\sum_{t=1}^T \lVert \mathbf Q_{t-1}\rVert_1)^{\frac 3 4})$. Therefore, when $\alpha$ is sufficiently large, we can choose $C$ as some power of $T$ to conclude $\mathbb E\left[\sqrt{\sum_{t=1}^T \lVert \mathbf Q_{t-1}\rVert_2^2}\right] \le \O\left(T^a \left(\mathbb E\left[\sqrt{\sum_{t=1}^T \lVert \mathbf Q_{t-1}\rVert_1}\right]\right)^b\right) + T^c$ where $0 < a,b,c < 1$ and $a+b < 1$.

\begin{proof}[Proof of \Cref{lemma-ssmw-sum-l12norm-unbounded}]
From the bouned $\alpha$-th moment assumption for queue length increments, we can also get a bound for the queue length increment itself:
\begin{align*}
    \mathbb E \left[\left. \left\lvert Q_{t,i} - Q_{t-1,i}\right\rvert\right\rvert \mathcal F_{t-1}\right] & = \mathbb E \left[\left. \left(\left\lvert Q_{t,i} - Q_{t-1,i}\right\rvert^\alpha\right)^{\frac 1 \alpha}\right\rvert \mathcal F_{t-1}\right] \\
    & \stackrel{(a)}\le \left(\mathbb E \left[\left. \left\lvert Q_{t,i} - Q_{t-1,i}\right\rvert^\alpha\right\rvert \mathcal F_{t-1}\right]\right)^{\frac 1 \alpha}\\ 
    & \le \left(M^\alpha\right)^{\frac 1 \alpha}\\
    & = M,
\end{align*}
where $(a)$ is due to the concavity of $x\mapsto x^{\frac 1 \alpha}$. Similarly, we have
\begin{equation*}
     \mathbb E \left[\left. \left( Q_{t,i} - Q_{t-1,i}\right)^2\right\rvert \mathcal F_{t-1}\right] \le M^2.
\end{equation*}
Then, for any $t \ge T_0$ and $i\in [K]$, we can write
\begin{align*}
    \mathbb E \left[\left.  Q_{t,i}\right\rvert \mathcal F_{T_0}\right] & \le Q_{T_0, i} + \mathbb E\left[\left. \sum_{s=T_0}^{t-1} \mathbb E \left[\left. \left\lvert Q_{s+1,i} - Q_{s,i}\right\rvert\right\rvert \mathcal F_s\right]\right\rvert \mathcal F_{T_0}\right] \\
    & \le Q_{T_0,i} + (t-T_0) M.
\end{align*}
If $T_
0  \le t \le T_0 + m$, then we have $\mathbb E \left[\left.  Q_{t,i}\right\rvert \mathcal F_{T_0}\right] \le Q_{T_0, i} + Mm$, hence
\begin{align*}
    \sum_{t=T_0}^{T_0 + m - 1} \mathbb E \left[\left.  \lVert Q_t \rVert_1\right\rvert \mathcal F_{T_0}\right] & \le \sum_{t=T_0}^{T_0 + m - 1} \sum_{i\in [K]}\left(Q_{T_0, i} + Mm\right) \\
    & = m\lVert Q_{T_0} \rVert_1 + KMm^2 \\
    & \stackrel{(a)}\le m\cdot 2KMm + KMm^2 \\
    & \le 3KMm^2,
\end{align*}
where step $(a)$ is due to our choice $m=\left\lceil \frac {\lVert\mathbf Q_{T_0}\rVert_\infty} {2M}\right\rceil$.

Similarly, for any $T_0 \le t \le T_0 + m$ and $i\in [K]$, we have
\begin{align*}
    \mathbb E \left[\left.  Q_{t,i}\right\rvert \mathcal F_{T_0}\right] & \ge Q_{T_0, i} - \mathbb E\left[\left. \sum_{s=T_0}^{t-1} \mathbb E \left[\left. \left\lvert Q_{s+1,i} - Q_{s,i}\right\rvert\right\rvert \mathcal F_s\right]\right\rvert \mathcal F_{T_0}\right] \\
    & \ge Q_{T_0,i} - (t-T_0) M \\
    & = Q_{T_0,i} - Mm,
\end{align*}
thus, denoting by $i^* = \argmax_{i\in [K]}Q_{T_0,i}$, we have
 \begin{align*}
    \sum_{t=T_0}^{T_0 + m - 1} \mathbb E \left[\left.  \lVert Q_t \rVert_1\right\rvert \mathcal F_{T_0}\right] & \ge \sum_{t=T_0}^{T_0 + m - 1} Q_{t,i^*}\\
    & \ge \sum_{t=T_0 + 1}^{T_0 + m} \left(Q_{T_0, i^*} - Mm\right) \\
    & = m\lVert Q_{T_0} \rVert_\infty - Mm^2 \\
    & \ge m\cdot (2Mm - 2M) - Mm^2 \\
    & = Mm^2 - 2Mm \\
    & \ge \frac 1 2 Mm^2,
\end{align*}
where the last step uses the assumption that $\lVert\mathbf Q_{T_0}\rVert_\infty \ge 8M$, which implies that $m \ge 4$. Therefore, we conclude that
\begin{equation*}
    \frac M 2 m^2 \le \mathbb E\left[\left.\sum_{t=1}^m \lVert\mathbf Q_{T_0 + t - 1}\rVert_1 \right\rvert \mathcal F_{T_0} \right] \le 3KMm^2.
\end{equation*}
For the bound for $\sum \lVert \mathbf Q_t\rVert_2^2$, we have
\begin{align*}
    & \quad \mathbb E \left[\left.  Q_{t,i}^2\right\rvert \mathcal F_{T_0}\right] \\
    & \le Q_{T_0, i}^2 + \mathbb E\left[\left. \sum_{s=T_0}^{t-1} \mathbb E \left[\left. \left( Q_{s+1,i} - Q_{s,i}\right)^2\right\rvert \mathcal F_s\right] + 2\sum_{T_0 \le s<s'<t} \mathbb E \left[\left. \left\lvert Q_{s+1,i} - Q_{s,i}\right\rvert\right\rvert \mathcal F_s\right] \cdot \mathbb E \left[\left. \left\lvert Q_{s'+1,i} - Q_{s',i}\right\rvert\right\rvert \mathcal F_{s'}\right]\right\rvert \mathcal F_{T_0}\right] \\
    & \le Q_{T_0,i}^2 + (t-T_0) M^2 + 2\binom{t-T_0}{2}M^2 \\
    & = Q_{T_0,i}^2 + (t-T_0)^2M^2 \\
    & \le Q_{T_0,i}^2 + M^2m^2
\end{align*}
for all $T_0 \le t \le T_0 + m$. Thus, we get
\begin{align*}
    \sum_{t=T_0}^{T_0 + m - 1} \mathbb E \left[\left.  \lVert Q_t \rVert_2^2\right\rvert \mathcal F_{T_0}\right] & \le \sum_{t=T_0}^{T_0 + m - 1} \sum_{i\in [K]}\left(Q_{T_0, i}^2 + M^2m^2\right) \\
    & = m\lVert Q_{T_0} \rVert_2^2 + KM^2m^3 \\
    & \stackrel{(a)}\le m\cdot K \cdot (2Mm)^2 + KM^2m^3 \\
    & \le 5KM^2m^3,
\end{align*}
where step $(a)$ is due to our choice $m=\left\lceil \frac {\lVert\mathbf Q_{T_0}\rVert_\infty} {2M}\right\rceil$. For the other direction, we have
\begin{align*}
    & \quad \mathbb E \left[\left.  Q_{t,i}^2\right\rvert \mathcal F_{T_0}\right] \\
    & \ge Q_{T_0, i}^2 - \mathbb E\left[\left. \sum_{s=T_0}^{t-1} \mathbb E \left[\left. \left( Q_{s+1,i} - Q_{s,i}\right)^2\right\rvert \mathcal F_s\right] - 2\sum_{T_0 \le s<s'<t} \mathbb E \left[\left. \left\lvert Q_{s+1,i} - Q_{s,i}\right\rvert\right\rvert \mathcal F_s\right] \cdot \mathbb E \left[\left. \left\lvert Q_{s'+1,i} - Q_{s',i}\right\rvert\right\rvert \mathcal F_{s'}\right]\right\rvert \mathcal F_{T_0}\right] \\
    & \ge Q_{T_0,i}^2 - (t-T_0) M^2 - 2\binom{t-T_0}{2}M^2 \\
    & = Q_{T_0,i}^2 - (t-T_0)^2M^2 \\
    & \ge Q_{T_0,i}^2 - M^2m^2.
\end{align*}
and thus denoting by $i^* = \argmax_{i\in [K]}Q_{T_0,i}$, we have
\begin{align*}
    \sum_{t=T_0}^{T_0 + m - 1} \mathbb E \left[\left.  \lVert Q_t \rVert_2^2\right\rvert \mathcal F_{T_0}\right] & \ge \sum_{t=T_0}^{T_0 + m - 1} \left(Q_{T_0, i^*}^2 - M^2m^2\right) \\
    & = m\lVert Q_{T_0} \rVert_\infty^2 - M^2m^3 \\
    & \stackrel{(a)}\ge m \cdot (2Mm - 2M)^2 - M^2m^3 \\
    & \stackrel{(b)}\ge m \cdot \left(\frac 3 2 Mm\right)^2 - M^2m^3 \\
    & \ge M^2 m^3,
\end{align*}
where step $(a)$ is due to our choice $m=\left\lceil \frac {\lVert\mathbf Q_{T_0}\rVert_\infty} {2M}\right\rceil$, step $(b)$ uses the assumption that $m\ge 4$ hence $2M \le \frac 1 2 Mm$. Therefore,
 \begin{equation*}
    M^2 m^3 \le \mathbb E\left[\left.\sum_{t=1}^m \lVert\mathbf Q_{T_0 + t - 1}\rVert_2^2 \right\rvert \mathcal F_{T_0} \right] \le 5KM^2m^3.
\end{equation*}

\end{proof}

\section{Detailed Analysis for Algorithm~\ref{softmw-moment}}
\label{sec-apdx-softmw-moment}

The analysis for \Cref{softmw-moment} is similar to \Cref{softmw}. First of all, we need to verify that the chosen exploration rates $\gamma_t$'s are guaranteed not to exceed $\frac 1 2$.
\begin{proposition}
\label{prop-softmw-moment-exp-rates}
For all $t\ge 1$, we have $\gamma_t \le \frac 1 2$ in \Cref{softmw-moment}.
\end{proposition}
\begin{proof}
\label{proof-lemma-ssmw-sum-l12norm-unbounded}
 Note that
    \begin{align*}
        \gamma_t & = t^{\frac \delta 4 } M L_{t-1}^{-1 } \lVert \mathbf Q_{t - 1}\rVert_1 \left( t^{-(\frac 1 4 - \frac \delta 2)} \sqrt{ 86L_{t-1}^2 K^6 t^{\frac 3 2} + \sum_{s=0}^{t-1} \lVert \mathbf Q_s \rVert_2^2} \right)^{-1} \\
        & \le t^{\frac \delta 4} \lVert \mathbf Q_{t - 1}\rVert_1 \left( t^{-(\frac 1 4 - \frac \delta 2)} \sqrt{ 86L_{t-1}^2 K^6 t^{\frac 3 2} + \sum_{s=0}^{t-1} \lVert \mathbf Q_s \rVert_2^2} \right)^{-1},
    \end{align*}
    since $L_t \ge M$ for all $t\ge 0$. Thus it suffices to verify that
    \begin{align*}
        2t^{\frac \delta 4 } \lVert \mathbf Q_{t - 1}\rVert_1 & \le t^{-(\frac 1 4 - \frac \delta 2)} \sqrt{ 86L_{t-1}^2 K^6 t^{\frac 3 2} + \sum_{s=0}^{t-1} \lVert \mathbf Q_s \rVert_2^2}, \\
    & \Updownarrow \\
        4\lVert \mathbf Q_{t - 1}\rVert_1^2 & \le t^{-\frac 1 2 + \frac \delta 2 } \left( 86L_{t-1}^2 K^6 t^{\frac 3 2} + \sum_{s=0}^{t-1} \lVert \mathbf Q_s \rVert_2^2\right), \\
    & \Uparrow (a) \\
        4\lVert \mathbf Q_{t - 1}\rVert_1^2 & \le t^{-\frac 1 2 + \frac \delta 2} \left(  86L_{t-1}^2 K^6 t^{\frac 3 2} + \frac 1 {3L_{t-1}K^3} \lVert \mathbf Q_{t -1} \rVert_1^3 \right) \\
    & \Updownarrow \\
        4\lVert \mathbf Q_{t - 1}\rVert_1^2 & \le t^{-\frac 1 2 + \frac \delta 2} \left( \frac 1 3 \cdot 
 258 L_{t-1}^2 K^6 t^{\frac 3 2} + \frac 2 3 \cdot \frac 1 {2L_{t-1}K^3} \lVert \mathbf Q_{t -1} \rVert_1^3\right), \\
    & \Uparrow (b) \\
        4\lVert \mathbf Q_{t - 1}\rVert_1^2 & \le t^{-\frac 1 2 + \frac \delta 2} \cdot 258^{\frac 1 3} 2^{-\frac 2 3} t^{\frac 1 2} \lVert \mathbf Q_{t - 1}\rVert_1^2 \\
    & \Updownarrow \\
        4 & \le t^{\frac \delta 2} \left(\frac {258} 4\right)^{\frac 1 3},
    \end{align*}
    and the last statement trivially holds. Here in step $(a)$ we apply \Cref{corollary-bounded-diff-S-rev} (it is applicable as long as we replace $M$ by $L_{t-1}$), in step $(b)$ we apply AM-GM inequality $\frac 1 3 x + \frac 2 3 y \ge x^{\frac 1 3} y ^{\frac 2 3}$.
\end{proof}

Having verified that $\gamma_t \ge \frac 1 2$, also notice that we choose $\gamma_t = Mt^{\frac \delta 4 }\eta_t \lVert \mathbf Q_{t - 1}\rVert_1$ in \Cref{softmw-moment}, which allows \texttt{EXP3.S+} to handle a new feedback $g_{t,a_t}$ as large as $Q_{t-1,a_t}\cdot Mt^{\frac \delta 4 }$. Compared to \Cref{softmw}, in \Cref{softmw-moment}, we clip the new service $S_{t,a_t}$ at $Mt^{\frac \delta 4 }$ to $S'_t$, and instead feed $Q_{t-1,a_t}S'_t$ into \texttt{EXP3.S+}, therefore, \Cref{thm-exp3s} is now applicable, and we can get

\begin{lemma}
\label{lemma-softmw-moment-regret}
Suppose Assumption~\ref{assumption-theta}, \ref{assumption-delta} and \ref{assumption-moment} hold, then, running \Cref{softmw-moment} guarantees
\begin{align}
    & \quad \sum_{t=1}^T \mathbb E\left[ \langle \mathbf Q_{t-1}, \mathbf S'_t \odot \vec \theta_t \rangle - Q_{t-1,a_t}S'_{t,a_t} \right] \nonumber \\
     & \le  \mathbb E\left[ L_{T-1}(1 + C_V) T^{\frac 1 4 - \frac \delta 2 } (4\ln T + \ln K) \sqrt{ 86L_{T-1}^2 K^6 T^{\frac 3 2} + \sum_{t=1}^T \lVert \mathbf Q_{t-1} \rVert_2^2}\right] \nonumber \\
     & \quad + \mathbb E\left[ 8M T^{\frac 1 4 - \frac \delta 4 } \sqrt{ 1 + \sum_{t=1}^T \lVert \mathbf Q_{t-1} \rVert_2^2} + 4ML_{T-1} \right] \label{eq-lemma-softmw-moment-regret}
\end{align}
for any time horizon length $T\ge 1$. Here $S_{t,i}$ is defined as
\begin{equation*}
    S'_{t,i} \triangleq \begin{cases}
        S_{t,i} & \text{if } S_{t,i} \le Mt^{\frac \delta 4} \\
        0 & \text{otherwise}
    \end{cases},
\end{equation*}
and $\{\vec \theta_t\}$ is the reference policy in Assumption~\ref{assumption-theta} and \ref{assumption-delta}.
\end{lemma}
\begin{proof}
    For each $\vec \theta_t$, we can find a $\vec\theta'_t \in \Delta^{[K], \beta_t}$ so that $\lVert \vec \theta_t - \vec \theta'_t \rVert_1 \le 2K \beta_t$ and $\lVert 
\vec \theta'_s - \vec \theta'_t \rVert_1 \le \lVert 
\vec \theta_s - \vec \theta_t \rVert_1$. For example, we can choose \begin{equation*}
    \vec \theta'_t = (1 - \beta_t) \vec \theta_t + \beta_t \mathbf 1.
\end{equation*} 
Then, we can write
\begin{align*}
    \langle \mathbf Q_{t-1}, \mathbf S'_t \odot \vec \theta_t \rangle & = \langle \mathbf Q_{t-1}, \mathbf S'_t \odot \vec \theta'_t \rangle + \langle \mathbf Q_{t-1} \odot \mathbf S'_t, \vec \theta_t - \vec \theta'_t \rangle \\
    & \le \langle \mathbf Q_{t-1}, \mathbf S'_t \odot \vec \theta'_t \rangle + \lVert 
\mathbf Q_{t-1} \odot \mathbf S'_t \rVert_\infty \cdot \lVert \vec \theta_t - \vec \theta'_t \rVert_1 \\
& \le \langle \mathbf Q_{t-1}, \mathbf S'_t \odot \vec \theta'_t \rangle + L_{t-1}t \cdot Mt^{\frac \delta 4 } \cdot 2K\beta_t \\
& = \langle \mathbf Q_{t-1}, \mathbf S'_t \odot \vec \theta'_t \rangle + 2L_{t-1}Mt^{\frac \delta 4 - 3} \\
& \le \langle \mathbf Q_{t-1}, \mathbf S'_t \odot \vec \theta'_t \rangle + 2L_{t-1}Mt^{-2}.
\end{align*}
Then, one can see the quantity
\begin{equation*}
    \sum_{t=1}^T \mathbb E\left[ \langle \mathbf Q_{t-1}, \mathbf S'_t \odot \vec \theta'_t \rangle - Q_{t-1,a_t}S'_{t,a_t} \right] = \sum_{t=1}^T \mathbb E\left[ \langle \mathbf Q_{t-1} \odot \mathbf S'_t, \vec \theta'_t \rangle - (\mathbf Q_{t-1} \odot \mathbf S'_t)_{a_t} \right]
\end{equation*}
satisfies the condition to apply \Cref{thm-exp3s}. On the other hand, the total difference of the regret to $\{\vec \theta_t\}$ and the regret to $\{\vec \theta'_t\}$  is within
\begin{align*}
    \sum_{t=1}^T 2L_{t-1}Mt^{-2} & \le 2ML_{T-1} \sum_{t=1}^T t^{-2} \\
    & \le \frac {\pi^2} 3 ML_{T-1}.
\end{align*}

\Cref{thm-exp3s} asserts that
\begin{align}
    & \quad \sum_{t=1}^T \mathbb E\left[ \langle \mathbf Q_{t-1}, \mathbf S'_t \odot \vec \theta'_t \rangle - Q_{t-1,a_t}S'_{t,a_t} \right] \nonumber \\
    & \le \left(1 + \sum_{t=1}^{T-1} \lVert \vec\theta'_{t+1} - \vec\theta'_t \rVert_1\right)\mathbb E\left[ \eta_T^{-1} \ln \frac 1 {\beta_T}\right] + e\mathbb E\left[\sum_{t=1}^T \eta_t \lVert 
\mathbf Q_{t-1}\odot \mathbf S'_t \rVert_2^2\right] \nonumber \\
& \quad + \mathbb E\left[\sum_{t=1}^T\gamma_t \left\langle \mathbf Q_{t-1}\odot \mathbf S'_t, \frac{\mathbf Q_{t-1}}{\lVert \mathbf Q_{t-1}\rVert_1} \right\rangle\right].\label{eq-proof-lemma-softmw-moment-regret-1}
\end{align}
Below, we bound each term in the RHS of \Cref{eq-proof-lemma-softmw-moment-regret-1}. Firstly, we have $\sum_{t=1}^{T-1} \lVert \vec\theta'_{t+1} - \vec\theta'_t \rVert_1 \le \sum_{t=1}^{T-1} \lVert \vec\theta_{t+1} - \vec\theta_t \rVert_1 \le C_V T^{\frac 1 2 - \delta}$ according to Assumption~\ref{assumption-delta}, hence
\begin{align*}
    & \quad \left(1 + \sum_{t=1}^{T-1} \lVert \vec\theta'_{t+1} - \vec\theta'_t \rVert_1\right)\mathbb E\left[ \eta_T^{-1} \ln \frac 1 {\beta_T}\right] \\ 
    & \le \left(1 + C_V T^{\frac 1 2 -\delta}\right)  \mathbb E\left[ \eta_T^{-1} \ln \frac 1 {\beta_T}\right] \\
    & \le \left(1 + C_V T^{\frac 1 2 -\delta}\right) \left( 4\ln T + \ln K\right) \mathbb E\left[ T^{-\left(\frac 1 4 - \frac \delta 2\right)} L_{T-1} \sqrt{ 86L_{T-1}^2 K^6 T^{\frac 3 2} + \sum_{s=0}^{T -1} \lVert \mathbf Q_s \rVert_2^2}\right] \\
    & \le \left(1 + C_V\right) T^{\frac 1 4 - \frac \delta 2 } \left( 4\ln T + \ln K\right) L_{T-1} \mathbb E\left[\sqrt{ 86L_{T-1}^2 K^6 T^{\frac 3 2} + \sum_{s=0}^{T -1} \lVert \mathbf Q_s \rVert_2^2}\right].
\end{align*}

For the second term, we have
\begin{align*}
    & \quad \sum_{t=1}^T\mathbb E\left[ \left. \eta_t \lVert
\mathbf Q_{t-1}\odot \mathbf S'_t \rVert_2^2 \right\rvert \mathcal F_{t-1} \right] \\
& = \sum_{t=1}^T \eta_t \sum_{i\in[K]} Q_{t-i,i}^2 \mathbb E\left[ \left. S_{t,i}^{\prime 2} \right\rvert \mathcal F_{t-1} \right]\\
& \stackrel{(a)}\le M^2 \sum_{t=1}^T \eta_t \lVert 
\mathbf Q_{t-1} \rVert_2^2 \\
& = M \sum_{t=1}^T ML_{t-1}^{-1} t^{\frac 1 4 - \frac \delta 2} \left(\sqrt{ 86L_{t-1}^2 K^6 t^{\frac 3 2} + \sum_{s=0}^{t -1} \lVert \mathbf Q_s \rVert_2^2}\right)^{-1} \cdot \lVert 
\mathbf Q_{t-1} \rVert_2^2 \\
& \stackrel{(b)}\le M \sum_{t=1}^T t^{\frac 1 4 - \frac \delta 2} \left(\sqrt{ 1 + \sum_{s=0}^{t -1} \lVert \mathbf Q_s \rVert_2^2}\right)^{-1} \cdot \lVert 
\mathbf Q_{t-1} \rVert_2^2 \\
& \le M T^{\frac 1 4 - \frac \delta 2}  \sum_{t=1}^T \left(\sqrt{ 1 + \sum_{s=0}^{t -1} \lVert \mathbf Q_s \rVert_2^2}\right)^{-1} \cdot \lVert 
\mathbf Q_{t-1} \rVert_2^2 \\
& \stackrel{(c)}\le 2 M T^{\frac 1 4 - \frac \delta 2 } \sqrt{1 + \sum_{s=0}^{T-1}\lVert \mathbf Q_s \rVert_2^2},
\end{align*}
where in step $(a)$, we apply the second-order moment bound for $S_{t,i}$ instead of the clipping threshold $Mt^{\frac \delta 4}$. Step $(b)$ is due to $L_t \ge M$ for all $t$. In step $(c)$, we use the fact that
\begin{equation*}
    \sum_{i=1}^n \frac {x_i}{\sqrt{1 + \sum_{j=1}^i x_j}} \le 2 \sqrt{1 + \sum_{i=1}^n x_i}
\end{equation*}
for non-negative $x_1\ldots, x_n$.

The third term can be bounded very similarly to the second term, to be specific, we have
\begin{align*}
    & \quad \sum_{t=1}^T \mathbb E\left[\left.\gamma_t \left\langle \mathbf Q_{t-1}\odot \mathbf S'_t, \frac{\mathbf Q_{t-1}}{\lVert \mathbf Q_{t-1}\rVert_1} \right\rangle \right\rvert \mathcal F_{t-1}\right] \\
    & \stackrel{(a)}\le  M \sum_{t=1}^T \gamma_t \lVert \mathbf Q_{t-1}\rVert_1^{-1} \lVert \mathbf Q_{t-1}\rVert_2^2 \\
    & \le M \sum_{t=1}^T ML_{t-1}^{-1} t^{\frac \delta 4} \cdot  t^{\frac 1 4 - \frac \delta 2} \left(\sqrt{ 1 + \sum_{s=0}^{t -1} \lVert \mathbf Q_s \rVert_2^2}\right)^{-1} \cdot \lVert 
\mathbf Q_{t-1} \rVert_2^2 \\
& \le M \sum_{t=1}^T t^{\frac 1 4 - \frac \delta 4} \left(\sqrt{ 1 + \sum_{s=0}^{t -1} \lVert \mathbf Q_s \rVert_2^2}\right)^{-1} \cdot \lVert 
\mathbf Q_{t-1} \rVert_2^2 \\
& \le 2 M T^{\frac 1 4 - \frac \delta 4} \sqrt{1 + \sum_{s=0}^{T-1}\lVert \mathbf Q_s \rVert_2^2},
\end{align*}
where in step $(a)$ we apply the first-order moment bound for $S_{t,i}$.

Combining everything together then taking expectation, we can see that
\begin{align*}
    & \quad \sum_{t=1}^T \mathbb E\left[ \langle \mathbf Q_{t-1}, \mathbf S'_t \odot \vec \theta'_t \rangle - Q_{t-1,a_t}S'_{t,a_t} \right] \\
    & \le \left(1 + C_V\right) T^{\frac 1 4 - \frac \delta 2 } \left( 4\ln T + \ln K\right) L_{T-1} \mathbb E\left[\sqrt{ 86L_{T-1}^2 K^6 T^{\frac 3 2} + \sum_{s=0}^{T -1} \lVert \mathbf Q_s \rVert_2^2}\right] \\
    & \quad + 2 M T^{\frac 1 4 - \frac \delta 2 }\mathbb E\left[ \sqrt{1 + \sum_{s=0}^{T-1}\lVert \mathbf Q_s \rVert_2^2}\right] + 2e M T^{\frac 1 4 - \frac \delta 4} \mathbb E\left[\sqrt{1 + \sum_{s=0}^{T-1}\lVert \mathbf Q_s \rVert_2^2}\right] + \frac {\pi^2} 3 M\mathbb E\left[L_{T-1}\right] \\
    & \le \left(1 + C_V\right) T^{\frac 1 4 - \frac \delta 2 } \left( 4\ln T + \ln K\right) L_{T-1} \mathbb E\left[\sqrt{ 86L_{T-1}^2 K^6 T^{\frac 3 2} + \sum_{s=0}^{T -1} \lVert \mathbf Q_s \rVert_2^2}\right] \\
    & \quad + 8M T^{\frac 1 4 - \frac \delta 4} \mathbb E\left[\sqrt{1 + \sum_{s=0}^{T-1}\lVert \mathbf Q_s \rVert_2^2}\right] + 4M \mathbb E\left[L_{T-1}\right].
\end{align*}    
\end{proof}

Then, we can turn \Cref{lemma-softmw-moment-regret}, the upper-bound for $(\mathbf Q_{t-1}\cdot \mathbf S'_t)$-objective regret to a $(\mathbf Q_{t-1}\cdot \mathbf S_t)$-objective one:

\begin{lemma}
\label{lemma-softmw-moment-regret-2}
Suppose Assumption~\ref{assumption-theta}, \ref{assumption-delta} and \ref{assumption-moment} hold, then, running \Cref{softmw-moment} guarantees
\begin{align}
    & \quad \sum_{t=1}^T \mathbb E\left[ \langle \mathbf Q_{t-1}, \mathbf S_t \odot \vec \theta_t \rangle - Q_{t-1,a_t}S_{t,a_t} \right] \nonumber \\
     & \le  \mathbb E\left[ L_{T-1}(1 + C_V) T^{\frac 1 4 - \frac \delta 2 } (4\ln T + \ln K) \sqrt{ 86L_{T-1}^2 K^6 T^{\frac 3 2} + \sum_{t=1}^T \lVert \mathbf Q_{t-1} \rVert_2^2}\right] \nonumber \\
     & \quad + \mathbb E\left[ 8M T^{\frac 1 4 - \frac \delta 4 } \sqrt{ 1 + \sum_{t=1}^T \lVert \mathbf Q_{t-1} \rVert_2^2} + 4ML_{T-1} \right] + o(T) \label{eq-lemma-softmw-moment-regret-2}
\end{align}
for any time horizon length $T\ge 1$. Here $\{\vec \theta_t\}$ is the reference policy in Assumption~\ref{assumption-theta} and \ref{assumption-delta}.
\end{lemma}
\begin{proof}
    Note that for each $t$, the single-step regret $\langle \mathbf Q_{t-1}, \mathbf S_t \odot \vec \theta_t \rangle - Q_{t-1,a_t}S_{t,a_t}$ can be upper-bounded by
    \begin{align*}
        \langle \mathbf Q_{t-1}, \mathbf S_t \odot \vec \theta_t \rangle - Q_{t-1,a_t}S_{t,a_t} & \le \underbrace{\langle \mathbf Q_{t-1}, \mathbf S'_t \odot \vec \theta_t \rangle - Q_{t-1,a_t}S'_{t,a_t}}_{(\mathbf Q_{t-1}\cdot \mathbf S'_t)\text{-objective single step regret}} + \mathbbm 1\left[S_{t,a_t} > Mt^{\frac \delta 4}\right]\langle \mathbf Q_{t-1}, \mathbf S_t \odot \vec \theta_t \rangle
    \end{align*}
    where 
    \begin{equation*}
    S'_{t,i} \triangleq \begin{cases}
        S_{t,i} & \text{if } S_{t,i} \le Mt^{\frac \delta 4} \\
        0 & \text{otherwise}
    \end{cases}.
\end{equation*}
    Therefore, it suffices to upper-bound $\sum_{t=1}^T \mathbb E\left[\mathbbm 1\left[S_{t,a_t} > Mt^{\frac \delta 4}\right]\langle \mathbf Q_{t-1}, \mathbf S_t \odot \vec \theta_t \rangle\right]$. In fact,
    \begin{align*}
        & \quad \mathbb E\left[\left.\mathbbm 1\left[S_{t,a_t} > Mt^{\frac \delta 4}\right]\langle \mathbf Q_{t-1}, \mathbf S_t \odot \vec \theta_t \rangle\right\rvert \mathcal F_{t-1}\right] \\
        & \le \lVert \mathbf Q_{t-1}\rVert_\infty \cdot \mathbb E\left[\mathbbm 1\left[S_{t,a_t} > Mt^{\frac \delta 4}\right]  \lVert\mathbf S_t\rVert_1\right] \\
        & = \lVert \mathbf Q_{t-1}\rVert_\infty \left( \mathbb E\left[\mathbbm 1\left[S_{t,a_t} > Mt^{\frac \delta 4}\right]  \sum_{i\in[K]:i\ne a_t}S_{t,i}\right] + \mathbb E\left[\mathbbm 1\left[S_{t,a_t} > Mt^{\frac \delta 4}\right]  S_{t,a_t}\right]\right).
    \end{align*}
    For the first term $\mathbb E\left[\mathbbm 1\left[S_{t,a_t} > Mt^{\frac \delta 4}\right]  \sum_{i\in[K]:i\ne a_t}S_{t,i}\right]$, we can write
    \begin{align*}
        \mathbb E\left[\mathbbm 1\left[S_{t,a_t} > Mt^{\frac \delta 4}\right]  \sum_{i\in[K]:i\ne a_t}S_{t,i}\right] & = \mathbb E\left[\mathbb P\left[\left.S_{t,a_t} > Mt^{\frac \delta 4}\right\rvert a_t\right]\cdot \mathbb E\left[\left.\sum_{i\in[K]:i\ne a_t}S_{t,i}\right\rvert a_t \right]\right] \\
        & = \mathbb E\left[\mathbb P\left[\left.S_{t,a_t} > Mt^{\frac \delta 4}\right\rvert a_t\right]\cdot \sum_{i\in[K]:i\ne a_t} \sigma_{t,i}\right] \\
        & \le KM\cdot \mathbb P\left[S_{t,a_t} > Mt^{\frac \delta 4}\right] \\
        & \le KM\cdot M^\alpha \cdot (Mt^{\frac \delta 4})^{-\alpha} \\
        & = KMt^{-\frac {\delta \alpha} 4},
    \end{align*}
    where the last inequality is due to the $\alpha$-th moment upper-bound for $S_{t,a_t}$ and Chebyshev's inequality. For the other term $\mathbb E\left[\mathbbm 1\left[S_{t,a_t} > Mt^{\frac \delta 4}\right]  S_{t,a_t}\right]$, we have
    \begin{align*}
        \mathbb E\left[\mathbbm 1\left[S_{t,a_t} > Mt^{\frac \delta 4}\right]  S_{t,a_t}\right] & \le \mathbb E\left[S_{t,a_t}^\alpha \cdot (Mt^{\frac \delta 4})^{-(\alpha - 1)}\right] \\
        & \le Mt^{-\frac{\delta(\alpha - 1)} 4}.
    \end{align*}
    Therefore,
    \begin{align*}
        \mathbb E\left[\left.\mathbbm 1\left[S_{t,a_t} > Mt^{\frac \delta 4}\right]\langle \mathbf Q_{t-1}, \mathbf S_t \odot \vec \theta_t \rangle\right\rvert \mathcal F_{t-1}\right] \le KMt^{-\frac {\delta(\alpha - 1)} 4}  \cdot \lVert \mathbf Q_{t-1}\rVert_\infty,
    \end{align*}
    \begin{align*}
        & \quad \sum_{t=1}^T \mathbb E\left[\mathbbm 1\left[S_{t,a_t} > Mt^{\frac \delta 4}\right]\langle \mathbf Q_{t-1}, \mathbf S_t \odot \vec \theta_t \rangle\right] \\
        & \le KM  \sum_{t=1}^T t^{-\frac {\delta(\alpha - 1)} 4} \mathbb E\left[\lVert \mathbf Q_{t-1}\rVert_\infty\right] \\
        & \le KM  \sum_{t=1}^T t^{ 1-\frac {\delta(\alpha - 1)} 4} \mathbb E\left[ L_t\right] \\
        & \le KM  \sum_{t=1}^T t^{ 1-\frac {\delta(\alpha - 1)} 4} \mathbb E\left[ L_t^\alpha\right]^{\frac 1 \alpha} \\
        & \le 2K^{1 + \frac 1 \alpha} M^2 \sum_{t=1}^T t^{ 1-\frac {\delta(\alpha - 1)} 4} \cdot t^{\frac 1 \alpha} \\
        & \le o(T)
    \end{align*}
    where the last step is due to the assumption of $\delta\cdot \alpha > 7$.
\end{proof}

Compared to \Cref{lemma-softmw-regret}, the RHS of the bound in \Cref{lemma-softmw-moment-regret-2} now involves $L_{T-1}$, the sample-path maximum queue length increment. Now we will do a further calculation to get rid of the $L_{T-1}$ factors.

Recall that $L_t$ is defined as
\begin{equation*}
    L_t \triangleq \max\left\{ M,\max_{1\le s\le t, i\in [K]}\left\lvert 
Q_{s,i} - Q_{s-1, i} \right\rvert\right\}
\end{equation*}
in \Cref{softmw-moment}, the bounded-moment-queue-length increment assumption gives an upper-bound for the $\alpha$-th moment of $L_t$
\begin{align*}
    \mathbb E\left[L_t^\alpha\right] & \le \max\left\{M^\alpha,  \sum_{s=1}^t \mathbb E\left[\left\lvert 
Q_{s,i} - Q_{s-1, i} \right\rvert^\alpha\right]\right\} \\
& \le \max\left\{ M^\alpha, 2KTM^\alpha\right\} \\
& = 2KTM^\alpha.
\end{align*}

This moment bound for $L_t$ enables us to deal with the $L_{T-1}$ factor in \Cref{lemma-softmw-moment-regret-2}. For example, we have
\begin{align*}
    \mathbb E\left[L_{T-1}\right] & \le \mathbb E\left[L_{T-1}^\alpha\right]^{\frac 1 \alpha} \le 2K^{\frac 1 \alpha} M T^{\frac 1 \alpha} = o(T).
\end{align*}

To bound the $\mathbb E\left[ L_{T-1}(1 + C_V) T^{\frac 1 4 - \frac \delta 2 } (4\ln T + \ln K) \sqrt{ 86L_{T-1}^2 K^6 T^{\frac 3 2} + \sum_{t=1}^T \lVert \mathbf Q_{t-1} \rVert_2^2}\right]$ term, we first simply write
\begin{align*}
    & \quad \mathbb E\left[ L_{T-1}(1 + C_V) T^{\frac 1 4 - \frac \delta 2 } (4\ln T + \ln K) \sqrt{ 86L_{T-1}^2 K^6 T^{\frac 3 2} + \sum_{t=1}^T \lVert \mathbf Q_{t-1} \rVert_2^2}\right] \\
    & \le \mathbb E\left[ L_{T-1}(1 + C_V) T^{\frac 1 4 - \frac \delta 2 } (4\ln T + \ln K) \cdot \left(\sqrt{ 86L_{T-1}^2 K^6 T^{\frac 3 2}} + \sqrt{\sum_{t=1}^T \lVert \mathbf Q_{t-1} \rVert_2^2}\right)\right] \\
    & = \mathbb E\left[ K^3(1 + C_V) T^{1 - \frac \delta 2 } (4\ln T + \ln K)L_{T-1}^2\right] + \mathbb E\left[ (1 + C_V) T^{\frac 1 4 - \frac \delta 2 } (4\ln T + \ln K)L_{T-1}\sqrt{\sum_{t=1}^T \lVert \mathbf Q_{t-1} \rVert_2^2}\right]
\end{align*}
and then bound the two terms one by one. The $L_{T-1}^2$ term is easy:
\begin{align*}
    \mathbb E\left[ K^3(1 + C_V) T^{1 - \frac \delta 2 } (4\ln T + \ln K)L_{T-1}^2\right] & \le K^3(1 + C_V) T^{1 - \frac \delta 2 } (4\ln T + \ln K)\mathbb E\left[L_{T-1}^\alpha\right]^{\frac 2 \alpha} \\
    & \le K^3(1 + C_V) T^{1 - \frac \delta 2 } (4\ln T + \ln K) \cdot 2K^{\frac 2 \alpha} M^2 T^{\frac 2 \alpha} \\
    & \le \O(T^{1-\frac \delta 2 + \frac 2 \alpha}) \\
    & \le o(T)
\end{align*}
where the last step is due to the assumption that $\delta \cdot \alpha > 7$. For the other term, we can write
\begin{align}
    & \quad L_{T-1}\sqrt{\sum_{t=1}^T \lVert \mathbf Q_{t-1} \rVert_2^2} \nonumber \\
    & \le \mathbbm 1 \left[L_{T-1}\le MT^{\frac \delta 4}\right] L_{T-1}\sqrt{\sum_{t=1}^T \lVert \mathbf Q_{t-1} \rVert_2^2} + \mathbbm 1 \left[L_{T-1} > MT^{\frac \delta 4}\right] L_{T-1}\sqrt{\sum_{t=1}^T \lVert \mathbf Q_{t-1} \rVert_2^2} \label{eq-softmw-moment-total-investment-split}.
\end{align}
For the term with factor $\mathbbm 1 \left[L_{T-1}\le MT^{\frac \delta 4}\right]$ in \Cref{eq-softmw-moment-total-investment-split}, we can apply \Cref{lemma-bounded-diff-S} to get
\begin{align*}
    \mathbbm 1 \left[L_{T-1}\le MT^{\frac \delta 4}\right] L_{T-1}\sqrt{\sum_{t=1}^T \lVert \mathbf Q_{t-1} \rVert_2^2} & \le \mathbbm 1 \left[L_{T-1}\le MT^{\frac \delta 4}\right] 2L_{T-1}^{\frac 5 4}\left(\sum_{t=1}^T \lVert \mathbf Q_{t-1} \rVert_1\right)^{\frac 3 4} \\
    & \le 2M^{\frac 5 4} T^{\frac 5 {16} \delta} \left(\sum_{t=1}^T \lVert \mathbf Q_{t-1} \rVert_1\right)^{\frac 3 4}.
\end{align*}
For the other term with factor $\mathbbm 1 \left[L_{T-1} > MT^{\frac \delta 4}\right]$, we can write
\begin{align*}
    \mathbbm 1 \left[L_{T-1} > MT^{\frac \delta 4}\right] L_{T-1}\sqrt{\sum_{t=1}^T \lVert \mathbf Q_{t-1} \rVert_2^2} & \le \mathbbm 1 \left[L_{T-1} > MT^{\frac \delta 4}\right] L_{T-1} \sqrt{\sum_{t=1}^T K(L_{T-1}T)^2} \\
    & \le \sqrt K T^{\frac 3 2} \mathbbm 1 \left[L_{T-1} > MT^{\frac \delta 4}\right] L_{T-1}^2,
\end{align*}
besides, we can bound its expectation by
\begin{align*}
    \mathbb E\left[\mathbbm 1 \left[L_{T-1} > MT^{\frac \delta 4}\right] L_{T-1}^2\right] & \le \mathbb E\left[ L_{T-1}^\alpha / (MT^{\frac \delta 4})^{\alpha - 2}\right] \\
    & \le 2KTM^\alpha \cdot (MT^{\frac \delta 4})^{-(\alpha - 2)} \\
    & \le 2KM^2 T^{1-\frac {\alpha - 2} 4 \delta},
\end{align*}
hence
\begin{align*}
    \mathbb E\left[\mathbbm 1 \left[L_{T-1} > MT^{\frac \delta 4}\right] L_{T-1}\sqrt{\sum_{t=1}^T \lVert \mathbf Q_{t-1} \rVert_2^2}\right] & \le \sqrt K T^{\frac 3 2} \cdot 2KM^2 T^{1 - \frac {\alpha - 2} 4 \delta} \\
    & = 2K^{\frac 3 2}M^2 T^{\frac 5 2 - \frac {\alpha - 2} 4 \delta} \\
    & = 2K^{\frac 3 2}M^2 T^{ \frac {10 + 2\delta - \delta\alpha} 4} \\
    & \le 2K^{\frac 3 2}M^2  T^{ \frac {11 - \delta\alpha} 4} \\
    & \le o(T)
\end{align*}
where the last inequality is due to the assumption that $\delta \cdot \alpha > 7$.

Combining bounds for all terms, we can conclude that 
\begin{align*}
    & \quad \mathbb E\left[ (1 + C_V) T^{\frac 1 4 - \frac \delta 2 } (4\ln T + \ln K)L_{T-1}\sqrt{\sum_{t=1}^T \lVert \mathbf Q_{t-1} \rVert_2^2}\right] \\
    & \le o(T) + \mathbb E\left[(1 + C_V) T^{\frac 1 4 - \frac \delta 2 } (4\ln T + \ln K) \cdot \mathbbm 1 \left[L_{T-1}\le MT^{\frac \delta 4}\right] L_{T-1}\sqrt{\sum_{t=1}^T \lVert \mathbf Q_{t-1} \rVert_2^2}\right] \\
    & \le o(T) + (1 + C_V) T^{\frac 1 4 - \frac \delta 2 } (4\ln T + \ln K) \cdot 2M^{\frac 5 4} T^{\frac 5 {16} \delta} \mathbb E\left[\left(\sum_{t=1}^T \lVert \mathbf Q_{t-1} \rVert_1\right)^{\frac 3 4}\right] \\
    & = o(T) + 2(1 + C_V)M^{\frac 5 4}T^{\frac 1 4 - \frac 3 {16}\delta} (4\ln T + \ln K) \mathbb E\left[\left(\sum_{t=1}^T \lVert \mathbf Q_{t-1} \rVert_1\right)^{\frac 3 4}\right].
\end{align*}

It remains to bound the $\mathbb E\left[ 8M T^{\frac 1 4 - \frac \delta 4 } \sqrt{ 1 + \sum_{t=1}^T \lVert \mathbf Q_{t-1} \rVert_2^2} \right]$ term in \Cref{lemma-softmw-moment-regret}. Since $T^{\frac 1 4 - \frac \delta 4 } \sqrt{ 1 + \sum_{t=1}^T \lVert \mathbf Q_{t-1} \rVert_2^2}  + T^{\frac 1 4 - \frac \delta 4 } \cdot \sqrt{ \sum_{t=1}^T \lVert \mathbf Q_{t-1} \rVert_2^2} = o(T) + T^{\frac 1 4 - \frac \delta 4 } \cdot \sqrt{\sum_{t=1}^T \lVert \mathbf Q_{t-1} \rVert_2^2}$, it suffices to bound $\mathbb E\left[ M T^{\frac 1 4 - \frac \delta 4 } \sqrt{ \sum_{t=1}^T \lVert \mathbf Q_{t-1} \rVert_2^2} \right]$. Similarly, we begin by writing
\begin{align*}
    T^{\frac 1 4 - \frac \delta 4 } \sqrt{ \sum_{t=1}^T \lVert \mathbf Q_{t-1} \rVert_2^2} \le \mathbbm 1 \left[L_{T-1} > MT^{\frac \delta 4}\right]  T^{\frac 1 4 - \frac \delta 4 } \sqrt{ \sum_{t=1}^T \lVert \mathbf Q_{t-1} \rVert_2^2} +  \mathbbm 1 \left[L_{T-1} \le MT^{\frac \delta 4}\right] T^{\frac 1 4 - \frac \delta 4 } \sqrt{ \sum_{t=1}^T \lVert \mathbf Q_{t-1} \rVert_2^2}
\end{align*}
and bound the two terms' expectations one by one. For the term with $\mathbbm 1 \left[L_{T-1} > MT^{\frac \delta 4}\right]$ factor, we have
\begin{align*}
    \mathbb E\left[\mathbbm 1 \left[L_{T-1} > MT^{\frac \delta 4}\right]T^{\frac 1 4 - \frac \delta 4 } \sqrt{ \sum_{t=1}^T \lVert \mathbf Q_{t-1} \rVert_2^2}\right] & \le T^{\frac 1 4 - \frac \delta 4 } \cdot \sqrt K T^{\frac 3 2} \mathbb E\left[\mathbbm 1 \left[L_{T-1} > MT^{\frac \delta 4}\right] L_{T-1} \right] \\
    & \le T^{\frac 1 4 - \frac \delta 4 } \cdot \sqrt K T^{\frac 3 2} \cdot 2KTM\cdot T^{-(\alpha - 1)\frac \delta 4} \\
    & = 2 qK^{\frac 3 2}MT^{\frac {11} 4 - \frac {\alpha \delta} 4} \\
    & \le o(T)
\end{align*}
where the last step is due to the assumption that $\delta \cdot \alpha > 7$. For the term with  $\mathbbm 1 \left[L_{T-1} > MT^{\frac \delta 4}\right]$ factor, we can bound it by
\begin{align*}
    & \quad \mathbb E\left[\mathbbm 1 \left[L_{T-1} \le MT^{\frac \delta 4}\right] T^{\frac 1 4 - \frac \delta 4 } \sqrt{ \sum_{t=1}^T \lVert \mathbf Q_{t-1} \rVert_2^2}\right] \\
    & \le T^{\frac 1 4 - \frac \delta 4 } \cdot \mathbb E\left[\mathbbm 1 \left[L_{T-1} \le MT^{\frac \delta 4}\right] 2L_{T-1}^{\frac 1 4} \left( \sum_{t=1}^T \lVert \mathbf Q_{t-1} \rVert_1\right)^{\frac 3 4}\right] \\
    & \le 2M^{\frac 1 4}T^{\frac 1 4 - \frac 3 {16} \delta} \mathbb E\left[\left(\sum_{t=1}^T \lVert \mathbf Q_{t-1} \rVert_1\right)^{\frac 3 4}\right].
\end{align*}
Therefore, we can conclude that
\begin{align}
    & \quad \sum_{t=1}^T \mathbb E\left[ \langle \mathbf Q_{t-1}, \mathbf S_t \odot \vec \theta_t \rangle - Q_{t-1,a_t}S_{t,a_t} \right] \nonumber \\
     & \le  o(T) + 2(1 + C_V)M^{\frac 5 4}T^{\frac 1 4 - \frac 3 {16}\delta} (4\ln T + \ln K) \mathbb E\left[\left(\sum_{t=1}^T \lVert \mathbf Q_{t-1} \rVert_1\right)^{\frac 3 4}\right] \nonumber \\
     & \quad + 16 M^{\frac 5 4}T^{\frac 1 4 - \frac 3 {16} \delta} \mathbb E\left[\left(\sum_{t=1}^T \lVert \mathbf Q_{t-1} \rVert_1\right)^{\frac 3 4}\right] \nonumber \\
     & \le o(T) + 18(1 + C_V)M^{\frac 5 4}T^{\frac 1 4 - \frac 3 {16}\delta} (4\ln T + \ln K) \mathbb E\left[\left(\sum_{t=1}^T \lVert \mathbf Q_{t-1} \rVert_1\right)^{\frac 3 4}\right] \nonumber\\
     & \le o(T) + 18(1 + C_V)M^{\frac 5 4}T^{\frac 1 4 - \frac 3 {16}\delta} (4\ln T + \ln K) \left(\mathbb E\left[\sum_{t=1}^T \lVert \mathbf Q_{t-1} \rVert_1\right]\right)^{\frac 3 4}. \label{eq-softmw-moment-regret-q-l1}
\end{align}

Plugging \Cref{eq-softmw-moment-regret-q-l1} into \Cref{eq-quad-lyapunov-total} in \Cref{lemma-quad-lyapunov}, we get
\begin{proposition}
\label{prop-softmw-ql1}
With Assumption~\ref{assumption-theta}, \ref{assumption-delta} 
and \ref{assumption-moment}, if $\delta \cdot \alpha > 7$, then running \Cref{softmw-moment} guarantees that 
\begin{align}
   & \quad \mathbb E\left[ 
\sum_{t=1}^{\mathcal T_T} \lVert \mathbf Q_{t-1}\rVert_1  \right] \le f(\mathcal T_T)  + h(\mathcal T_T) \cdot \left(\mathbb E\left[\sum_{t=1}^{\mathcal T_T} \lVert \mathbf Q_{t-1} \rVert_1\right]\right)^{\frac 3 4}\label{eq-softmw-moment-ql1-pre-444}
\end{align}
for any $T\ge \max\{\frac {4} {C_W}, C_W\}$, where $\mathcal T_T$ is some constant no more than $2T$,
\begin{equation*}
    f(T) = \frac {(K+1)M^2 + 2C_W (KM^2 + \epsilon KM) } \epsilon T + o(T),
\end{equation*}
and
\begin{equation*}
    h(T) = \frac{18(1 + C_V)M^{\frac 5 4}T^{\frac 1 4 - \frac 3 {16}\delta} (4\ln T + \ln K)} \epsilon =\tilde \O(T^{\frac 1 4 - \frac 3 {16} \delta}\epsilon^{-1}).
\end{equation*}
\end{proposition}

Applying \Cref{lemma-444} to \Cref{eq-softmw-moment-ql1-pre-444}, \Cref{eq-softmw-moment-ql1-pre-444} solves to
\begin{align*}
    \mathbb E\left[ 
\sum_{t=1}^{\mathcal T_T} \lVert \mathbf Q_{t-1}\rVert_1  \right] 
 & \le \left(h(\mathcal T_T)^{\frac 1 4} + f(\mathcal T_T)\right)^4 \\
 & \le \frac {(K+1)M^2 + 2C_W (KM^2 + \epsilon KM) } \epsilon \mathcal T_T + o(\mathcal T_T),
\end{align*}
thus
\begin{align*}
    \frac 1 T \mathbb E\left[ 
\sum_{t=1}^T \lVert \mathbf Q_{t-1}\rVert_1  \right] \le \frac 1 T \mathbb E\left[ 
\sum_{t=1}^{\mathcal T_T} \lVert \mathbf Q_{t-1}\rVert_1  \right] & \le \frac {(K+1)M^2 + 2C_W (KM^2 + \epsilon KM) } \epsilon \frac {\mathcal T_T} T + o(\mathcal T_T / T) \\
& \le  \frac {2(K+1)M^2 + 4C_W (KM^2 + \epsilon KM) } \epsilon + o(1)
\end{align*}
as desired.

\section{Detailed Analysis for Algorithm~\ref{ssmw-moment}}
\label{sec-apdx-ssmw-moment}

Similar to the analysis of \Cref{ssmw}, we first build generalized version of the regret bound lemma of each \texttt{EXP.3} epoch (\Cref{lemma-ssmw-epoch}):

\begin{lemma}
\label{lemma-ssmw-moment-epoch}
Suppose Assumption~\ref{assumption-theta}, \ref{assumption-delta3} and \ref{assumption-moment} hold, then, let $T_0$ be some time step on which we start a new \texttt{EXP3.S+} instance of length $m$ in \Cref{ssmw-moment}, we have
\begin{align}
    & \quad \mathbbm 1[T_0\text{ ends an EXP3 instance, and the new EXP3 instance is of length }m, m\ge 2] \nonumber \\
    &\quad \cdot \sum_{t=1}^m \mathbb E\left[\left. \langle \mathbf Q_{T_0 + t - 1} \odot \mathbf S_{T_0 + t}, \vec \theta_{T_0 + t} \rangle - Q_{T_0 + t - 1,a_{T_0 + t}}S_{T_0 + t,a_{T_0 + t}} \right\vert \mathcal F_{T_0} \right] \nonumber \\
& \le  21(1+C_V)M^3Km^{2 - \frac \delta 3} \cdot \left( 3\ln m + \ln K \right) + 4M^2. \label{eq-lemma-ssmw-moment-epoch}
\end{align}
\end{lemma}
\begin{proof}
     The high level idea is to apply \Cref{thm-exp3s}, but we need to verify that all $\gamma_\tau$'s are no more than $\frac 1 2$ first. 

     Our choice of $m$ guarantees that $m\ge \frac{\lVert 
\mathbf Q_{T_0} \rVert_\infty}{2M}$, hence $\lVert 
\mathbf Q_{T_0} \rVert_\infty \le 2Mm$, therefore
\begin{align*}
    \gamma_\tau & = \frac 1 4 M^{-2}m^{-1-\frac \delta 3}\lVert\mathbf Q_{T_0}\rVert_\infty \\
    & \le \frac 1 4 M^{-2}m^{-1-\frac \delta 3} \cdot 2Mm \\
    & = \frac 1 2 M^{-1}m^{-\frac \delta 3} \\
    & \le \frac 1 2.
\end{align*}

For any $1 \le \tau \le m$, define
\begin{align*}
    g_{T_0 + \tau, i} \triangleq 
    \begin{cases}
        Q_{T_0 + \tau - 1, i}S_{T_0 + \tau,i} & \text{if } Q_{T_0 + \tau - 1, i}S_{T_0 + \tau,i} \le m^{\frac \delta 3} M Q_{T_0,i} \\
        0 & \text{otherwise}
    \end{cases}.
\end{align*}

Below, we first bound
\begin{equation*}
    \mathbb E\left[\left.\sum_{\tau = 1}^m\left( \left\langle \mathbf g_{T_0+\tau}, \vec \theta_{T_0 + \tau}\right\rangle - g_{T_0 + \tau, a_{T_0 + \tau}}\right)\right\rvert\mathcal F_{T_0}\right],
\end{equation*}
i.e., the $\mathbf g_{T_0 + \tau}$-objective regret, then derive a bound for the original $\mathbf Q_{T_0 + \tau - 1}\odot \mathbf S_{T_0 + \tau}$-objective regret.

For each $\vec \theta_{T_0 + \tau}$, we can find a $\vec\theta'_{T_0 + \tau} \in \Delta^{[K], \beta_\tau}$ so that $\lVert \vec \theta_{T_0 + \tau} - \vec \theta'_{T_0 + \tau} \rVert_1 \le 2K \beta_\tau$ and $\lVert 
\vec \theta'_s - \vec \theta'_t \rVert_1 \le \lVert 
\vec \theta_{T_0 + s} - \vec \theta_{T_0 + t} \rVert_1$ for any $1\le s,t\le m$. For example, we can choose \begin{equation*}
    \vec \theta'_{T_0 + \tau} = (1 - \beta_\tau) \vec \theta_{T_0 + \tau} + \beta_\tau \mathbf 1.
\end{equation*} 
Then, we can write
\begin{align*}
    \langle \mathbf g_{T_0 + \tau}, \vec \theta_{T_0 + \tau} \rangle & = \langle \mathbf g_{T_0 + \tau}, \vec \theta'_{T_0 + \tau} \rangle + \langle \mathbf g_{T_0 + \tau}, \vec \theta_{T_0 + \tau} - \vec \theta'_{T_0 + \tau} \rangle \\
    & \le \langle \mathbf g_{T_0 + \tau}, \vec \theta'_{T_0 + \tau} \rangle + \lVert 
\mathbf g_{T_0 + \tau} \rVert_\infty \cdot \lVert \vec \theta_{T_0 + \tau} - \vec \theta'_{T_0 + \tau} \rVert_1 \\
& \le \langle \mathbf g_{T_0 + \tau}, \vec \theta'_{T_0 + \tau} \rangle + m^{\frac \delta 3} M \lVert 
\mathbf Q_{T_0} \rVert_\infty \cdot \lVert \vec \theta_{T_0 + \tau} - \vec \theta'_{T_0 + \tau} \rVert_1\\
& \le \langle \mathbf g_{T_0 + \tau}, \vec \theta'_{T_0 + \tau} \rangle + m^{\frac \delta 3} M \cdot 2Mm \cdot \lVert \vec \theta_{T_0 + \tau} - \vec \theta'_{T_0 + \tau} \rVert_1\\
& \le \langle \mathbf g_{T_0 + \tau}, \vec \theta'_{T_0 + \tau} \rangle + 2M^2 m^{1 + \frac \delta 3} \cdot 2K\beta_\tau \\
& = \langle \mathbf g_{T_0 + \tau}, \vec \theta'_{T_0 + \tau} \rangle + 4M^2m^{\frac \delta 3 -2} \\
& \le \langle \mathbf g_{T_0 + \tau}, \vec \theta'_{T_0 + \tau} \rangle + 4M^2m^{-1}.
\end{align*}

Then, one can see the quantity
\begin{align*}
    \sum_{\tau=1}^m \mathbb E\left[\left. \langle \mathbf g_{T_0 + \tau}, \vec \theta'_{T_0 + \tau} \rangle - g_{T_0 + \tau, a_{T_0 + \tau}}\right \rvert \mathcal F_{T_0} \right]
\end{align*}
satisfies the condition to apply \Cref{thm-exp3s}. \Cref{thm-exp3s} asserts that
\begin{align}
    & \quad \sum_{\tau=1}^m \mathbb E\left[\left. \langle \mathbf g_{T_0 + \tau}, \vec \theta'_{T_0 + \tau} \rangle - g_{T_0 + \tau, a_{T_0 + \tau}}\right \rvert \mathcal F_{T_0} \right]\nonumber \\
    & \le \left(1 + \sum_{\tau=1}^{m-1} \lVert \vec\theta'_{T_0 + \tau+1} - \vec\theta'_{T_0 + \tau} \rVert_1\right)\mathbb E\left[\left. \eta_m^{-1} \ln \frac 1 {\beta_m}\right \rvert \mathcal F_{T_0}\right] + e\mathbb E\left[\left.\sum_{\tau=1}^m \eta_\tau \lVert 
\mathbf g_{T_0 + \tau} \rVert_2^2\right\rvert \mathcal F_{T_0}\right] \nonumber \\
& \quad + \mathbb E\left[\left.\sum_{\tau=1}^m\gamma_\tau \left\langle \mathbf g_{T_0 + \tau}, \frac{\mathbf 1}{K} \right\rangle\right\vert \mathcal F_{T_0}\right]\nonumber \\
& \le \left(1 + \sum_{\tau=1}^{m-1} \lVert \vec\theta'_{T_0 + \tau+1} - \vec\theta'_{T_0 + \tau} \rVert_1\right) \cdot 4M^3Km^{1+\frac 2 3 \delta} \cdot \left( 3\ln m + \ln K \right) \nonumber \\
& \quad + e \frac 1 4 M^{-1}K^{-1}m^{-1-\frac 2 3 \delta}\mathbb E\left[\left.\sum_{\tau=1}^m \lVert 
\mathbf g_{T_0 + \tau}\rVert_2^2\right\rvert \mathcal F_{T_0}\right] + \frac 1 4 K^{-1}m^{-1-\frac 1 3 \delta}\lVert \mathbf Q_{T_0}\rVert_\infty\mathbb E\left[\left.\sum_{\tau=1}^m \left\lVert \mathbf g_{T_0 + \tau} \right\rVert_1\right\rvert \mathcal F_{T_0}\right]\nonumber \\
& \le 4\left(1 + C_V m^{1-\delta}\right) \cdot M^3Km^{1+\frac 2 3 \delta} \cdot \left( 3\ln m + \ln K \right) \nonumber \\
& \quad + \frac e 4 M^{-1}K^{-1}m^{-1-\frac 2 3 \delta}\mathbb E\left[\left.\sum_{\tau=1}^m \lVert 
\mathbf g_{T_0 + \tau}\rVert_2^2\right\rvert \mathcal F_{T_0}\right] + \frac 1 4K^{-1}m^{-1-\frac 1 3 \delta}\lVert \mathbf Q_{T_0}\rVert_\infty\mathbb E\left[\left.\sum_{\tau=1}^m \left\lVert \mathbf g_{T_0 + \tau} \right\rVert_1\right\rvert \mathcal F_{T_0}\right], \nonumber
\end{align}
where in the last step, we use the bound for $\sum_{\tau=1}^{m-1} \lVert \vec\theta'_{T_0 + \tau+1} - \vec\theta'_{T_0 + \tau} \rVert_1$ in Assumption~\ref{assumption-delta3}. To bound the two expectation terms involving $\mathbf g_{T_0 + \tau}$ factors, recall that $\mathbf g_{T_0 + \tau}$ still enjoys the moment bound before clipping:
\begin{align*}
    \mathbb E\left[\left. g_{T_0 + \tau, i}^2\right\rvert\mathcal F_{T_0}\right] & \le \mathbb E\left[\left. Q_{T_0 + \tau -1, i}^2 S_{T_0 + \tau ,i}^2\right\rvert\mathcal F_{T_0}\right] \\
    & = \mathbb E\left[\left. Q_{T_0 + \tau -1, i}^2 \right\rvert\mathcal F_{T_0}\right] \cdot \mathbb E\left[\left. S_{T_0 + \tau ,i}^2\right\rvert\mathcal F_{T_0}\right] \\
    & \stackrel{(a)}\le (Q_{T_0,i}^2 + M^2(\tau - 1)^2) \cdot M^2 \\
    & \le ((2Mm)^2 + M^2m^2) \cdot M^2 \\
    & = 5M^4m^2,
\end{align*}
where in step $(a)$, we bound $\mathbb E\left[\left. Q_{T_0 + \tau -1, i}^2 \right\rvert\mathcal F_{T_0}\right]$ by $Q_{T_0,i}^2 + M^2(\tau - 1)^2$ in the same way when we prove \Cref{lemma-ssmw-sum-l12norm-unbounded}.

Similarly,
\begin{align*}
    \mathbb E\left[\left. g_{T_0 + \tau, i}\right\rvert\mathcal F_{T_0}\right] & \le \mathbb E\left[\left. Q_{T_0 + \tau -1, i} S_{T_0 + \tau ,i}\right\rvert\mathcal F_{T_0}\right] \\
    & = \mathbb E\left[\left. Q_{T_0 + \tau -1, i} \right\rvert\mathcal F_{T_0}\right] \cdot \mathbb E\left[\left. S_{T_0 + \tau ,i}\right\rvert\mathcal F_{T_0}\right] \\
    & \le (Q_{T_0,i} + M(\tau - 1)) \cdot M \\
    & \le (2Mm + Mm) \cdot M \\
    & = 3M^2 m.
\end{align*}

Thus the two expectation factors can be bounded by
\begin{align*}
\mathbb E\left[\left.\sum_{\tau=1}^m \lVert 
\mathbf g_{T_0 + \tau}\rVert_2^2\right\rvert \mathcal F_{T_0}\right] & \le 5KM^4m^3, \\
\lVert \mathbf Q_{T_0}\rVert_\infty\mathbb E\left[\left.\sum_{\tau=1}^m \left\lVert \mathbf g_{T_0 + \tau} \right\rVert_1\right\rvert \mathcal F_{T_0}\right] & \le 2Mm \cdot 3KM^2m \\
& = 6KM^3m^2.
\end{align*}

Therefore, we can conclude that
\begin{align*}
    & \quad \sum_{\tau=1}^m \mathbb E\left[\left. \langle \mathbf g_{T_0 + \tau}, \vec \theta'_{T_0 + \tau} \rangle - g_{T_0 + \tau, a_{T_0 + \tau}}\right \rvert \mathcal F_{T_0} \right] \\
    & \le 4\left(1 + C_V m^{1-\delta}\right) \cdot M^3Km^{1+\frac 2 3 \delta} \cdot \left( 3\ln m + \ln K \right) \nonumber \\
    & \quad + \frac {5e} 4 M^3m^{2-\frac 2 3 \delta} + \frac 3 2 M^3m^{2-\frac 1 3 \delta} \\
    & \le 4\left(1 + C_V m^{1-\delta}\right) \cdot M^3Km^{1+\frac 2 3 \delta} \cdot \left( 3\ln m + \ln K \right) \nonumber \\
    & \quad + \frac 7 2 M^3m^{2-\frac 2 3 \delta} + \frac 3 2 M^3m^{2-\frac 1 3 \delta} \\
    & \le 9(1+C_V)M^3Km^{2 - \frac \delta 3} \cdot \left( 3\ln m + \ln K \right).
\end{align*}

Finally, note that for each $\tau$, the single-step regret $\langle \mathbf Q_{T_0 + \tau -1}, \mathbf S_{T_0 + \tau} \odot \vec \theta_{T_0 + \tau} \rangle - Q_{T_0 + \tau-1,a_{T_0 + \tau}}S_{T_0 + \tau,a_{T_0 + \tau}}$ can be upper-bounded by
    \begin{align*}
        & \quad \langle \mathbf Q_{T_0 + \tau -1}, \mathbf S_{T_0 + \tau} \odot \vec \theta_{T_0 + \tau} \rangle - Q_{T_0 + \tau-1,a_{T_0 + \tau}}S_{T_0 + \tau,a_{T_0 + \tau}} \\
        & \le \underbrace{ \langle \mathbf g_{T_0 + \tau}, \vec \theta_{T_0 + \tau} \rangle - g_{T_0 + \tau, a_{T_0 + \tau}}}_{\mathbf g_{T_0 + \tau}\text{-objective single step regret}} + \mathbbm 1\left[Q_{T_0 + \tau-1,a_{T_0 + \tau}}S_{T_0 + \tau,a_{T_0 + \tau}} > Mm^{\frac \delta 3}\lVert\mathbf Q_{T_0}\rVert_\infty\right]\langle \mathbf Q_{T_0 + \tau -1}, \mathbf S_{T_0 + \tau} \odot \vec \theta_{T_0 + \tau} \rangle.
    \end{align*}
    And each difference term can be controlled by
    \begin{align*}
        & \quad \mathbb E\left[\left.\mathbbm 1\left[Q_{T_0 + \tau-1,a_{T_0 + \tau}}S_{T_0 + \tau,a_{T_0 + \tau}} > Mm^{\frac \delta 3}\lVert\mathbf Q_{T_0}\rVert_\infty\right]\langle \mathbf Q_{T_0 + \tau -1}, \mathbf S_{T_0 + \tau} \odot \vec \theta_{T_0 + \tau} \rangle\right\rvert \mathcal F_{T_0}\right] \\
        & \le \mathbb E\left[\left.\mathbbm 1\left[Q_{T_0 + \tau-1,a_{T_0 + \tau}}S_{T_0 + \tau,a_{T_0 + \tau}} > Mm^{\frac \delta 3}\lVert\mathbf Q_{T_0}\rVert_\infty\right]\sum_{i=1}^K Q_{T_0 + \tau - 1, i}S_{T_0 + \tau, i}\right\rvert \mathcal F_{T_0}\right]\\
        & = \mathbb E\left[\left.\mathbbm 1\left[Q_{T_0 + \tau-1,a_{T_0 + \tau}}S_{T_0 + \tau,a_{T_0 + \tau}} > Mm^{\frac \delta 3}\lVert\mathbf Q_{T_0}\rVert_\infty\right]\sum_{i\in[K]:i\ne a_t} Q_{T_0 + \tau - 1, i}S_{T_0 + \tau, i}\right\rvert \mathcal F_{T_0}\right] \\
        & \quad + \mathbb E\left[\left.\mathbbm 1\left[Q_{T_0 + \tau-1,a_{T_0 + \tau}}S_{T_0 + \tau,a_{T_0 + \tau}} > Mm^{\frac \delta 3}\lVert\mathbf Q_{T_0}\rVert_\infty\right] Q_{T_0 + \tau - 1, a_t}S_{T_0 + \tau, a_t}\right\rvert \mathcal F_{T_0}\right].
    \end{align*}
    For the first term $\mathbb E\left[\left.\mathbbm 1\left[Q_{T_0 + \tau-1,a_{T_0 + \tau}}S_{T_0 + \tau,a_{T_0 + \tau}Q_{T_0,a_t}} > Mm^{\frac \delta 3}\lVert\mathbf Q_{T_0}\rVert_\infty\right]\sum_{i\in[K]:i\ne a_t} Q_{T_0 + \tau - 1, i}S_{T_0 + \tau, i}\right\rvert \mathcal F_{T_0}\right]$, we can write
    \begin{align*}
        & \quad \mathbb E\left[\left.\mathbbm 1\left[Q_{T_0 + \tau-1,a_{T_0 + \tau}}S_{T_0 + \tau,a_{T_0 + \tau}} > Mm^{\frac \delta 3}\lVert\mathbf Q_{T_0}\rVert_\infty\right]\sum_{i\in[K]:i\ne a_t} Q_{T_0 + \tau - 1, i}S_{T_0 + \tau, i}\right\rvert \mathcal F_{T_0}\right] \\
        & = \mathbb E\left[\left.\mathbb P\left[\left.Q_{T_0 + \tau-1,a_{T_0 + \tau}}S_{T_0 + \tau,a_{T_0 + \tau}} > Mm^{\frac \delta 3}\lVert\mathbf Q_{T_0}\rVert_\infty\right\rvert \mathcal F_{T_0}, a_t\right]\cdot \mathbb E\left[\left.\sum_{i\in[K]:i\ne a_t} Q_{T_0 + \tau - 1, i}S_{T_0 + \tau, i}\right\rvert \mathcal F_{T_0}, a_t \right] \mathcal F_{T_0}\right\rvert\right].
    \end{align*}
    Recall that for each $1 \le \tau \le m$ and $i\in [K]$, we can bound the second monent of $Q_{T_0 + \tau - 1, i}S_{T_0 + \tau, i}$ by
    \begin{align*}
        \mathbb E\left[\left.Q_{T_0 + \tau - 1, i}^2S_{T_0 + \tau, i}^2\right\rvert \mathcal F_{T_0}\right] & = \mathbb E\left[\left.Q_{T_0 + \tau - 1, i}^2\right\rvert \mathcal F_{T_0}\right] \cdot \mathbb E\left[\left.S_{T_0 + \tau, i}^2\right\rvert \mathcal F_{T_0}\right] \\
        & \stackrel{(a)}\le (Q_{T_0,i}^2 + M^2(\tau - 1)^2) \cdot M^2 \\
        & \le ((2Mm)^2 + M^2m^2) \cdot M^2 \\
        & = 5M^4m^2,
    \end{align*}
    therefore, by Chebyshev's inequality, we have
    \begin{align*}
        \mathbb P\left[\left.Q_{T_0 + \tau-1,a_{T_0 + \tau}}S_{T_0 + \tau,a_{T_0 + \tau}} > Mm^{\frac \delta 3}\lVert\mathbf Q_{T_0}\rVert_\infty\right\rvert \mathcal F_{T_0}, a_t\right] & \le 5M^4m^2\cdot (Mm^{\frac \delta 3}\lVert\mathbf Q_{T_0}\rVert_\infty)^{-2} \\
        & \stackrel{(a)}\le  5M^4m^2\cdot (Mm^{\frac \delta 3}\cdot Mm)^{-2} \\
        & = 5m^{-\frac 2 3 \delta},
    \end{align*}
    where step $(a)$ is due to the assumption that $m \ge 2$, thus $\lVert\mathbf Q_{T_0}\rVert_\infty \ge 2M(m-1) \ge Mm$. For succeeding factor $\mathbb E\left[\left.\sum_{i\in[K]:i\ne a_t} Q_{T_0 + \tau - 1, i}S_{T_0 + \tau, i}\right\rvert \mathcal F_{T_0}, a_t \right]$, we have
    \begin{align*}
        \mathbb E\left[\left.\sum_{i\in[K]:i\ne a_t} Q_{T_0 + \tau - 1, i}S_{T_0 + \tau, i}\right\rvert \mathcal F_{T_0}, a_t \right] \le (K-1) \cdot \sqrt 5 M^2 m.
    \end{align*}
    For the other term $\mathbb E\left[\left.\mathbbm 1\left[Q_{T_0 + \tau-1,a_{T_0 + \tau}}S_{T_0 + \tau,a_{T_0 + \tau}} > Mm^{\frac \delta 3}\lVert\mathbf Q_{T_0}\rVert_\infty\right] Q_{T_0 + \tau - 1, a_t}S_{T_0 + \tau, a_t}\right\rvert \mathcal F_{T_0}\right]$, we have
    \begin{align*}
        & \quad \mathbb E\left[\left.\mathbbm 1\left[Q_{T_0 + \tau-1,a_{T_0 + \tau}}S_{T_0 + \tau,a_{T_0 + \tau}} > Mm^{\frac \delta 3}\lVert\mathbf Q_{T_0}\rVert_\infty\right] Q_{T_0 + \tau - 1, a_t}S_{T_0 + \tau, a_t}\right\rvert \mathcal F_{T_0}\right] \\
        & \le \mathbb E\left[\left. Q_{T_0 + \tau - 1, a_t}^2S_{T_0 + \tau, a_t}^2\right\rvert \mathcal F_{T_0}\right] \cdot \left(Mm^{\frac \delta 3}\lVert\mathbf Q_{T_0}\rVert_\infty\right)^{-1}  \\
        & \le 5M^4m^2\cdot (Mm^{\frac \delta 3}\cdot Mm)^{-1} \\
        & = 5M^2m^{1-\frac \delta 3}.
    \end{align*}
Therefore, combining the different parts of bounds, we get
\begin{align*}
    & \quad \sum_{\tau = 1}^m\mathbb E\left[\left.\mathbbm 1\left[Q_{T_0 + \tau-1,a_{T_0 + \tau}}S_{T_0 + \tau,a_{T_0 + \tau}} > Mm^{\frac \delta 3}\lVert\mathbf Q_{T_0}\rVert_\infty\right]\langle \mathbf Q_{T_0 + \tau -1}, \mathbf S_{T_0 + \tau} \odot \vec \theta_{T_0 + \tau} \rangle\right\rvert \mathcal F_{T_0}\right] \\
    & \le 5\sqrt 5 (K-1) M^2 m^{2 - \frac 2 3\delta} + 5M^2 m^{2- \frac \delta 3} \\
    & \le 12KM^2m^{2-\frac \delta 3}.
\end{align*}
Putting this regret difference bound with the $\mathbf g_{T_0 + \tau}$-objective regret bound, we get
\begin{align*}
    & \quad \sum_{t=1}^m \mathbb E\left[\left. \langle \mathbf Q_{T_0 + t - 1} \odot \mathbf S_{T_0 + t}, \vec \theta_{T_0 + t} \rangle - Q_{T_0 + t - 1,a_{T_0 + t}}S_{T_0 + t,a_{T_0 + t}} \right\vert \mathcal F_{T_0} \right] \\
    & \le 9(1+C_V)M^3Km^{2 - \frac \delta 3} \cdot \left( 3\ln m + \ln K \right) + 4M^2 + 12KM^2m^{2-\frac \delta 3} \\
    & \le 21(1+C_V)M^3Km^{2 - \frac \delta 3} \cdot \left( 3\ln m + \ln K \right) + 4M^2.
\end{align*}
\end{proof}

Combining \Cref{lemma-ssmw-moment-epoch} and  \Cref{lemma-ssmw-sum-l12norm-unbounded}, we can get a regret upper-bound in $\sum_{\tau=1}^m \lVert\mathbf Q_{T_0 + \tau - 1}\rVert_1$:

\begin{lemma}
\label{lemma-ssmw-moment-epoch-ql1}
    Suppose Assumption~\ref{assumption-theta}, \ref{assumption-delta3} and \ref{assumption-moment} hold, then, let $T_0$ be some time step on which we start a new \texttt{EXP3.S+} instance of length $m$ in \Cref{ssmw-moment}, we have
\begin{align}
    & \quad \mathbbm 1[T_0\text{ ends an EXP3 instance, and the new EXP3 instance is of length }m, m\ge 2] \nonumber \\
    &\quad \cdot \sum_{t=1}^m \mathbb E\left[\left. \langle \mathbf Q_{T_0 + t - 1} \odot \mathbf S_{T_0 + t}, \vec \theta_{T_0 + t} \rangle - Q_{T_0 + t - 1,a_{T_0 + t}}S_{T_0 + t,a_{T_0 + t}} \right\vert \mathcal F_{T_0} \right] \nonumber \\
& \le  42(1+C_V)M^2Km^{- \frac \delta 3} \cdot \left( 3\ln m + \ln K \right) \cdot \mathbb E\left[\left.\sum_{\tau=1}^m \lVert 
\mathbf Q_{T_0 + \tau-1}\rVert_2^2\right\rvert \mathcal F_{T_0}\right] + 4M^2. \label{eq-lemma-ssmw-moment-epoch-l1}
\end{align}
\end{lemma}

With \Cref{lemma-ssmw-moment-epoch-ql1}, the remaining steps are the same as the analysis of \Cref{ssmw}. 

Denote by $\tau_i$ ($i \ge 0$) the time on which the $i$-th \texttt{EXP3.S+} instance finishes. Then, $\tau_0 = 0$, $\{\tau_i\}$ is a sequence of non-decreasing $\{\mathcal F_t\}$-adapted stopping-times. Further more, each $\tau_{i+1}$ is $\mathcal F_{\tau_i}$-measurable. Fix any $T\ge 1$, define
\begin{equation*}
    \tau'_i \triangleq \begin{cases}
        0 & \text{if }i = 0 \\
        \tau_i & \text{if } i > 0 \text{ and } \tau'_{i-1} < T \\
        \tau'_{i-1} & \text{otherwise}
    \end{cases},
\end{equation*}then $\{\tau'_i\}$ is a sequence of non-decreasing $\{\mathcal F_t\}$-adapted stopping-times, each $\tau'_{i+1}$ is $\mathcal F_{\tau'_i}$-measurable, $\tau'_{i+1} = \tau'_i$ if any only if $\tau'_i \ge T$. Thus, we can restate \Cref{lemma-ssmw-moment-epoch-ql1} as the following:
\begin{lemma}
\label{lemma-ssmw-moment-epoch-l1-stopping-time}
    Suppose Assumption~\ref{assumption-theta}, \ref{assumption-delta3} and \ref{assumption-moment} hold, then we have
    \begin{align*}
        & \quad\mathbbm 1\left[ \lVert \mathbf Q_{\tau'_i}\rVert_\infty \ge  4M\right] \sum_{t=1}^{\tau'_{i+1} - \tau'_i} \mathbb E\left[\left. \langle \mathbf Q_{\tau'_i + t - 1} \odot \mathbf S_{\tau'_i + t}, \vec \theta_{\tau'_i + t} \rangle - Q_{\tau'_i + t - 1,a_{\tau'_i + t}}S_{\tau'_i + t,a_{\tau'_i + t}}\right\vert \mathcal F_{\tau'_i}\right] \\
        & \le h(\tau'_{i+1} - \tau'_i) + f(\tau'_{i+1} - \tau'_i) \cdot \sum_{t=1}^{\tau'_{i+1} - \tau'_i} \mathbb E\left[\left. 
\lVert \mathbf Q_{\tau'_i + t - 1} \rVert_1 \right\vert \mathcal F_{\tau'_i}\right]
    \end{align*}
for any $i\ge 0$, where
\begin{equation*}
    f(m) = 42(1+C_V)M^2Km^{- \frac \delta 3} \cdot \left( 3\ln m + \ln K \right),
\end{equation*}
\begin{equation*}
    h(m) = \mathbbm 1[m > 0]\cdot 4M^2.
\end{equation*}
\end{lemma}

Fix some $T \ge 1$, let $\mathcal T_0 \triangleq \sup \{\tau_i : i\ge 0, \tau_i < T\}$, $\mathcal T_1 \triangleq \inf \{\tau_i : i\ge 0, \tau_i \ge T\}$, then $\mathcal T_0$ and $\mathcal T_1$ are both $\{\mathcal F_t\}$-adapted stopping-time, $\mathcal T_1$ is $\mathcal F_{\mathcal T_0}$-measurable. Note that $\mathcal T_0 < T \le \mathcal T_1$. Furthermore, since $\mathcal T_1 - \mathcal T_0 \le \frac {\lVert \mathbf Q_{\mathcal T_0}\rVert_\infty}{2M} + 1 \le \frac{\mathcal T_0 \cdot M}{2M} + 1 = \frac{\mathcal T_0} 2 + 1  \le \frac{\mathcal T_0} 2 + T$, hence we have $\mathcal T_1 \le \frac 5 2 T$. 

Below, we will combine \Cref{lemma-quad-lyapunov} and \Cref{lemma-ssmw-moment-epoch-l1-stopping-time} to bound $\mathbb E[\sum_{t=1}^{\mathcal T_1} \lVert\mathbf Q_{t - 1}\rVert_1]$ in $\O(\mathbb E[\mathcal T_1]) = \O(T)$ so that we can conclude that $\mathbb E[\sum_{t=1}^T \lVert\mathbf Q_{t - 1}\rVert_1]$ is also $\O(T)$.

Recall $\epsilon > 0$ is the lower-bound of the ``average advantage of departure against arrival'' of the reference policy $\{\vec \theta_t\}$ in Assumption~\ref{assumption-theta}, define
\begin{equation*}
    m_0 \triangleq \inf \left\{ m : m\ge 2,  f(m') \le \frac \epsilon 2 \space \forall m'\ge m\right\},
\end{equation*}
then $m_0$ is a constant that only depends on $\delta$ and $\epsilon$, in fact,
\begin{equation*}
    m_0 \le \left((1+C_V)M^2K\ln K \epsilon^{-1}\right)^{\O(1/\delta)}.
\end{equation*}

By discussing whether each epoch length $\tau'_{i+1} - \tau'_i$ is greater than $m_0$ or not, we conclude from \Cref{lemma-ssmw-moment-epoch-l1-stopping-time} that

    \begin{align}
        & \quad\sum_{t=1}^{\tau'_{i+1} - \tau'_i} \mathbb E\left[\left. \langle \mathbf Q_{\tau'_i + t - 1} \odot \mathbf S_{\tau'_i + t}, \vec \theta_{\tau'_i + t} \rangle - Q_{\tau'_i + t - 1,a_{\tau'_i + t}}S_{\tau'_i + t,a_{\tau'_i + t}}\right\vert \mathcal F_{\tau'_i}\right] \nonumber \\
        & \le h(\tau'_{i+1} - \tau'_i) + \frac \epsilon 2 \sum_{t=1}^{\tau'_{i+1} - \tau'_i} \mathbb E\left[\left. \lVert \mathbf Q_{\tau'_i + t - 1} \rVert_1 \right\vert \mathcal F_{\tau'_i}\right] + \mathbbm 1[\tau'_{i+1} - \tau'_i \le m_0]\sum_{t=1}^{\tau'_{i+1} - \tau'_i} \mathbb E\left[\left. \langle \mathbf Q_{\tau'_i + t - 1} \odot \mathbf S_{\tau'_i + t}, \vec \theta_{\tau'_i + t} \rangle \right\vert \mathcal F_{\tau'_i}\right] \nonumber \\
        & \le h(\tau'_{i+1} - \tau'_i) + \frac \epsilon 2 \sum_{t=1}^{\tau'_{i+1} - \tau'_i} \mathbb E\left[\left. \lVert \mathbf Q_{\tau'_i + t - 1} \rVert_1 \right\vert \mathcal F_{\tau'_i}\right] + \mathbbm 1[\tau'_{i+1} - \tau'_i \le m_0] M\sum_{t=1}^{\tau'_{i+1} - \tau'_i} \mathbb E\left[\left. \lVert\mathbf Q_{\tau'_i + t - 1}\rVert_1\right\vert \mathcal F_{\tau'_i}\right] \nonumber \\
        & \stackrel{(a)}\le h(\tau'_{i+1} - \tau'_i) + \frac \epsilon 2 \sum_{t=1}^{\tau'_{i+1} - \tau'_i} \mathbb E\left[\left. \lVert \mathbf Q_{\tau'_i + t - 1} \rVert_1 \right\vert \mathcal F_{\tau'_i}\right] + \mathbbm 1[\tau'_{i+1} - \tau'_i \le m_0] M\cdot 3KM(\tau'_{i+1} - \tau'_i)^2 \nonumber \\
        & \le h(\tau'_{i+1} - \tau'_i) + \frac \epsilon 2 \sum_{t=1}^{\tau'_{i+1} - \tau'_i} \mathbb E\left[\left. \lVert \mathbf Q_{\tau'_i + t - 1} \rVert_1 \right\vert \mathcal F_{\tau'_i}\right] +  3KM^2m_0(\tau'_{i+1} - \tau'_i) \label{eq-ssmw-moment-epoch-regret-stopping-time-m0}
    \end{align}
for all $i\ge 0$. Here in step $(a)$ we apply \Cref{lemma-ssmw-sum-l12norm-unbounded}.

Summing \Cref{eq-ssmw-moment-epoch-regret-stopping-time-m0} over all $i\ge 0$ and then taking total expectations, we get
\begin{align*}
        \mathbb E\left[\sum_{t=1}^{\mathcal T_1}  \langle \mathbf Q_{t - 1} \odot \mathbf S_t, \vec \theta_t \rangle - Q_{t - 1,a_t}S_{t,a_t} \right]  \le \frac \epsilon 2 \mathbb E\left[ \sum_{t=1}^{\mathcal T_1} \lVert \mathbf Q_{t - 1} \rVert_1 \right] + 3KM^2m_0 \mathbb E\left[\mathcal T_1\right] + \mathbb E\left[\sum_{i=0}^\infty h(\tau'_{i+1} - \tau'_i) \right].
\end{align*}
In any sample path, we have $\sum_{i=0}^\infty h(\tau'_{i+1} - \tau'_i) \le 4M^2\mathcal T_1$, therefore 
\begin{align}
        \mathbb E\left[\sum_{t=1}^{\mathcal T_1}  \langle \mathbf Q_{t - 1} \odot \mathbf S_t, \vec \theta_t \rangle - Q_{t - 1,a_t}S_{t,a_t} \right] & \le \frac \epsilon 2 \mathbb E\left[ \sum_{t=1}^{\mathcal T_1} \lVert \mathbf Q_{t - 1} \rVert_1 \right] + (3KM^2m_0 + 4M^2) \mathbb E\left[\mathcal T_1\right]. \label{eq-ssmw-moment-analysis-1}
\end{align}
According to \Cref{lemma-quad-lyapunov-theta}, we can also find a constant $\mathcal T_2$ depending on $\mathcal T_1$, such that  $\mathcal T_2 \le \mathcal T_1 + \sqrt{\frac {\mathcal T_1} {C_W}} + 1$ and
\begin{equation}
     -\mathbb E\left[\sum_{t=1}^{\mathcal T_2} \langle \mathbf Q_{t-1}, \vec \sigma_t \odot \vec \theta_t - \vec \lambda_t\rangle\right] \le  -\epsilon \mathbb E\left[ \sum_{t=1}^{\mathcal T_1} \lVert \mathbf Q_{t - 1} \rVert_1 \right] + (KM^2 + \epsilon KM)C_W \mathbb E[\mathcal T_2]. \label{eq-ssmw-moment-analysis-2}
\end{equation}
If $T \ge \frac 4 {C_W} + C_W$, we have $\mathcal T_1 \ge \max\{C_W, \frac 4 {C_W}\}$ hence $\mathcal T_2 \le 2\mathcal T_1 \le 5T$. Also, \Cref{lemma-quad-lyapunov} guarantees that
\begin{equation}
    \mathbb E\left[\sum_{t=1}^{\mathcal T_2} Q_{t-1,a_t}S_{t,a_t} -\langle \mathbf Q_{t-1}, \vec \lambda_t\rangle\right] \le \frac{(K+1)M^2\mathbb E[\mathcal T_2]} 2. \label{eq-ssmw-moment-analysis-3}
\end{equation}

Combining \Cref{eq-ssmw-moment-analysis-1,eq-ssmw-moment-analysis-2,eq-ssmw-moment-analysis-3} together, we get
\begin{align*}
    \frac \epsilon 2 \mathbb E\left[ \sum_{t=1}^{\mathcal T_1} \lVert \mathbf Q_{t - 1} \rVert_1 \right] & \le \left[\frac{(K+1)M^2} 2 + (KM^2 + \epsilon KM)C_W + 3KM^2m_0 + 4M^2\right] \mathbb E[\mathcal T_2] \\
    & \le \left[\frac{(K+1)M^2} 2 + (KM^2 + \epsilon KM)C_W + 3KM^2m_0 + 4M^2\right] \cdot 5T.
\end{align*}
Thus, when $T \ge \frac 4 {C_W} + C_W$, we have
\begin{equation*}
     \frac 1 T \mathbb E\left[ \sum_{t=1}^T \lVert \mathbf Q_{t - 1} \rVert_1 \right] \le \frac 1 T \mathbb E\left[ \sum_{t=1}^{\mathcal T_1} \lVert \mathbf Q_{t - 1} \rVert_1 \right] \le \left[3KM^2m_0 + \frac{(K+1)M^2} 2 + (KM^2 + \epsilon KM)C_W + 4M^2\right] \cdot \frac {10} \epsilon.
\end{equation*}

\end{document}